\documentclass[12pt]{article}
\usepackage{amsmath,amssymb,amsfonts,amsthm}
\newenvironment{myabstract}{\par\noindent
{\bf Abstract . } \small }
{\par\vskip8pt minus3pt\rm}
\newcounter{item}[section]
\newcounter{kirshr}
\newcounter{kirsha}
\newcounter{kirshb}
\newenvironment{enumroman}{\setcounter{kirshr}{1}
\begin{list}{(\roman{kirshr})}{\usecounter{kirshr}} }{\end{list}}
\newenvironment{enumarab}{\setcounter{kirshb}{1}
\begin{list}{(\arabic{kirshb})}{\usecounter{kirshb}} }{\end{list}}
\newtheorem{theorem}{Theorem}[section]

\newtheorem{lemma}[theorem]{Lemma}
\newtheorem{corollary}[theorem]{Corollary}
\newenvironment{demo}[1]{\noindent{\bf #1.}\upshape\mdseries}
{\nopagebreak{\hfill\rule{2mm}{2mm}\nopagebreak}\par\normalfont}
\theoremstyle{definition}
\newtheorem{remark}[theorem]{Remark}

\newtheorem{example}[theorem]{Example}
\newtheorem{definition}[theorem]{Definition}
\def\R{\mathbb{R}}

\def\C{{\mathfrak{C}}}
\def\Fm{{\mathfrak{Fm}}}

\def\At{{\bf At}}
\def\Nr{{\mathfrak{Nr}}}

\def\Sg{{\mathfrak{Sg}}}
\def\Fm{{\mathfrak{Fm}}}
\def\A{{\mathfrak{A}}}
\def\B{{\mathfrak{B}}}
\def\C{{\mathfrak{C}}}
\def\D{{\mathfrak{D}}}
\def\M{{\mathfrak{M}}}
\def\N{{\mathfrak{N}}}

\def\CA{{\bf CA}}

\def\QEA{{\bf QEA}}

\def\Df{{\bf Df}}

\def\PA{{\bf PA}}
\def\PEA{{\bf PEA}}

\def\K{{\bf K}}
\def\K{{\bf K}}

\def\RCA{{\bf RCA}}

\def\Rd{{\ Rd}}
\def\(R)RA{{\bf (R)RA}}
\def\RA{{\bf RA}}

\def\R{\mathbb{R}}

\def\Sc{{\bf Sc}}

\def\c #1{{\cal #1}}
 \def\CA{{\sf CA}}
\def\B{{\sf B}}
\def\G{{\sf G}}
\def\w{{\sf w}}
\def\y{{\sf y}}
\def\g{{\sf g}}

\def\r{{\sf r}}
\def\K{{\sf K}}

 \def\Cm{{\mathfrak{Cm}}}
\def\Nr{{\mathfrak{Nr}}}
\def\SNr{{\bf S}{\mathfrak{Nr}}}

\def\restr #1{{\restriction_{#1}}}

\def\R{\sf R}
\def\Ra{{\mathfrak{Ra}}}
\def\Ca{{\mathfrak{Ca}}}
\def\set#1{\{#1\} }
\def\Ra{{\mathfrak{Ra}}}
\def\Nr{{\mathfrak{Nr}}}
\def\Tm{{\mathfrak{Tm}}}
\def\A{{\mathfrak{A}}}
\def\B{{\mathfrak{B}}}
\def\C{{\mathfrak{C}}}
\def\D{{\mathfrak{D}}}
\def\E{{\mathfrak{E}}}
\def\A{{\mathfrak{A}}}
\def\B{{\mathfrak{B}}}
\def\C{{\mathfrak{C}}}
\def\D{{\mathfrak{D}}}
\def\E{{\mathfrak{E}}}

\def\GG{{\mathfrak{GG}}}

\def\L{{\mathfrak{L}}}
\def\Rd{{\mathfrak{Rd}}}

\def\At{{\mathfrak{At}}}
\def\L{{\mathfrak{L}}}

\def\CA{{\bf CA}}
\def\RA{{\bf RA}}

\def\RCA{{\bf RCA}}

\def\G{{\bf G}}

\def\CRA{{\sf CRA}}

\def\F{{\mathfrak{F}}}
\def\At{{\sf{At}}}
\def\N{\mathbb{N}}
\def\R{\mathfrak{R}}

\def\CRA{{\sf CRA}}

\def\Cs{{\sf Cs}}
\def\RPEA{{\sf RPEA}}

\def\c #1{{\cal #1}}

\def\pa{$\forall$}
\def\pe{$\exists$}

\def\ef{Ehren\-feucht--Fra\"\i ss\'e}

\def\nodes{{\sf nodes}}

\def\restr #1{{\restriction_{#1}}}

\def\Ra{{\mathfrak{Ra}}}
\def\Nr{{\mathfrak{Nr}}}
\def\Z{{\cal Z}}
\def\CA{{\bf CA}}
\def\RCA{{\bf RCA}}
\def\c#1{{\mathcal #1}}

\def\set#1{ \{#1\}}

\def\Ca{{\mathfrak Ca}}

\def\pe{$\exists$}
\def\pa{$\forall$}
\def\Cm{{\mathfrak Cm}}
\def\Sg{{\mathfrak Sg}}

\def\N{{\cal N}}

\def\At{{\sf At}}

\def\rng{{\sf rng}}
\def\dom{{\sf dom}}
\def\Cm{{\sf Cm}}
\def\Mat{{\sf Mat}}
\def\w{{\sf w}}
\def\g{{\sf g}}
\def\y{{\sf y}}
\def\r{{\sf r}}

\def\ws{winning strategy}
\def\ef{Ehren\-feucht--Fra\"\i ss\'e}

 \def\CA{{\sf CA}}
\def\Cs{{\sf Cs}}
\def\RCA{{\sf RCA}}

\def\RA{{\sf RA}}
\def\PA{{\sf PA}}

\def\PEA{\sf PEA}
\def\QEA{{\sf QEA}}

\def\y{{\sf y}}
\def\g{{\sf g}}
\def\r{{\sf r}}
\def\w{{\sf w}}
\def\Z{{\mathbb{Z}}}
\def\N{{\mathbb{N}}}

\def\c{{\sf c}}

\def\d{ Dedekind-MacNeille}

\def\Sc{{\sf Sc}}
\def\Df{{\sf Df}}

\title{Strongly representable atom structures}
\author{Tarek Sayed Ahmed and Mohammed Assem\\
Department of Mathematics, Faculty of Science,\\
Cairo University, Giza, Egypt.
  }
\begin{document}
\maketitle
\begin{myabstract} We show that the variety of representable algebras for many cylindric like algebras of finite dimension $>2$ is
not atom canonical.
Our result works for any  class $\K$ between diagonal free
algebras and polyadic equality algebras
of finite dimension $>2$. We also show, modifying a construction of Hirsch and Hodkinson,
 that for such classes of algebras (in the same dimensions)
the class of strongly representable atom structures is not elementary.
For finite dimensions $n>2$
we also show that several proper subvarieties of the representable cylindric and polyadic equality
algebras of dimension $n$, whose members have a neat embedding property,
are not atom canonical. From this we infer that the
omitting types theorem fails for $L_n$ (first order logic restricted to the first $n$ variables, $n$ finite $>2$)
in quite a strong sense, namely, if we consider clique guarded semantics, and first order definable expansions
of $L_n$. Finally,
we show that it is undecidable to tell whether a finite $\CA_3$ is
in ${\sf S}\Nr_3\CA_{3+k}$ for $k\geq 3$, by reducing the problem to the analogous one for relation algebras.
\end{myabstract}

\section{Introduction}

We follow the notation of \cite{1} which is in conformity with that of \cite{tarski}.
Assume that we have a class of Boolean algebras with operators for which we have a semantical notion of representability
(like Boolean set algebras or cylindric set algebras).
A weakly representable atom structure is an atom structure such that at least one atomic algebra based on it is representable.
It is strongly representable  if all atomic algebras having this atom structure are representable.
The former is equivalent to that the term algebra, that is, the algebra generated by the atoms,
in the complex algebra is representable, while the latter is equivalent  to that the complex algebra is representable.

Could an atom structure be possibly weakly representable but {\it not} strongly representable?
Ian Hodkinson \cite{Hodkinson}, showed that this can indeed happen for both cylindric  algebras of finite dimension $\geq 3$, and relation algebras,
in the context of showing that the class of representable algebras, in both cases,  is not closed under \d\ completions.
In fact, he showed that this can be witnessed on an atomic algebras, so that
the variety of representable relation algebras algebras and cylindric algebras of finite dimension $>2$
are not atom-canonical. (The complex algebra of an atom structure is
the completion of the term algebra.)
This construction is somewhat complicated using a rainbow atom structure. It has the striking consequence
that there are two atomic algebras sharing the same
atom structure, one is representable the other is not.

This construction was simplified and streamlined,
by many authors, including the first author \cite{weak}, but Hodkinson's  construction,
as we indicate below, has a very large potential to prove analogous  theorems on \d\ completions, and atom-canonicity
for several proper  sub-varieties of
the variety of representable cylindric-like algebras such as polyadic algebras with and without equality and Pinter's substitution algebras.

Our first theorem extends this result (existence of weakly representable atom structures that are not strongly representable)
to many cylindric like algebras of relations
using three different constructions based on three different graphs.
The first is due to Hokinson,
the second was given in \cite{weak} and the third is new. The three constructions
presented herein model theoretically gives a polyadic atomic equality algebra of dimension $n>2$
such that  the diagonal free reduct
of its completion is not
representable.

We also show
(the stronger result) that there for $n>2$ finite, there is is a polyadic equality atomic algebra $\A$ such that
$\Rd_{ca}\Cm\At\A\notin \SNr_n\CA_{n+4}$ inferring that the varieties
$S\Nr_n\K_{n+k}$, for $n>2$ finite for any $k\geq 4$, and for $\K\in \{\CA, \PEA\}$ are
not closed under \d\ completions.

Now that we have two distinct classes,
namely, the class of weakly atom structures and that of strongly atom structures; the most pressing need is to try to
classify them.
Venema proved (in a more general setting) that the former is elementary,
while Hirsch and Hodkinson show that the latter is {\it not} elementary. Their proof is amazing
depending on an ultraproduct of Erdos probabilistic graphs, witness theorem \ref{el}.

We know that there is a sequence of strongly representable atom structures whose ultraproduct is {\it only } weakly representable,
it is {\it not } strongly representable. This gives that the class $\K=\A\in \CA_n: \text { $\A$ is atomic and }
\Cm\At\A\in \sf RCA_n\}$ is not elementary, as well.

Here we extend Hirsch and Hodkinson's result to many cylindric-like algebras, answering a question of Hodkinson's for $\PA$ and $\PEA$.
The proof is based on the algebras constructed in \cite{hirsh} by noting that these algebras can be endowed with
polyadic operations (in an obvious way) and that they are generated by elements
whose dimension sets $<n$. The latter implies that an algebra is representable if and only if
its diagonal free reduct
is representable (One side is trivial, the other does not hold in general).

Lately, it has become fashionable in algebraic logic to
study representations of abstract algebras that has a complete representation.
A representation of $\A$ is roughly an injective homomorphism from $f:\A\to \wp(V)$
where $V$ is a set of $n$-ary sequences; $n$ is the dimension of $\A$, and the operations on $\wp(V)$
are concrete  and set theoretically
defined, like the boolean intersection and cylindrifiers or projections.
A complete representation is one that preserves  arbitrary disjuncts carrying
them to set theoretic unions.
If $f:\A\to \wp(V)$ is such a representation, then $\A$ is necessarily
atomic and $\bigcup_{x\in \At\A}f(x)=V$.

Let us focus on cylindric algebras.
It is known that there are countable atomic $\RCA_n$s when $n>2$,
that have no complete representations;
in fact, the class of completely representable $\CA_n$s when $n>2$, is not even elementary \cite{HHbook2}.

Such a phenomena is also closely
related to the algebraic notion of {\it atom canonicity}, as indicated, which is an important persistence property in modal logic
and to the metalogical property of  omitting types in finite variable fragments of first order logic.
Recall that a variety $V$ of boolean algebras with operators is atom-canonical,
if whenever $A\in V$, and $\A$ is atomic, then $\Cm\At\A\in V$.

If $\At$ is a weakly representable but
not strongly  representable, then
$\Cm\At$ is not representable; this gives that $\RCA_n$ for $n>2$ $n$ finite, is not atom canonical.
Also $\Cm\At\A$  is the \d\ completion of $\A$, and so obviously $\RCA_n$ is not closed under \d\ completions.

On the other hand, $\A$
cannot be completely  representable for, it can be shown without much ado,  that
a complete representation of $\A$ induces a representation  of $\Cm\At\A$.

Finally, if $\A$ is countable, atomic and has no complete representation
then the set of co-atoms (a co-atom is the complement of an atom), viewed in the corresponding Tarski Lindenbaum algebra,
$\Fm_T$, as a set of formulas, is a non principal-type that cannot be omitted in any model of $T$.

The reader is referred to \cite{Sayed} for an extensive discussion of such notions
and more.

Our main new results are

Let $n>2$ be finite.
\begin{enumarab}

\item Showing that for any class $\K$ between $\Df_n$ and $\PEA_n,$
there is a weakly representable atom structure that is not strongly representable.
In fact, we show that there is a relation algebra $\R$, such that the set of all $n$ basic matrices form a cylindric bases,
$\Tm{\sf Mat_n}\At\R$ is representable but  $\Rd_{df}\Cm{\sf Mat_n}\At\R$ is not. This gives the same result for $\R,$
cf. theorem \ref{hodkinson}

\item Showing that for $\K\in\{\PEA, \CA\},$ the class $S\Nr_n\K_{n+k}$ is not atom canonical for $k\geq 4$, theorem \ref{can}, extending
a result in \cite{can}. We give two entirely different proofs. This answers a question in \cite[problem 12, p.627]{HHbook}
which  was raised again in \cite[problem 1, p.131]{Sayedneat}.

\item Showing that the class of strongly representable atom structures  for any $\K$ between $\Df_n$ and $\PEA_n$ is not elementary,
theorem \ref{el}. This answers a question of Hodkinson's \cite[p.284]{AU}.

\item We show that  the omitting types theorem fails for finite variable fragments of first order logic, as long as we have more than $2$
variables,  even if we consider clique guarded
semantics, theorem  \ref{OTT}. This is stronger, in a way,  than the result reported in \cite{Sayed} and proved in detail in \cite{ANT}.
The latter result is also approached, and it is strengthened conditionally, cf. theorem \ref{blurs}.
The condition is the existence of certain finite relation
algebras.

\item Studying the omitting types theorem in many multi-dimensional modal logics,
that are cylindrifier free, but are endowed with the operations of finite substitutions or only
some of them, as a sample witness \ref{modal}.
We also prove a new negative omitting types theorem addressing first order definable expansions of first order logic as defined
in \cite{Biro}, theorem \ref{first}.

\item Showing that for a finite algebra in $\CA_3$ and $\PEA_3$,
it is undecidable whether $\A\in S\Nr_3\CA_{3+k}$ or not, for $k\geq 2$, theorem \ref{decidability}.
It follows that are finite algebras with no finite relativized representations. In contrast, we show that when 
we restrict taking subneat reducts to finite $m$ dimensional algebras, then for any $2\leq n<m$, 
every finite algebra has a finite relativized representation.

\end{enumarab}

\subsection*{On the notation}

The notation used is in conformity with  \cite{1} but the following list may help.
For rainbow cylindric algebras, we deviate from the notation therein,  we find it more convenient
to write $\CA_{\G, \Gamma}$ - treating the greens $\sf G$, as a parameter that can vary -
for the rainbow algebra $R(\Gamma)$.
The latter is defined to be $\Cm(\rho(\K))$ where $\K$ is a class of models in the rainbow signature satisfying the
$L_{\omega_1,\omega}$ rainbow theory, and $\rho(\K)$ is the rainbow atom structure \cite{HHbook2}.

We may also add polyadic operations, obtaining polyadic equality rainbow algebras,
denoted by $\PEA_{\G, \Gamma}.$
It is obvious how to add the polyadic operations.

But in all cases our view is the conventional (more restrictive) one adopted in \cite{Hodkinson};
we view these models as coloured graphs,
that is complete graphs labelled by the rainbow colours.
As usual,  we interpret `an edge coloured by a green say', to mean that the pair defining
the edge satisfies the corresponding green binary relation.
Our atom structures will consist of surjections from $n$ (the dimension) to finite coloured graphs,
roughly finite coloured graphs.

$\PEA$ denotes polyadic equality algebras, $\PA$ denotes polyadic algebras, $\Sc$ denotes Pinter's algebras, and of course
$\CA$ denotes cylindric algebras  \cite{1}.
$\PEA_n$ denotes $\PEA$s of dimension $n$ and same for all other algebras.
The following information is folklore: $\CA$s and $\PA$s are proper reducts of $\PEA$s and $\Sc$s
are proper  reducts of all. Finally, following the usual notation again,
$\Df$ denotes diagonal free $\CA$s, and these are
proper reducts of all of the above.
For an atomic $\A$ we write $\At\A$ for its atom structure and for an atom structure $\At$ we write $\Cm\At$ for its complex algebra.
In particular $\Cm\At\A$ is the completion of $\A$ in case our varieties are completely additive.
($\PA$s and $\Sc$s  are not completely additive).

${\sf QEA}$ stands for the class of quasi-polyadic equality algebras,
these coincide with $\PEA$ for finite dimensions,
but in the infinite dimensional case, the difference between these classes are substantial,
for example the latter always has uncountably many operations, which is not the case with the former in case of countable
dimensions.

Furthermore, $\sf QEA$ can be formulated in the context of a system of varieties
definable by a schema in the infinite dimensional case; which is not the case with $\PEA$.
Finally, $\sf QA$ denotes quasipolyadic algebras, and in comparison to $\PA$
the same can be said. To avoid conflicting notation, we write $\sf QEA$ and $\sf QA$
only for infinite dimensions, meaning that we are only dealing with finite cylindrifiers and finite substitutions.

By a graph we will mean a pair $\Gamma=(G, E)$,
where $G\not=\phi$ and $E\subseteq G\times G$ is a reflexive and
symmetric binary relation on $G$. We will often use the same
notation for $\Gamma$ and for its set of nodes ($G$ above). A pair
$(x, y)\in E$ will be called an edge of $\Gamma$. See \cite{graph}
for basic information (and a lot more) about graphs.
For neat reducts we follow \cite{Sayedneat} and for omitting types we follow
\cite{Sayed}.

$L_n$ or $L^n$ will denote first order logic restricted to the first $n$ variables and $L_{\infty, \omega}^n$
will denote $L_{\infty, \omega}$ restricted to the first $n$ variables. (In the latter case infinitary conjunctions, hence also infinitary
disjunctions, are allowed).

\section{Preliminaries}

The action of the non-boolean operators in a completely additive
atomic Boolean algebra with operators is determined by their behavior over the atoms, and
this in turn is encoded by the atom structure of the algebra.

\begin{definition}(\textbf{Atom Structure})
Let $\A=\langle A, +, -, 0, 1, \Omega_{i}:i\in I\rangle$ be
an atomic boolean algebra with operators $\Omega_{i}:i\in I$. Let
the rank of $\Omega_{i}$ be $\rho_{i}$. The \textit{atom structure}
$\At\A$ of $\A$ is a relational structure
$$\langle \At\A, R_{\Omega_{i}}:i\in I\rangle$$
where $\At\A$ is the set of atoms of $\A$ as
before, and $R_{\Omega_{i}}$ is a $(\rho(i)+1)$-ary relation over
$\At\A$ defined by
$$R_{\Omega_{i}}(a_{0},
\cdots, a_{\rho(i)})\Longleftrightarrow\Omega_{i}(a_{1}, \cdots,
a_{\rho(i)})\geq a_{0}.$$
\end{definition}
Similar 'dual' structure arise in other ways, too. For any not
necessarily atomic $BAO$ $\A$ as above, its
\textit{ultrafilter frame} is the
structure
$$\A_{+}=\langle {\sf Uf}(\A),R_{\Omega_{i}}:i\in I\rangle,$$
 where ${\sf Uf}(\A)$
is the set of all ultrafilters of (the boolean reduct of)
$\A$, and for $\mu_{0}, \cdots, \mu_{\rho(i)}\in
{\sf Uf}(\A)$, we put $R_{\Omega_{i}}(\mu_{0}, \cdots,
\mu_{\rho(i)})$ iff $\{\Omega(a_{1}, \cdots,
a_{\rho(i)}):a_{j}\in\mu_{j}$ for
$0<j\leq\rho(i)\}\subseteq\mu_{0}$.
\begin{definition}(\textbf{Complex algebra})
Conversely, if we are given an arbitrary structure
$\mathcal{S}=\langle S, r_{i}:i\in I\rangle$ where $r_{i}$ is a
$(\rho(i)+1)$-ary relation over $S$, we can define its
\textit{complex
algebra}
$$\Cm(\mathcal{S})=\langle \wp(S),
\cup, \setminus, \phi, S, \Omega_{i}\rangle_{i\in
I},$$
where $\wp(S)$ is the power set of $S$, and
$\Omega_{i}$ is the $\rho(i)$-ary operator defined
by$$\Omega_{i}(X_{1}, \cdots, X_{\rho(i)})=\{s\in
S:\exists s_{1}\in X_{1}\cdots\exists s_{\rho(i)}\in X_{\rho(i)},
r_{i}(s, s_{1}, \cdots, s_{\rho(i)})\},$$ for each
$X_{1}, \cdots, X_{\rho(i)}\in\wp(S)$.
\end{definition}
It is easy to check that, up to isomorphism,
$\At(\Cm(\mathcal{S}))\cong\mathcal{S}$ always, and
$\A\subseteq\Cm(\At\A)$ for any
completely additive atomic boolean algebra with operators $\A$. If $\A$ is
finite then of course
$\A\cong\Cm(\At\A)$.\\
\begin{enumerate}
\item{Atom structure of diagonal free-type algebra is $\mathcal{S}=\langle S, R_{c_{i}}:i<n\rangle$, where the $R_{c_{i}}$
is binary relation on $S$.}
\item{Atom structure of cylindric-type algebra is $\mathcal{S}=\langle S, R_{c_{i}}, R_{d_{ij}}:i, j<n\rangle$, where the
$R_{d_{ij}}$, $R_{c_{i}}$ are unary and binary relations on $S$. The
reduct $\mathfrak{Rd}_{df}\mathcal{S}=\langle S,
R_{c_{i}}:i<n\rangle$ is an atom structure of diagonal free-type.}
\item{Atom structure of substitution-type algebra is $\mathcal{S}=\langle S, R_{c_{i}}, R_{s^{i}_{j}}:i, j<n\rangle$, where the
$R_{d_{ij}}$, $R_{s^{i}_{j}}$ are unary and binary relations on $S$,
respectively. The reduct $\mathfrak{Rd}_{df}\mathcal{S}=\langle S,
R_{c_{i}}:i<n\rangle$ is an atom structure of diagonal free-type.}
\item{Atom structure of quasi polyadic-type algebra is $\mathcal{S}=\langle S, R_{c_{i}}, R_{s^{i}_{j}}$, $R_{s_{ij}}:i, j<n\rangle$, where the
$R_{c_{i}}$, $R_{s^{i}_{j}}$ and $R_{s_{ij}}$ are binary relations
on $S$. The reducts $\mathfrak{Rd}_{df}\mathcal{S}=\langle S,
R_{c_{i}}:i<n\rangle$ and $\mathfrak{Rd}_{Sc}\mathcal{S}=\langle S,
R_{c_{i}}, R_{s^{i}_{j}}:i, j<n\rangle$ are atom structures of
diagonal free and substitution types, respectively.}

\item{ The atom structure of quasi polyadic equality-type algebra
is $\mathcal{S}=\langle S, R_{c_{i}}, R_{d_{ij}}, R_{s^{i}_{j}}, R_{s_{ij}}:i, j<n\rangle$, where the $R_{d_{ij}}$
is unary relation on $S$, and $R_{c_{i}}$, $R_{s^{i}_{j}}$ and
$R_{s_{ij}}$ are binary relations on $S$.
\begin{enumerate}\item{The
reduct $\mathfrak{Rd}_{df}\mathcal{S}=\langle S, R_{c_{i}}:i\in
I\rangle$ is an atom structure of diagonal free-type.}
\item{The
reduct $\mathfrak{Rd}_{ca}\mathcal{S}=\langle S, R_{c_{i}},
R_{d_{ij}}:i, j\in I\rangle$ is an atom structure of
cylindric-type.}
\item{The
reduct $\mathfrak{Rd}_{Sc}\mathcal{S}=\langle S, R_{c_{i}},
R_{s^{i}_{j}}:i, j\in I\rangle$ is an atom structure of
substitution-type.}
\item{The
reduct $\mathfrak{Rd}_{qa}\mathcal{S}=\langle S, R_{c_{i}},
R_{s^{i}_{j}}, R_{s_{ij}}:i, j\in I\rangle$ is an atom structure of
quasi polyadic-type.}
\end{enumerate}}
\end{enumerate}
\begin{definition}
An algebra is said to be representable if and only if it is isomorphic to a subalgebra of a direct product of set algebras of the same type.
\end{definition}

\begin{definition}\label{strong rep}Let $\mathcal{S}$ be an $n$-dimensional algebra atom
structure. $\mathcal{S}$ is \textit{strongly representable} if every
atomic $n$-dimensional algebra $\A$ with
$\At\A=\mathcal{S}$ is representable. We write $\sf SDfS_{n}$,
$\sf SCS_{n}$, $\sf SSCS_{n}$, $\sf SQS_{n}$ and $\sf SQES_{n}$ for the classes of
strongly representable ($n$-dimensional) diagonal free, cylindric,
substitution, quasi polyadic and quasi polyadic equality algebra
atom structures, respectively.
\end{definition}
Note that for any $n$-dimensional algebra $\A$ and atom
structure $\mathcal{S}$, if $\At\A=\mathcal{S}$ then
$\A$ embeds into $\Cm\mathcal{S}$, and hence
$\mathcal{S}$ is strongly representable iff
$\Cm\mathcal{S}$ is representable.

\section{Weakly representable atom structures that are not strongly representable, and atom canonicity}

Here we use fairly simple model theoretic arguments, to prove the non atom-canonicity of several classes consisting of subneat reducts.
We need  some standard model theoretic preparations.
This proof unifies Rainbow and Monk like constructions used in \cite{Hodkinson} and
\cite{weak}, and introduces  a new Monk-like algebra based on a new graph (different than those used in the later two references).
Such Monk-like algebras will be abstracted later on, presented as atom structures of models of a first order theory
as presented in \cite{HHbook2}; however, the construction is different.

\begin{theorem}
Let $\Theta$ be an $n$-back-and-forth system
of partial isomorphism on a structure $A$, let $\bar{a}, \bar{b} \in {}^{n}A$,
and suppose that $ \theta = ( \bar{a} \mapsto \bar{b})$ is a map in
$\Theta$. Then $ A \models \phi(\bar{a})$ iff $ A \models
\phi(\bar{b})$, for any formula $\phi$ of $L^n_{\infty \omega}$.
\end{theorem}
\begin{proof} By induction on the structure of $\phi$.
\end{proof}
Suppose that $W \subseteq {}^{n}A$ is a given non-empty set. We can
relativize quantifiers to $W$, giving a new semantics $\models_W$
for $L^n_{\infty \omega}$, which has been intensively studied in
recent times. If $\bar{a} \in W$:
\begin{itemize}
\item for atomic $\phi$, $A\models_W \phi(\bar{a})$
iff $A \models \phi(\bar{a})$

\item the boolean clauses are as expected

\item for $ i < n, A \models_W \exists x_i \phi(\bar{a})$ iff $A \models_W
\phi(\bar{a}')$ for some $ \bar{a}' \in W$ with $\bar{a}' \equiv_i
\bar{a}$.
\end{itemize}

\begin{theorem} If $W$ is $L^n_{\infty \omega}$ definable, $\Theta$ is an
 $n$-\textit{back-and-forth} system
of partial isomorphisms on $A$, $\bar{a}, \bar{b} \in W$, and $
\bar{a} \mapsto \bar{b} \in \Theta$, then $ A \models \phi(\bar{a})$
iff $ A \models \phi(\bar{b})$ for any formula $\phi$ of
$L^n_{\infty \omega}$.
\end{theorem}
\begin{proof} Assume that $W$ is definable by the $L^n_{\infty \omega}$
formula $\psi$, so that $W = \{ \bar{a} \in {}^{n}A:A\models \psi(a)\}$. We may
relativize the quantifiers of $L^n_{\infty \omega}$-formulas to
$\psi$. For each $L^n_{\infty
\omega}$-formula $\phi$ we obtain a relativized one, $\phi^\psi$, by
induction, the main clause in the definition being:
\begin{itemize}
\item $( \exists x_i \phi)^\psi = \exists x_i ( \psi \wedge
\phi^\psi)$.
\end{itemize}
 Then clearly, $ A \models_W \phi(\bar{a})$ iff $ A \models
 \phi^\psi(\bar{a})$, for all $ \bar{a} \in W$.
\end{proof}

The following theorem unifies and generalizes the main theorem in \cite{Hodkinson} and in \cite{weak}.
It shows that sometime Monk like algebras and rainbow algebras do the same thing.
We shall see that each of the constructions has its assets and liabilities.
In the rainbow case the construction can be refined by truncating the greens and reds to be finite, to give sharper results
as shown in theorem \label{smooth}.

While Monk's algebra in one go gives the required result for both relation and cylindric like algebras.
and it can be generalized to show that the class of strongly representable atom structures for
both relation and cylindric algebras is not elementary reproving a profound result
of Hirsch and Hodkinson, using Erdos probabilistic graphs, or {\it anti Monk} ultraproducts \ref{el}.

First a piece of notation. For an atomic  relation algebra $\R$, ${\sf Mat_n}\At\R$
denotes the set of basic matrices on $\R$. We say that ${\sf Mat_n}\At\R$ is a polyadic basis
if it is a symmetric cylindric bases
(closed under substitutions).

The idea of the proof, which is quite technical,  is summarized in the following:

\begin{enumarab}

\item  We construct a labelled hypergraph $M$ that can be viewed as an $n$ homogeneous  model of a certain theory
(in the rainbow case it is an $L_{\omega_1, \omega}$ theory, in the Monk case it is first order.)
This model gets its labels from a fixed in advance graph $\G$; that also determines the signature of $M$.
By $n$ homogeneous we mean that every partial isomorphism of $M$ of size $\leq n$ can be extended
to an automorphism of $M$.

\item We have finitely many shades of red, outside the signature; these are basically non principal ultrafilters, but
they can be used as labels.

\item We build a relativized set algebra based on $M$, by discarding all assignments whose edges are labelled
by shades of red getting $W\subseteq {}^nM$.

\item $W$ is definable in $^nM$ be an $L_{\infty, \omega}$ formula
hence the semantics with respect to  $W$ coincides with classical Tarskian semantics (when assignments are in $^nM$).
This is proved using certain $n$ back and forth systems.

\item The set algebra based on $W$ (consisting of sets of sequences (without shades of reds labelling edges)
satisfying formulas in $L^n$ in the given signature)
will be an atomic algebra such that its completion is not representable.
The completion will consist of interpretations of $L_{\infty,\omega}^n$ formulas; though represented over $W$,
it will not be, and cannot be, representable in the classical  sense.

\end{enumarab}

\begin{theorem}\label{hodkinson}

\begin{enumarab}
\item There exists a polyadic equality atom structure $\At$ of dimension $n$,
such that $\Tm\At$ is representable as a $\sf PEA_n$, but not strongly representable.
In fact $\Rd_t\Cm\At$ is not representable for any signature $t$ between that of $\sf Df$ and $\sf Sc$.

\item Furthermore,
there exists an atomic  relation algebra $\R$ that that the set of all basic matrices
forms an $n$ dimensional polyadic basis
such that $\Tm{\sf Mat}_n\At\R$ is  representable as a polyadic equality algebra, while
$\Rd_{df}\Cm{\sf Mat_n}\At\R$ is not (as a diagonal free cylindric algebra of dimension $n$).
In particular, $\At\R$ is weakly but not strongly representable.
\end{enumarab}
\end{theorem}

\begin{proof}
\begin{enumarab}

\item {\sf Constructing an $n$ homogeneous model}

$L^+$ is the rainbow signature consisting of the binary
relation symbols $\g_i :i<n-1 , \g_0^i: i< |\sf G|$, $\w, \w_i: i <n-2, \r_{jk}^i:  i<\omega, j<k<|\sf R|$
and the $(n-1)$ ary-relation symbols
$\y_S: S\subseteq {\sf G}$
together with a shade of red $\rho$ which is outside the rainbow signature,
but it is a binary relation, in the sense that it can label edges in coloured
graphs. Here we take, like Hodkinson \cite{Hodkinson}, $\sf G=\sf R=\omega$.
graphs. In the following theorem we shall see that by varying these parameters,
namely when $|\sf G|=n+2$ and $|\sf R|=n+1$ we get sharper results.

Let $\GG$ be the class of all coloured graph in this rainbow signature.
Let $T_r$ denote the rainbow $L_{\omega_1,\omega}$ theory \cite{HHbook2}.
Let $\G$ by a countable disjoint union of cliques each of size $n(n-1)/2$ or $\N$ with edge relation defined by $(i,j)\in E$ iff $0<|i-j|<N$.
Let $L^+$ be the signature consisting of the binary
relation symbols $(a, i)$, for each $a \in \G \cup \{ \rho \}$ and
$i < n$. Let $T_m$ denote the following (Monk) theory:

$M\models T_m$ iff
for all $a,b\in M$, there is a unique $p\in \G\cup \{\rho\}\times n$, such that
$(a,b)\in p$ and if  $M\models (a,i)(x,y)\land (b,j)(y,z)\land (c,l)(x,z)$, then $| \{ i, j, l \}> 1 $, or
$ a, b, c \in \G$ and $\{ a, b, c\} $ has at least one edge
of $\G$, or exactly one of $a, b, c$ -- say, $a$ -- is $\rho$, and $bc$ is
an edge of $\G$, or two or more of $a, b, c$ are $\rho$.
We denote the class of models of $T_m$ which can also be viewed as coloured graphs with labels coming from
$\G\cup {\rho}\times n$ also by $\GG$.

We deliberately
use this notation to emphasize the fact
that the proof for all three cases considered is essentially the same; the
difference is that our our  three term algebras, to be constructed  are set algebras with domain a subset of $^nM$, where
$M$ is a model that embeds al coloured graphs and to be defined shortly,
are just based on a different graph. So such a notation simplifies matters considerably.

Now in all cases,  there is a countable $n$ homogeneous coloured  graph  (model) $M\in \GG$ of both theories with the following
property:\\
$\bullet$ If $\triangle \subseteq \triangle' \in \GG$, $|\triangle'|
\leq n$, and $\theta : \triangle \rightarrow M$ is an embedding,
then $\theta$ extends to an embedding $\theta' : \triangle'
\rightarrow M$.

We proceed as follows. We use a simple game.
Two players, $\forall$ and $\exists$, play a game to build a
labelled graph $M$. They play by choosing a chain $\Gamma_0
\subseteq \Gamma_1 \subseteq\ldots $ of finite graphs in $\GG$; the
union of
the chain will be the graph $M.$
There are $\omega$ rounds. In each round, $\forall$ and $\exists$ do
the following. Let $ \Gamma \in \GG$ be the graph constructed up to
this point in the game. $\forall$ chooses $\triangle \in \GG$ of
size $< n$, and an embedding $\theta : \triangle \rightarrow
\Gamma$. He then chooses an extension $ \triangle \subseteq
\triangle^+ \in \GG$, where $| \triangle^+ \backslash \triangle |
\leq 1$. These choices, $ (\triangle, \theta, \triangle^+)$,
constitute his move. $\exists$ must respond with an extension $
\Gamma \subseteq \Gamma^+ \in \GG$ such that $\theta $ extends to an
embedding $\theta^+ : \triangle^+ \rightarrow \Gamma^+$. Her
response ends the round.
The starting graph $\Gamma_0 \in \GG$ is arbitrary but we will take
it to be the empty graph in $\GG$.
We claim that $\exists$ never gets stuck -- she can always find a suitable
extension $\Gamma^+ \in \GG$.  Let $\Gamma \in \GG$ be the graph built at some stage, and let
$\forall$ choose the graphs $ \triangle \subseteq \triangle^+ \in
\GG$ and the embedding $\theta : \triangle \rightarrow \Gamma$.
Thus, his move is $ (\triangle, \theta, \triangle^+)$.
We now describe $\exists$'s response. If $\Gamma$ is empty, she may
simply plays $\triangle^+$, and if $\triangle = \triangle^+$, she
plays $\Gamma$. Otherwise, let $ F = rng(\theta) \subseteq \Gamma$.
(So $|F| < n$.) Since $\triangle$ and $\Gamma \upharpoonright F$ are
isomorphic labelled graphs (via $\theta$), and $\GG$ is closed under
isomorphism, we may assume with no loss of generality that $\forall$
actually played $ ( \Gamma \upharpoonright F, Id_F, \triangle^+)$,
where $\Gamma \upharpoonright F \subseteq \triangle^+ \in \GG$,
$\triangle^+ \backslash F = \{\delta\}$, and $\delta \notin \Gamma$.
We may view $\forall$'s move as building a labelled graph $ \Gamma^*
\supseteq \Gamma$, whose nodes are those of $\Gamma$ together with
$\delta$, and whose edges are the edges of $\Gamma$ together with
edges from $\delta$ to every node of $F$. The labelled graph
structure on $\Gamma^*$ is given by\\
$\bullet$ $\Gamma$ is an induced subgraph of $\Gamma^*$ (i.e., $
\Gamma \subseteq \Gamma^*$)\\
$\bullet$ $\Gamma^* \upharpoonright ( F \cup \{\delta\} ) =
\triangle^+$.
Now $ \exists$ must extend $ \Gamma^*$ to a complete
graph on the same node and complete the colouring yielding a graph
$ \Gamma^+ \in \GG$. Thus, she has to define the colour $
\Gamma^+(\beta, \delta)$ for all nodes $ \beta \in \Gamma \backslash
F$, in such a way as to meet the required conditions.  For rainbow case \pe\ plays as follows:
\begin{enumarab}

\item If there is no $f\in F$,
such that $\Gamma^*(\beta, f), \Gamma^*(\delta ,f)$ are coloured $\g_0^t$ and $\g_0^u$
for some $t,u$, then \pe\ defined $\Gamma^+(\beta, \delta)$ to  be $\w_0$.

\item Otherwise, if for some $i$ with $0<i<n-1$, there is no $f\in F$
such that $\Gamma^*(\beta,f), \Gamma^*(\delta, f)$ are both coloured $\g_i$, then \pe\
defines the colour $\Gamma^+(\beta,\delta)$ to
to be $\w_i$ say the least such

\item Otherwise $\delta$ and $\beta$ are both the apexes on $F$ in $\Gamma^*$ that induce
the same linear ordering on (there are nor green edges in $F$ because
$\Delta^+\in \GG$, so it has no green triangles).
Now \pe\ has no choice but to pick a red. The colour she chooses is $\rho.$

\end{enumarab}
This defines the colour of edges. Now for hyperedges,
for  each tuple of distinct elements
$\bar{a}=(a_0,\ldots a_{n-2})\in {}^{n-1}(\Gamma^+)$
such that $\bar{a}\notin {}^{n-1}\Gamma\cup {}^{n-1}\Delta$ and with no edge $(a_i, a)$
coloured greed in  $\Gamma^+$, \pe\ colours $\bar{a}$ by $\y_{S}$
where
$S=\{i <\omega: \text { there is a $i$ cone with base  } \bar{a}\}$.
Notice that $|S|\leq F$. This strategy works.

For the Monk case \pe\ plays as follows:

Now $ \exists$ must extend $ \Gamma^*$ to a complete
graph on the same node and complete the colouring yielding a graph
$ \Gamma^+ \in \GG$. Thus, she has to define the colour $
\Gamma^+(\beta, \delta)$ for all nodes $ \beta \in \Gamma \backslash
F$, in such a way as to meet the conditions of definition 1. She
does this as follows. The set of colours of the labels in $ \{
\triangle^+(\delta, \phi) : \phi \in F \} $ has cardinality at most
$ n - 1 $. Let  $ i < n$ be a "colour" not in this set. $ \exists$
labels $(\delta, \beta) $ by $(\rho, i)$ for every $ \beta \in
\Gamma \backslash F$. This completes the definition of $ \Gamma^+$.
\\
It remains to check that this strategy works--that the conditions
from the definition of $\GG$ are met. It is
sufficient to note that

\begin{itemize}
 \item if $\phi \in F$ and $ \beta \in \Gamma \backslash F$, then
the labels in $ \Gamma^+$ on the edges of the triangle $(\beta,
\delta, \phi)$ are not all of the same colour ( by choice of $i$ )

\item if $ \beta, \gamma \in \Gamma \backslash F$, then two the
labels in $ \Gamma^+$ on the edges of the triangle $( \beta, \gamma,
\delta )$ are $( \rho, i)$.\\

\end{itemize}

Now there are only countably many
finite graphs in $ \GG$ up to isomorphism, and each of the graphs
built during the game is finite. Hence $\forall$ may arrange to play
every possible $(\triangle, \theta, \triangle^+)$ (up to
isomorphism) at some round in the game. Suppose he does this, and
let $M$ be the union of the graphs played in the game. We check that
$M$ is as required. Certainly, $M \in \GG$, since $\GG$ is clearly
closed under unions of chains. Also, let $\triangle \subseteq
\triangle' \in \GG$ with $|\triangle'| \leq n$, and $ \theta :
\triangle \rightarrow M$ be an embedding. We prove that $\theta$
extends to $\triangle'$, by induction on $d = | \triangle'
\backslash \triangle|.$ If this is $0$, there is nothing to prove.
Assume the result for smaller $d$. Choose  $ a \in \triangle'
\backslash \triangle $ and let $ \triangle^+ = \triangle'
\upharpoonright ( \triangle \cup \{ a \} ) \in \GG$. As, $
|\triangle| < n$, at some round in the game, at which the graph
built so far was $\Gamma$, say, $\forall$ would have played
$(\triangle, \theta, \triangle^+)$ (or some isomorphic triple).
Hence, if $\exists$ constructed $ \Gamma^+$ in that round, there is
an embedding $ \theta^+ : \triangle^+ \rightarrow \Gamma^+ $
extending $\theta$. As $ \Gamma^+ \subseteq M,  \theta^+$ is also an
embedding: $ \triangle^+ \rightarrow M$. Since $ |\triangle'
\backslash \triangle^+| < d, \theta^+$ extends inductively to an
embedding $\theta' : \triangle' \rightarrow M $, as required.

\item {\sf Relativization, back and forth systems ensuring that relativized semantics coincide with the classical semantics}

For the rainbow algebra, let $$W_r = \{ \bar{a} \in {}^n M : M \models ( \bigwedge_{i < j < n,
l < n} \neg \rho(x_i, x_j))(\bar{a}) \},$$
and for the Monk like algebra, $W_m$
is defined exactly the same way by discarding assignments whose edges are coloured by one of the $n$ reds
$(\rho, i)$, $i<n$.  In the former case we are discarding assignments who have a $\rho$ labelled edge,
and in the latter we are discarding assignments
that involve edges coloured by any of the $n$ reds $(\rho, i)$, $i<n$.
We denote both by $W$ relying on context.

The $n$-homogeneity built into
$M$, in all three cases by its construction implies that the set of all partial
isomorphisms of $M$ of cardinality at most $n$ forms an
$n$-back-and-forth system. But we can even go
further. We formulate our statement for the Monk algebra based on $\G$ whose underlying set
is $\N$ since this is the new case. (For the other case the reader is referred to \cite{Hodkinson} and \cite{weak}).

A definition: Let $\chi$ be a permutation of the set $\omega \cup \{ \rho\}$. Let
$ \Gamma, \triangle \in \GG$ have the same size, and let $ \theta :
\Gamma \rightarrow \triangle$ be a bijection. We say that $\theta$
is a $\chi$-\textit{isomorphism} from $\Gamma$ to $\triangle$ if for
each distinct $ x, y \in \Gamma$,
\begin{itemize}
\item If $\Gamma ( x, y) = (a, j)$ with $a\in \N$, then there exist unique $l\in \N$ and $r$ with $0\leq r<N$ such that
$a=Nl+r$.
\begin{equation*}
\triangle( \theta(x),\theta(y)) =
\begin{cases}
( N\chi(i)+r, j), & \hbox{if $\chi(i) \neq \rho$} \\
(\rho, j),  & \hbox{otherwise.} \end{cases}
\end{equation*}
\end{itemize}

\begin{itemize}
\item If $\Gamma ( x, y) = ( \rho, j)$, then
\begin{equation*}
\triangle( \theta(x),\theta(y)) \in
\begin{cases}
\{( N\chi(\rho)+s, j): 0\leq s < N \}, & \hbox{if $\chi(\rho) \neq \rho$} \\
\{(\rho, j)\},  & \hbox{otherwise.} \end{cases}
\end{equation*}
\end{itemize}

We now have for any permutation $\chi$ of $\omega \cup \{\rho\}$, $\Theta^\chi$
is the set of partial one-to-one maps from $M$ to $M$ of size at
most $n$ that are $\chi$-isomorphisms on their domains. We write
$\Theta$ for $\Theta^{Id_{\omega \cup \{\rho\}}}$.

We claim that for any any permutation $\chi$ of $\omega \cup \{\rho\}$, $\Theta^\chi$
is an $n$-back-and-forth system on $M$.

Clearly, $\Theta^\chi$ is closed under restrictions. We check the
``forth" property. Let $\theta \in \Theta^\chi$ have size $t < n$.
Enumerate $\dom(\theta)$, $\rng(\theta).$ respectively as $ \{ a_0,
\ldots, a_{t-1} \}$, $ \{ b_0,\ldots b_{t-1} \}$, with $\theta(a_i)
= b_i$ for $i < t$. Let $a_t \in M$ be arbitrary, let $b_t \notin M$
be a new element, and define a complete labelled graph $\triangle
\supseteq M \upharpoonright \{ b_0,\ldots, b_{t-1} \}$ with nodes
$\{ b_0,\ldots, b_{t} \}$ as follows.\\

Choose distinct "nodes"$ e_s < N$ for each $s < t$, such that no
$(e_s, j)$ labels any edge in $M \upharpoonright \{ b_0,\dots,
b_{t-1} \}$. This is possible because $N \geq n(n-1)/2$, which
bounds the number of edges in $\triangle$. We can now define the
colour of edges $(b_s, b_t)$ of $\triangle$ for $s = 0,\ldots, t-1$.

\begin{itemize}
\item If $M ( a_s, a_t) = ( Ni+r, j)$, for some $i\in \N$ and $0\leq r<N$, then
\begin{equation*}
\triangle( b_s, b_t) =
\begin{cases}
( N\chi(i)+r, j), & \hbox{if $\chi(i) \neq \rho$} \\
\{(\rho, j)\},  & \hbox{otherwise.} \end{cases}
\end{equation*}
\end{itemize}

\begin{itemize}
\item If $M ( a_s, a_t) = ( \rho, j)$, then assuming that $e_s=Ni+r,$ $i\in \N$ and $0\leq r<N$,
\begin{equation*}
\triangle( b_s, b_t) =
\begin{cases}
( N\chi(\rho)+r, j), & \hbox{if $\chi(\rho) \neq \rho$} \\
\{(\rho, j)\},  & \hbox{otherwise.} \end{cases}
\end{equation*}
\end{itemize}

This completes the definition of $\triangle$. It is easy to check
that $\triangle \in \GG$. Hence, there is a graph embedding $ \phi : \triangle \rightarrow M$
extending the map $ Id_{\{ b_0,\ldots b_{t-1} \}}$. Note that
$\phi(b_t) \notin \rng(\theta)$. So the map $\theta^+ = \theta \cup
\{(a_t, \phi(b_t))\}$ is injective, and it is easily seen to be a
$\chi$-isomorphism in $\Theta^\chi$ and defined on $a_t$.
The converse,``back" property is similarly proved ( or by symmetry,
using the fact that the inverse of maps in $\Theta$ are
$\chi^{-1}$-isomorphisms).

We now  derive a connection between classical and
relativized semantics in $M$, over the set $W$:\\

Recall that $W$ is simply the set of tuples $\bar{a}$ in ${}^nM$ such that the
edges between the elements of $\bar{a}$ don't have a label involving
$\rho$ (these are $(\rho, i) i<n$). Their labels are all of the form $(Ni+r, j)$. We can replace $\rho$-labels by suitable $(a, j)$-labels
within an $n$-back-and-forth system. Thus, we may arrange that the
system maps a tuple $\bar{b} \in {}^n M \backslash W$ to a tuple
$\bar{c} \in W$ and this will preserve any formula
containing no relation symbols $(a, j)$ that are ``moved" by the
system.

Indeed, we show that the
classical and $W$-relativized semantics agree.
$M \models_W \varphi(\bar{a})$ iff $M \models \varphi(\bar{a})$, for
all $\bar{a} \in W$ and all $L^n$-formulas $\varphi$.

The proof is by induction on $\varphi$. If $\varphi$ is atomic, the
result is clear; and the boolean cases are simple.
Let $i < n$ and consider $\exists x_i \varphi$. If $M \models_W
\exists x_i \varphi(\bar{a})$, then there is $\bar{b} \in W$ with
$\bar{b} =_i \bar{a}$ and $M \models_W \varphi(\bar{b})$.
Inductively, $M \models \varphi(\bar{b})$, so clearly, $M \models_W
\exists x_i \varphi(\bar{a})$.
For the (more interesting) converse, suppose that $M \models_W
\exists x_i \varphi(\bar{a})$. Then there is $ \bar{b} \in {}^n M$
with $\bar{b} =_i \bar{a}$ and $M \models \varphi(\bar{b})$. Take
$L_{\varphi, \bar{b}}$ to be any finite subsignature of $L$
containing all the symbols from $L$ that occur in $\varphi$ or as a
label in $M \upharpoonright \rng(\bar{b})$. (Here we use the fact
that $\varphi$ is first-order. The result may fail for infinitary
formulas with infinite signature.) Choose a permutation $\chi$ of
$\omega \cup \{\rho\}$ fixing any $i'$ such that some $(Ni'+r, j)$
occurs in $L_{\varphi, \bar{b}}$ for some $r<N$, and moving $\rho$.
Let $\theta = Id_{\{a_m : m \neq i\}}$. Take any distinct $l, m \in
n \setminus \{i\}$. If $M(a_l, a_m) = (Ni'+r, j)$, then $M( b_l,
b_m) = (Ni'+r, j)$ because $ \bar{a} = _i \bar{b}$, so $(Ni'+r, j)
\in L_{\varphi, \bar{b}}$ by definition of $L_{\varphi, \bar{b}}$.
So, $\chi(i') = i'$ by definition of $\chi$. Also, $M(a_l, a_m) \neq
( \rho, j)$(any $j$) because $\bar{a} \in W$. It now follows that
$\theta$ is a $\chi$-isomorphism on its domain, so that $ \theta \in
\Theta^\chi$.
Extend $\theta $ to $\theta' \in \Theta^\chi$ defined on $b_i$,
using the ``forth" property of $ \Theta^\chi$. Let $
\bar{c} = \theta'(\bar{b})$. Now by choice of of $\chi$, no labels
on edges of the subgraph of $M$ with domain $\rng(\bar{c})$ involve
$\rho$. Hence, $\bar{c} \in W$.
Moreover, each map in $ \Theta^\chi$ is evidently a partial
isomorphism of the reduct of $M$ to the signature $L_{\varphi,
\bar{b}}$. Now $\varphi$ is an $L_{\varphi, \bar{b}}$-formula.
We have $M \models \varphi(\bar{a})$ iff $M \models \varphi(\bar{c})$.
So $M \models \varphi(\bar{c})$. Inductively, $M \models_W
\varphi(\bar{c})$. Since $ \bar{c} =_i \bar{a}$, we have $M
\models_W \exists x_i \varphi(\bar{a})$ by definition of the
relativized semantics. This completes the induction.

\item {\sf The atoms, the required atomic algebra}

Now let $L$ be $L^+$ without the red relation symbols.
The logics $L_n$ and $L^n_{\infty \omega}$ are taken in this
signature.

For an $L^n_{\infty \omega}$-formula $\varphi $,  define
$\varphi^W$ to be the set $\{ \bar{a} \in W : M \models_W \varphi
(\bar{a}) \}$.

Then the set algebra $\A$ (actually  in all three cases) is taken to be the relativized set algebra with domain
$$\{\varphi^W : \varphi \,\ \textrm {a first-order} \;\ L_n-
\textrm{formula} \}$$  and unit $W$, endowed with the algebraic
operations ${\sf d}_{ij}, {\sf c}_i, $ etc., in the standard way, and of course formulas are taken in the suitable
signature.
The completion of $\A$ is the algebra $\C$ with universe $\{\phi^W: \phi\in L_{\infty,\omega}^n\}$
with operations defined as for $\A$, namely, usual cylindrifiers and diagonal elements,
reflecting existential quantifiers, polyadic operations and equality.
Indeed, $\A$ is a representable (countable) atomic polyadic algebra of dimension $n$

Let $\cal S$ be the polyadic set algebra with domain  $\wp ({}^{n} M )$ and
unit $ {}^{n} M $. Then the map $h : \A
\longrightarrow S$ given by $h:\varphi ^W \longmapsto \{ \bar{a}\in
{}^{n} M: M \models \varphi (\bar{a})\}$ can be checked to be well -
defined and one-one. It clearly respects the polyadic operations, also because relativized semantics and classical semantics coincide on $L_n$
formulas in the given signature, this is a representation of $\A.$

\item {\sf The atoms}

A formula $ \alpha$  of  $L_n$ is said to be $MCA$
('maximal conjunction of atomic formulas') if (i) $M \models \exists
x_0\ldots x_{n-1} \alpha $ and (ii) $\alpha$ is of the form
$$\bigwedge_{i \neq j < n} \alpha_{ij}(x_i, x_j),$$
where for each $i,j,\alpha_{ij}$ is either $x_i=x_i$ or $R(x_i,x_j)$
for some binary relation symbol $R$ of $L$.

Let $\varphi$ be any $L^n_{\infty\omega}$-formula, and $\alpha$ any
$MCA$-formula. If $\varphi^W \cap \alpha^W \neq \emptyset $, then
$\alpha^W \subseteq \varphi^W $.
Indeed, take $\bar{a} \in  \varphi^W \cap \alpha^W$. Let $\bar{a} \in
\alpha^W$ be arbitrary. Clearly, the map $( \bar{a} \mapsto
\bar{b})$ is in $\Theta$. Also, $W$ is
$L^n_{\infty\omega}$-definable in $M$, since we have
$$ W = \{
\bar{a} \in {}^n M : M \models (\bigwedge_{i < j< n} (x_i = x_j \vee
\bigvee_{R \in L} R(x_i, x_j)))(\bar{a})\}.$$
We have $M \models_W \varphi (\bar{a})$
 iff $M \models_W \varphi (\bar{b})$. Since $M \models_W \varphi (\bar{a})$, we have
$M \models_W \varphi (\bar{b})$. Since $\bar{b} $ was arbitrary, we
see that $\alpha^W \subseteq \varphi^W$.
Let $$F = \{ \alpha^W : \alpha \,\ \textrm{an $MCA$},
L^n-\textrm{formula}\} \subseteq \A.$$
Evidently, $W = \bigcup F$. We claim that
$\A$ is an atomic algebra, with $F$ as its set of atoms.
First, we show that any non-empty element $\varphi^W$ of $\A$ contains an
element of $F$. Take $\bar{a} \in W$ with $M \models_W \varphi
(\bar{a})$. Since $\bar{a} \in W$, there is an $MCA$-formula $\alpha$
such that $M \models_W \alpha(\bar{a})$. Then $\alpha^W
\subseteq \varphi^W $. By definition, if $\alpha$ is an $MCA$ formula
then $ \alpha^W$ is non-empty. If $ \varphi$ is
an $L^n$-formula and $\emptyset \neq \varphi^W \subseteq \alpha^W $,
then $\varphi^W = \alpha^W$. It follows that each $\alpha^W$ (for
$MCA$ $\alpha$) is an atom of $\A$.

Now for the rainbow signature, a formula $ \alpha$  of  $MCA$ formulas of $L_n$
are adapted to the rainbow signature.
$\alpha$ is such if  $M \models \exists
x_0\ldots x_{n-1} \alpha $ and (ii) $\alpha$ is of the form
$$\bigwedge_{i \neq j < n} \alpha_{ij}(x_i, x_j) \land y_S(x_0,\ldots x_{n-1}),$$
where for each $i,j,\alpha_{ij}$ is either $x_i=x_i$ or $R(x_i,x_j)$
for some binary relation symbol $R$ of the rainbow signature.
In this case there is a one to one correspondence
between coloured graphs whose edges do not involve the red shade and the $MCA$ formulas, and both are the the atom
of $\A$, and for that matter its completion.
(here any $s$ satisfying an $MCA$ formula defines a coloured graph and for any other $s'$,  the graph determined by
it is isomorphic to that determined by $s$; and thats why precisely
$MCA$ formulas are atoms, namely, surjections from $n$ to coloured graphs.

\item {\sf The complex algebra}

Define $\C$ to be the complex algebra over $\At\A$, the atom structure of $\A$.
Then $\C$ is the completion of $\A$. The domain of $\C$ is $\wp(\At\A)$. The diagonal ${\sf d}_{ij}$ is interpreted
as the set of all $S\in \At\A$ with $a_i=a_j$ for some $\bar{a}\in S$.
The cylindrification ${\sf c}_i$ is interpreted by ${\sf c}_iX=\{S\in \At\A: S\subseteq c_i^{\A}(S')\text { for some } S'\in X\}$, for $X\subseteq \At\A$.
Finally ${\sf p}_{ij}X=\{S\in \At\A: S\subseteq {\sf p}_{ij}^{\A}(S')\text { for some } S'\in X\}.$
Let $\cal D$ be the relativized set algebra with domain $\{\phi^W: \phi\text { an $L_{\infty\omega}^n$ formula }\}$,  unit $W$
and operations defined like those of $\cal A$.

${\C}\cong \D$, via the map $X\mapsto \bigcup X$.

The map is clearly injective. It is surjective, since
$$\phi^W=\bigcup\{\alpha^W: \alpha \text { an $MCA$-formula}, \alpha^W\subseteq \phi^W\}$$
for any $L_{\infty\omega}^n$ formula $\phi$.
Preservation of the Boolean operations and diagonals is clear.
We check cylindrifications. We require that for any $X\subseteq \At\A$,
we have
$\bigcup {\sf c}_i^{\cal C}X={\sf c}_i^{\cal D}(\bigcup X)$ that is
$$\bigcup \{S\in At\A: S\subseteq {\sf c}_i^{\A}S'\text { for some $S'\in X$ }\}=$$
$$\{\bar{a}\in W: \bar{a}\equiv_i \bar{a'} \text { for some } \bar{a'}\in \bigcup X\}.$$
Let $\bar{a}\in S\subseteq c_iS'$, where $S'\in X$. So there is $\bar{a'}\equiv_i \bar{a}$ with $\bar{a'}\in S'$, and so $\bar{a'}\in \bigcup X$.

Conversely, let $\bar{a}\in W$ with $\bar{a}\equiv_i \bar{a'}$ for some $\bar{a'}\in \bigcup X$.
Let $S\in At\A$, $S'\in X$ with $\bar{a}\in S$ and $\bar{a'}\in S'$.
Choose $MCA$ formulas $\alpha$ and $\alpha'$ with $S=\alpha^W$ and $S'=\alpha'^{W}$.
then $\bar{a}\in \alpha^{W}\cap (\exists x_i\alpha')^W$ so $\alpha^W\subseteq (\exists x_i\alpha')^W$, or $S\subseteq c_i^{\A}(S')$.
The required now follows. We leave the checking of substitutions to the reader.

The non representability of the complex algebra in the case of the rainbow algebra is in \cite{Hodkinson},
witness also the argument in the coming theorem \ref{smooth}.

\item{\sf Using the Monk algebras}

To prove item (2) we use Monk's algebras. Any of the two will do just as well:

We show that their atom structure consists of the $n$
basic matrices of a relation algebra.
In more detail, we define a relation algebra atom structure $\alpha(\G)$ of the form
$(\{1'\}\cup (\G\times n), R_{1'}, \breve{R}, R_;)$.
The only identity atom is $1'$. All atoms are self converse,
so $\breve{R}=\{(a, a): a \text { an atom }\}.$
The colour of an atom $(a,i)\in \G\times n$ is $i$. The identity $1'$ has no colour. A triple $(a,b,c)$
of atoms in $\alpha(\G)$ is consistent if
$R;(a,b,c)$ holds. Then the consistent triples are $(a,b,c)$ where

\begin{itemize}

\item one of $a,b,c$ is $1'$ and the other two are equal, or

\item none of $a,b,c$ is $1'$ and they do not all have the same colour, or

\item $a=(a', i), b=(b', i)$ and $c=(c', i)$ for some $i<n$ and
$a',b',c'\in \G$, and there exists at least one graph edge
of $G$ in $\{a', b', c'\}$.

\end{itemize}

$\alpha(\G)$ can be checked to be a relation atom structure. It is exactly the same as that used by Hirsch and Hodkinson in \cite{HHbook}, except
that we use $n$ colours, instead of just $3$, so that it a Monk algebra not a rainbow one. However, some monochromatic triangles
are allowed namely the 'dependent' ones.
This allows the relation algebra to have an $n$ dimensional cylindric basis
and, in fact, the atom structure of of $M(\G)$ is isomorphic (as a cylindric algebra
atom structure) to the atom structure $\Mat_n$ of all $n$-dimensional basic
matrices over the relation algebra atom structure $\alpha(\G)$.

Indeed, for each  $m  \in {\Mat}_n, \,\ \textrm{let} \,\ \alpha_m
= \bigwedge_{i,j<n}  \alpha_{ij}. $ Here $ \alpha_{ij}$ is $x_i =
x_j$ if $ m_{ij} = 1$' and $R(x_i, x_j)$ otherwise, where $R =
m_{ij} \in L$. Then the map $(m \mapsto
\alpha^W_m)_{m \in {\cal M}_n}$ is a well - defined isomorphism of
$n$-dimensional cylindric algebra atom structures.

We show $\Cm\alpha(\G)$ is not in $\sf RRA$.
Hence the full complex cylindric algebra over the set of $n$ by $n$ basic matrices
- which is isomorphic to $\cal C$ is not in $\RCA_n$ for we have a relation algebra
embedding of $\Cm\alpha(\G)$ onto $\Ra\Cm\M_n$.

Assume for contradiction that $\Cm\alpha(\G)\in \RCA_n$
For $Y\subseteq \N$ and $s<n$, set
$$[Y,s]=\{(l,s): l\in Y\}.$$
For $r\in \{0, \ldots N-1\},$ $N\N+r$ denotes the set $\{Nq+r: q\in \N\}.$
Let $$J=\{1', [N\N+r, s]: r<N,  s<n\}.$$
Then $\sum J=1$ in $\Cm\alpha(\G).$
As $J$ is finite, we have for any $x,y\in X$ there is a $P\in J$ with
$(x,y)\in h(P)$.
Since $\Cm\alpha(\G)$ is infinite then $X$ is infinite.
By Ramsey's Theorem, there are distinct
$x_i\in X$ $(i<\omega)$, $J\subseteq \omega\times \omega$ infinite
and $P\in J$ such that $(x_i, x_j)\in h(P)$ for $(i, j)\in J$, $i\neq j$. Then $P\neq 1'$.
Also $(P;P)\cdot P\neq 0$. if $x,y, z\in M$,
$a,b,c\in \Cm\alpha(\G)$, $(x,y)\in h(a)$, $(y, z)\in h(b)$, and
$(x, z)\in h(c)$, then $(a;b)\cdot c\neq 0$.
A non -zero element $a$ of $\Cm\alpha(\G)$ is monochromatic, if $a\leq 1'$,
or $a\leq [\N,s]$ for some $s<n$.
Now  $P$ is monochromatic, it follows from the definition of $\alpha$ that
$(P;P)\cdot P=0$.

This contradiction shows that
$\Cm\alpha(\G)$ is not representable and so $\Cm\Mat_n\alpha(G)$ is non representable

\item {\sf Summarizing}

We have in all cases a labelled graph  defined $M$ as a model of a first order theory in the Monk case
and of $L_{\omega_1, \omega}$ in the rainbow case.
We have also relativized $^nM$ to  $W\subseteq {}^nM$ by deleting assignments
whose edges involve reds, and we defined an algebra containing the term algebra, namely,
$\A$ as the atomic relativized set algebra
with unit $W$. In the case of  Monk's algebra we defined the relation algebra $\sf R$, such that
$\At\A\cong {\sf Mat}_n\At\R$, and in all cases the complex algebra $\C$
of the atom structure is isomorphic to the set algebra
$\{\phi^M: \phi\in L_{\infty, \omega}^n\}$.
Finally, we the $\sf Df$ reduct of this algebra, namely, $\C=\Cm\At\A$, cannot be representable for else this induces a complete representation
of $\Rd_{df}\A$, hence a complete representation of $\A$,
which in turn induces a representation of $\C$, but we have shown that the latter cannot be representable.
\end{enumarab}
\end{proof}
\begin{corollary}Let $\sf L$ be any signature between $\sf Df$ and $\sf PEA$. 
\begin{enumarab}
\item There is an $\sf L$ atom structure that is weakly representable, but not strongly representable.
In particular, there is an $\sf L$ atom structure, that carries two (atomic) algebras, one is representable, the other is not.
\item The class of $\sf L$ representable algebras is not atom-canonical, it is not closed under 
\d\ completions and cannot be axiomatized by Sahlqvitst equations. 
\item The same  holds for the class of representable 
relation algebras
\end{enumarab}

\end{corollary} 
\begin{proof} A special case of theorem \ref{sah} below.
\end{proof}
In contrast, we have:
\begin{theorem} Let $\R$ and $M$ be as above. If $\R$
is finitely generated and ${\sf Th}M$ is $\omega$ categorial or $M$ is ultrahomogeneous, then
$\Tm{\Mat_n\At\R}$ is strongly representable.
\end{theorem}
\begin{proof}
Assume that $M$ is ultrahomogeneous and that $\R$ is finitely generated.
Then $M$ is a countable model in binary relational finite signature,
hence it has quantifier elimination. Since the relativized semantics for $L^n$ formulas
coincide with the classical one, we have that $\A\in \Nr_n\CA_{\omega}$. (Remember that semantics  is perturbed
at formulas in $L_{\infty, \omega}^n$.) Now we show that $\A$ is completely representable, then we build from this complete representation
a representation of the complex algebra.
Let $X=\At\A$. Let $\B\in {\sf Lf}_{\omega}$ be such that $\A=\Nr_n\B$ and $A$ generates $\B$, hence $\sum X=1$ in $\B$.
Let $Y$ be the set of co-atoms, then $Y\subseteq \Nr_n\B$ is a non-principal type. By the usual Orey Henkin omitting
types theorem for first order logic, we get a model, that is a representation that omits $Y$.
That is for every non-zero $b\in \B$, we get a set algebra $\C$ having countable base, and
$f:\B\to \C$ such that $f(a)\neq 0$ and $\bigcap_{y\in Y} f(y)=\emptyset$.
Define $g(x)=\{s\in {}^nU: s^+\in f(x)\},$ where $s^+=s$ on $n$ and fixes all other points in $\omega\sim n$,
then $g$ gives a complete simple representation such that $g(a)\neq 0$.
Taking disjoint unions over the non zero elements in $\A$ as is common practise in algebraic logic, we
get the desired complete representation, call it $f$.

Next we represent the complex algebra $\C$. Recall that $\C$ is the
completion of $\A$. Let $c\in C$, let $M$ be the base of the complete representation of $\A$,
and let $c^*=\{x\in \A: x\leq c\}$. Define $g(c)=\bigcup_{x\in c^*} f(x)$, then clearly $g$ defines
a representation of $\C$ into $\wp(^nM),$ and we are done.
\end{proof}
Recall that we dealt with relativized algebras in our previous proof.
For an algebra $\A$ and a logic $L$, that is an extension of $n$
variable first order logic,
$L(A)$ denotes the signature (taken in this logic) obtained by adding an $m$
relation symbol for every element of $\A$.
The notion of relativized representations is indeed important
in algebraic logic. $M$ is a relativized representation of an abstract algebra $\A$, if there exists $V\subseteq {}^mM$, and an injective
homomorphism $f:\A\to \wp(V)$. We write $M\models 1(\bar{s})$, if $\bar{s}\in V$.

\begin{definition}\label{rel}
In the next definition $m$ denotes the dimension.
\begin{enumarab}
\item  Let $M$ be a relativized representation of a $\CA_m$.
A clique in $M$ is a subset $C$ of $M$ such that $M\models 1(\bar{s})$ for all $\bar{s}\in {}^mC$.
For $n>m$, let $C^{n}(M)=\{\bar{a}\in {}^nM: \text { $\rng(\bar{a})$ is a clique in $M$}\}.$
\item Let $\A\in \CA_m$, and $M$ be a relativized representation of $\A$. $M$ is said to be $n$ square, $n>m$,
if whenever $\bar{s}\in C^n(M)$, $a\in A$, and $M\models {\sf c}_ia(\bar{s})$,
then there is a $t\in C^n(M)$ with $\bar{t}\equiv _i \bar{s}$,
and $M\models a(\bar{t})$.
\item  $M$ is infinitary $n$ flat if  for all $\phi\in L(A)_{\infty, \omega}^n$, for all $\bar{a}\in C^n(M)$, for all $i,j<n$, we have
$$M\models \exists x_i\exists x_j\phi\longleftarrow \exists x_j\exists x_i\phi(\bar{a}).$$
$M$ is just $n$ flat, if the above holds for first order formulas using $n$ variables,
namely, for $\phi\in L(A)^n$.
\item $M$ is said to be $n$ smooth if it is $n$ square, and there is are equivalence relations $E^t$, for $t=1,\ldots n$ on $C^{t}(M)$
such that  $\Theta=\{(\bar{x}\to \bar{y}): (\bar{x}, \bar{y})\in \bigcup_{1\leq t\leq n} E^t\}$
(here   $\bar{x}\to \bar{y}$ is the map $x_i\mapsto y_i$)
is an $n$ back and for system of partial isomorphisms on $M$.
\end{enumarab}
\end{definition}

Note that all of the above are a local form of representation where roughly cylindrifiers have witnesses only
on $<m$ cliques.
Closely related to relativized representations are the notion of hyperbasis. This is only formulated for relation algebras \cite[12.2.2]{HHbook}.
We extend it to the polyadic case, and all its reducts down to $\Sc$s.
\begin{definition}
Let $\A\in \sf PEA_j$. Let $j+1\leq m\leq n\leq k<\omega$, and let $\Lambda$ be a non-empty set.
An $n$ wide $m$ dimensional $\Lambda$ hypernetwork over $\A$ is a map
$N:{}^{\leq n}m\to \Lambda\cup \At\A$ such that
$N(\bar{x})\in \At\A$ if $|\bar{x}|=j$ and $N(\bar{x})\in \Lambda$ if $|\bar{x}|\neq j$,
(so that $j$ edges are labelled by atoms, and other edges with labels from $\Lambda$, 
with the following properties for each $i,k<j$, $\delta\in {}^{\leq n}m$, $\bar{z}\in ^{\leq n-2}m$
and $d\in m$:
\begin{itemize}
\item $N(\delta^i_j)\leq {\sf d}_{ij}$
\item $N(x,x, \bar{z})\leq d_{01}$
\item $N(\delta[i\to d])\leq {\sf c}_iN(\delta)$
\item $N(\delta\circ [i,j])={\sf s}_{[i,j]}N(\delta).$
\item If $\bar{x}, \bar{y}\in {}^{\leq n}m$, $|\bar{x}|=|\bar{y}|$ and $\exists \bar{z}$, such that
$N(x_i,y_i,\bar{z})\leq {\sf d}_{01}$,  for all $i<|\bar{x}|$, then $N(\bar{x})=N(\bar{y})$
\item when $n=m$, then $N$ is called an $n$ dimensional $\Lambda$ hypernetwork.
\end{itemize}
\end{definition}

Then $n$ wide $m$ dimensional {\it hyperbasis} can be defined like the relation algebra case \cite{HHbook}, where the amalgamation property is
the most important (it corresponds to commutativity of cylindrifiers).

For an $n$ wide $m$ dimensional $\Lambda$ hypernetworks, and $\bar{x}\in {}^{<\omega}m)$, we define
$N\equiv_x M$
if $N(\bar{y})=M(\bar{y})$ for all
$\bar{y}\in {}^{\leq n}(m\sim \rng(\bar{x}))$.

In more detail:

\begin{definition} The set of all $n$ wide $m$ dimensional hypernetworks will be denoted by $H_m^n(\A,\Lambda)$.
An $n$ wide $m$ dimensional $\Lambda$, with $j+1\leq m\leq n$
hyperbasis for $\A$ is a set $H\subseteq H_m^n(\A,\lambda)$ with the following properties:
\begin{itemize}
\item For all $a\in \At\A$, there is an $N\in H$ such that $N(0,1,\ldots j)=a$
\item For all $N\in H$ all $\bar{x}\in {}^n\nodes(N)$, for all $i<j$ for all $a\in\At\A$ such that
$N(\bar{x})\leq {\sf c_i}a$, there exists $\bar{y}\equiv_i \bar{x}$ such that $N(\bar{y})=a$
\item For all $M,N\in H$ and $x,y<n$, with $M\equiv_{xy}N$, there is $L\in H$ such that
$M\equiv_xL\equiv_yN$. Here $M\equiv_SN$, means that $M$ and $N$ agree off of $S$.

\item Assume that $j+1\leq m\leq n\leq k$.
For a $k$ wide $n$ dimensional hypernetwork $N$, we let $N|_m^k$ the restriction of the map $N$ to $^{\leq k}m$.
So we restrict the first $m$ nodes but keep all hyperlabels on them.
For $H\subseteq H_n^k(\A,\Lambda)$ we let $H|_k^m=\{N|_m^k: N\in H\}$.

\item When $n=m$, $H_n(\A,\Lambda)$ is called an $n$ dimensional hyperbases.
\end{itemize}
We say that $H$ is symmetric, if whenever $N\in H$ and
$\sigma:m\to m$, then $N\circ\sigma\in H$, the latter, which we denote simply by $N\sigma$ is defined by
$N\sigma(\bar{x})=N(\sigma(x_0),\ldots ,\sigma(x_{n)})$, for $\bar{x}\in {}^n\nodes(N)$.

\end{definition}

When $n=m$, then $H_n^n(\A, \Lambda)$ denoted simply by $H_n(\A, \Lambda)$ is called an $n$ dimensional hyperbasis.
If an addition, $|\Lambda|=1$, then $H_m^n(\A, \Lambda)$ is called an $n$ dimensional cylindric basis. (These were defined
by Maddux, and later generalized by Hirsch and Hodkinson to hyperbasis).

\begin{lemma}\label{step} If  $\A\in \PEA_m$ has an $n$ dimensional hyperbasis
then $\A$ has a relativized $n$ smooth representation.
\end{lemma}

\begin{proof}
We build a relativized representation $M$
in a step- by-step fashion; it will be a labelled hypergraph satisfying the the following properties

\begin{enumarab}

\item Each hyperedge is labelled by an atom of $\A$

\item $M(\bar{x})\leq {\sf d}_{ij}$ iff $x_i=x_j$. (In this case, we say that $M$ is strict).

\item For any clique $\{x_0,\ldots, x_{n-1}\}\subseteq M$,
there is a unique $N\in H$ such that $(x_0,\ldots x_{n-1})$ is labelled by $N$ and we write this as
$M(x_0,\ldots x_{n-1})=N$

\item  If $x_0,\ldots, x_{n-1}\in M$ and $M(\bar{x})=N\in H$,
then for all $i_0\ldots i_{m-1}<n$, $(x_{i_0}\ldots x_{i_{m-1}})$
is a hyperedge and $M(x_{i_0}, x_{i_1}, \ldots  x_{i_{m-1}})=N(i_0,\ldots i_{m-1})\in \At\A$.

\item $M$ is symmetric; it closed under substitutions. That is, if $x_0,\ldots x_{n-1}\in M$
are such that  $M(x_0,\ldots x_{n-1})=N$,
and $\sigma:n\to n$ is any map, then  we have
$M(x_{\sigma(0)}, \ldots x_{\sigma(n-1)})=N\sigma$.

\item If $\bar{x}$ is a clique, $k<n$ and $N\in H$, then $M(\bar{x})\equiv_k N$ if
and only if there is
$y\in M$ such that $M(x_0,\ldots x_{k-1}, y, x_{k+1}, \ldots x_{n-1})=N.$

\item For every $N\in H$, there are $x_0,\ldots,  x_{n-1}\in M$, $M(\bar{x})=N$

\end{enumarab}

We build a chain of hypergraphs $M_t:t<\omega$ their limit (defined in a precise
sense) will as required.  Each $M_t$ will have hyperedges of length $m$ labelled by
atoms of $\A$, the labelling of other hyperedges will be done later.

To avoid confusion we call hyperedges that are to be labelled by atoms atom-hyperedges, these have length
exactly $m$, and those that have $\leq n$ length not equal to $m$ (recall that $3\leq m<n$)
that are labelled by elements in $n$ simply by hyperedges.

We proceed by a step by step manner, where the defects are treated one by one, and they are diffused at the limit obtaining the required
hypergraph, which also consists of
two parts, namely, atom-hyperedges and hyperedges.

This limiting hypergraph will be compete (all atom hypergraphs and hypergraphs will be labelled;
which might not be the case with $M_t$ for $t<\omega$.)
Every atom-hyperedge will indeed be labelled by an atom of $\A$ and hyperedges with length
$\neq  m$ will also be labelled, but  by indices from $n$.

We require inductively that $M_t$ also satisfies:

Any clique in $M_t$ is contained in $\rng(v)$ for some $N\in H$ and some embedding $v:N\in M_t$.
such that $v:n\to \dom(M_t)$
and this embedding satisfies the following two conditions:

(a) if $(v(i_0), \ldots v_(i_{m-1})$ is a an atom  hyperedge of $M_t$, then
$M_t(\bar{v})=N(i_0,\ldots i_{n-1}).$

(b) Whenever $a\in {}^{\leq ^n}n$ with $|a|\neq m$, then $v(a)$ is a
hyperedge of $M_t$ and is labelled by $N(\bar{a})$.

(Note that an embedding might not be injective).
A network is strict if  $\N(\bar{x})\leq {\sf d}_{ij}$, then $x_i=x_j$.
For the base of the induction we construct  $M_0.$ We proceed as follows.
Let $N$ be a hypernetwork and let $S\subseteq n$. $N|S$ is maximal strict if it is strict
and  any for  $T\subseteq n$, such that  $T$ contains $S$, $N|T$ is not strict.

Let $M_0$ be the disjoint union
of all maximal strict labelled hypergraph
$N|S$ where $N\in H$ and $S\subseteq N$
and the atom-hyperedges and hyperedges are induced (the natural way) from $N$.
For all $i<n$, there is a unique $s_i\in S$, and $\bar{z}$, such that
$N(i, s_i, \bar{z})\leq {\sf d}_{01}$,  then set  $v(i)=s_i$, for $i<n$

Now assume inductively that $M_t$ has been  defined for all $t<\omega$.

We now define $M_{t+1}$ such that for every quadruple
$(N, v, k, N')$ where $N, N'\in H$, $k<n$ and $M\equiv_k N'$ and $v:N\to M_t$
is an embedding, then  the restriction $v\upharpoonright n\sim \{k\}$ extends to an embedding
$v'; N'\to M_{t+1}$

\begin{itemize}

\item For each such $(N, v,k, N')$ we add just one new node $\pi$,
and we add the following atom hyperedges $(\pi, v(i_1)\ldots v(i_{m-2})$ for each $i_1,\ldots  i_{m-2}\in n\sim k$
that are pairwise distinct. Such new atom-hyperedges are labelled by $N'(k, i_0,\ldots i_{m-2}) $. This is well defined.
We extend $ v$ by defining $v'(k)=\pi$.

\item We add a new hyperedge $v'(\bar{a})$ for every $\bar{a}\in {}^nn$ of length $\neq m$,
with $k\in \rng (\bar{a})$ labelled by $N'(\bar{a})$.

\end{itemize}
Then $M_{t+1}$ will be $M_t$ with its old labels, atom hyperedges,  labelled hyperedges
together with the new ones define as above.

It is easy to check that the inductive hypothesis  are preserved.

Now define a labelled hypergraph as follows:
$\nodes(M)=\bigcup \nodes(M_t)$,
for any $\bar{x}$ that is an atom-hyperedge; then it is a one in some $M_t$
and its label is defined in $M$ by $M_t(\bar{x})$

The hyperedges are $n$ tuples $(x_0,\ldots x_{n-1})$.
For each such tuple, we let $t<\omega$,
such that $\{x_0\ldots x_{n-1}\} \subseteq  M_t$,
and we set $M(x_0,\ldots x_{n-1})$   to be the unique
$N\in H$ such that there is an embedding
$v:N\to M$ with  $\bar{x}\subseteq \rng(v).$
This can be easily checked to be well defined. We check existence and uniqueness. The latter is
clear from the definition of embedding. Note that there is an $N\in H$
and an embedding $v:N\to M_t$ with $(x_0, \ldots x_{n-1})\subseteq \rng(v)$.
So take $\tau:n\to n$, such that $x_i=v(\tau(i))$ for each $i<n$.
As $H$ is symmetric, $N\tau\in H$, and clearly $v\circ \tau: N\tau\to M_t$
is also an embedding. But $x_i=v\circ \tau(u)$ for each $i<n$,
then let $M(x_0,\ldots x_{n-1})=N\tau$

Now we show that $M$ is an $n$ square relativized representation.
Let $L(A)$ be the signature obtained by adding an $n$ ary relation symbol for each element
of $\A$.
Define $M\models r(\bar{x})$ if $\bar{x}$ is an atom  hyperedge and $M(\bar{x})\leq r$

Now let $\bar{x}\in C^n(M),$ $k<m$ and $M\models r(\bar{x})$. We require that there exists $y\in C^n(M)$
$y\equiv_k x$ and $M\models r(\bar{y})$. Take $i_0, \ldots i_{m-1}<n$ different from $k$.
Now $M(\bar{x})=N$ for some $N\in H$.
By properties of hyperbasis, there is $P\in H$  with $P\equiv_k N$.
Hence $P(i_0, \ldots i_{k-1}, k, i_{k+1}\ldots )\leq r$. But by properties of $M$, there  is
an $n$ tuple $\bar{y}\equiv k\bar{x}$ such that $M(y)=P$.
Then $\bar{y}\in C^n(M)$ is as required.

Finally, for $n$ smoothness we need to define equivalence relations
$E^m$ on $C^m(M)$, satisfying the definition, for $0\leq m\leq n$.
For each such $m$ and $\bar{a}\in C^m(M)$,
define  $a^*=(a_0, a_1,\ldots a_{m-1}, a_0,\ldots a_0)\in C^n(M)$, where we have $n$
copies of $a_0$ after $a_0,a_1,\ldots a_{m-1}$. Now for $\bar{a}$, $\bar{b}\in C^n(M)$, set
$$E^m(\bar{a},\bar{b})\longleftrightarrow M(a^*)=M(b^*).$$
This is as required.
\end{proof}

\begin{theorem}\label{smooth} For every $n\geq 3$ there exists a
countable atomic $\sf PEA_n$, such that the $\CA$ reduct of
its completion  does not have an $n+4$ smooth representation, in particular, it is not  representable. Furthermore, its
${\sf Df}$ reduct is not representable.
\end{theorem}
\begin{proof}
Here we closely follow \cite{Hodkinson}; but our reds and greens are finite,
so we obtain a stronger result. We take $|{\sf G}|=n+2$, and ${\sf R}=n+1$.
Let $L^+$ be the rainbow signature consisting of the binary
relation symbols $\g_i :i<n-1, \g_0^i: \ i< n+2, \w, \w_i: i <n-2, \r_{jk}^i  (i<n+1, j<k<n)$
and the $(n-1)$ ary-relation symbols
$\y_S: S\subseteq n+2)$, together with a shade of red $\rho$ that is outside the rainbow
signature but is a binary relation in the sense that it can label edges
of coloured graphs. Let $\GG$ be the class of corresponding rainbow coloured graphs.
By the same methods as above, there is a countable model  $M\in \GG$ with the following
property:\\
$\bullet$ If $\triangle \subseteq \triangle' \in \GG$, $|\triangle'|
\leq n$, and $\theta : \triangle \rightarrow M$ is an embedding,
then $\theta$ extends to an embedding $\theta' : \triangle'
\rightarrow M$. Now let $W = \{ \bar{a} \in {}^n M : M \models ( \bigwedge_{i < j < n,
l < n} \neg \rho(x_i, x_j))(\bar{a}) \}$. Then $\A$ with universe $\{\phi^W: \phi\in L_n\}$ and operations defined the usual way,
is representable, and its
completion, the complex algebra over the  above rainbow atom structure, $\C$ has universe $\{\phi^W: \phi\in L_{\infty, \omega}^n\}$.

We show that $\C$ is as desired. Assume, for contradiction, that $g:\C\to \wp(V)$ induces a relativized $n+4$ flat representation.
Then $V\subseteq {}^nN$ and we can assume that
$g$ is injective because $\C$ is simple. First there are $b_0,\ldots b_{n-1}\in N$ such $\bar{b}\in
h(\y_{n+2}(x_0,\ldots x_{n-1}))^W$, cf \cite[lemma 5.7]{Hodkinson}.
This tuple will be the base of finitely many cones, that will be used to force an inconsistent triple of reds.
This is because $\y_{n+2}(\bar{x})^W\neq \emptyset$.  For any $t<n+3$, there is a $c_t\in N$, such
that $\bar{b}_t=(b_0,\ldots b_{n-2},\ldots c_t)$ lies in $h(\g_0^t(x_0, x_{n-1})^W$ and in $h(\g_i(x_i, x_{n-1})^W$ for each $i$ with
$1\leq i\leq n-2$, cf \cite[lemma 5.8]{Hodkinson}.  The $c_t$'s are the apexes of the cones with base $\y_{n+2}$.
Take the formula
$$\phi_t=\y_{n+2}(x_0,\ldots ,x_{n-2})\to \exists x_{n-1}(\g_0^t(x_0, x_{n-1}))\land \bigwedge_{1\leq i\leq n-2}\g_i(x_i, x_{n-1})),$$
then $\phi_t^{W}=W$. Pick $c_t$ and $\bar{b_t}$ as above, and define for each $s<t<n+3,$ $\bar{c_{st}}$ to be
$(c_s,b_1,\ldots  b_{n-2}, c_t)\in {}^nN.$
Then $\bar{c}_{st}\notin h((x_0,\ldots  x_{n-1})^W$. Let $\mu$ be the formula
$$x_0=x_{n-1}\lor \w_0(x_0, x_{n-1})\lor \bigvee \g(x_0, x_{n-1}),$$
the latter formula is a first order formula consisting of the disjunction of  the (finitely many ) greens.
For $j<k<n$, let $R_{jk}$ be the $L_{\infty\omega}^n$-formula $\bigvee_{i<\omega}\r_{jk}^i(x_0, x_{n-1})$.
Then
$\bar{c}_{st}\notin h(\mu^W)$, now for each $s<t< n+3$, there are $j<k<n$ with $c_{st}\in h(R_{jk})^W.$
By the pigeon- hole principle, there are $s<t< n+3$ and $j<k<n$
with $\bar{c}_{0s}, \bar{c}_{0t}\in h(R_{jk}^W)$. We have also $\bar{c}_{st}\in h(R_{j',k'}^W)$
for some $j', k'$ then the sequence $(c_0, c_s,\ldots,  b_2,\ldots b_{n-2},\ldots, c_t)\in h(\chi^W)$
where
$$\chi=(\exists_1R_{jk})(\land (\exists x_{n-1}(x_{n-1}=x_1\land \exists x_1 R_{jk})\land (\exists x_0(x_0=x_1)\land \exists x_1R_{j'k})),$$
so $\chi^W\neq \emptyset$. Let $\bar{a}\in \chi ^W$. Then $M\models _W  R_{jk}(a_0,a_{n-1})\land R_{jk}(a_0,a_1)\land R_{j'k'}(a_1, a_{n-1})$.
Hence there are
$i$, $i'$ and $i''<\omega$ such that
$$M\models _W\r_{jk}^{i}(a_0,a_{n-1})\land \r_{jk}^{i'}(a_0,a_1)\land \r_{j'k'}^{i''}(a_1, a_{n-1}),$$
cf. \cite[lemma 5.12]{Hodkinson}.
But this triangle is inconsistent. Note that this inconsistent red was forced by an $n+4$ red clique
labelling edges between apexes of the same cone, with base labelled by $\y_{n+2}$.

For the last part, if its $\sf Df$ reduct is representable, then $\Rd_{df}\A$ will be completely representable,
hence $\A$ itself will be completely representable because it is generated by elements whose dimension set
$<n$, which is a contradiction.
\end{proof}
\begin{corollary}\label{can} We have $\C\notin S\Nr_n\CA_{n+4}$. In particular for any $k\geq 4$,  the variety
$S\Nr_n\CA_{n+k}$ is not atom canonical.
\end{corollary}
\begin{proof}
The first part. Assume, for contradiction,  that $\C\in S\Nr_n\CA_{n+4}$; let $\C\subseteq \Nr_n\D$.
Then $\C^+\in S_c\Nr_n\D^+$, and $\D^+$ is of course atomic. We show that $\C^+$ has an $n+4$ dimensional hyperbasis,
then that it has an $n+4$ smooth representation, which contradicts the previous
theorem. Here again hyperbasis are defined for cylindric algebras by a straightforward lifting from the relation algebra case.

First note that for every $n\leq l\leq m$, $\Nr_l\D^+$ is atomic.
Indeed, if $x$ is an atom in $\D^+$, and and $n\leq l<m$,
then ${\sf c}_{l}\ldots {\sf c}_{m-l+1}x$ is an atom in $\Nr_l\D^+$,
so if $c\neq 0$ in the latter, then there exists $0\neq a\in \At\D^+$,
such that $a\leq c$, and so ${\sf c}_{l}\ldots {\sf c}_{m-1+1}a\leq c_{l}\ldots {\sf c}_{m-1+1}c=c$.

Let $\Lambda=\bigcup_{k<n+3}\At\Nr_k\D^+$, and let $\lambda\in \Lambda$.
In this proof we follow closely section 13.4 in \cite{HHbook}. The details are omitted because they are identical
to the corresponding ones in op.cit.
For each atom $x$ of $\D$, define $N_x$, easily checked to be an $m$ dimensional   $\Lambda$ hypernetwork, as follows.
Let $\bar{a}\in {}^{n+4}n+4$ Then if $|a|=n$,  $N_x(a)$ is the unique atom $r\in \At\D$ such that $x\leq {\sf s}_{\bar{a}}r$.
Here substitutions are defined as above.
If $n\neq |\bar{a}| <n+3$, $N_x(\bar{a})$ the unique atom $r\in \Nr_{|a|}\D$ such that $x\leq s_{\bar{a}}r.$
$\Nr_{|a|}\D$ is easily checked to be atomic, so this is well defined.

Otherwise, put  $N_x(a)=\lambda$.
Then $N_x$ as an $n+4$ dimensional $\Lambda$ hyper-network, for each such chosen $x$ and $\{N_x: x\in \At\C\}$
is an $n+4$ dimensional $\Lambda$ hyperbasis.
Then viewing those as a saturated set of mosaics, one can can construct a complete
$n+4$ smooth representation of $M$ of $\C$, see \cite[proposition 13.37]{HHbook}. But this contradicts the previous theorem
\ref{smooth}
\end{proof}
\subsubsection{A different view, blowing up and blurring a finite rainbow cylindric algebra}

In the following we denote the rainbow algebra $R(\Gamma)$ defined in \cite{HHbook2} by $\CA_{\sf G, \sf R}$ where $\sf R=\Gamma$
is the graph of reds, which will be a complete irreflexive graph, and
$\sf G$ the indices greens with subscript $0$.

The idea used here is a typical instance of a blow up and blur construction. Let ${\sf L}\subseteq \RCA_n$ be closed under forming subalgebras.
Start with a finite (atomic) algebra $\C$ such that $\C\notin \sf L$.
Then blow up and blur its atom structure buy splitting each of its atoms into infinitely
many. This way we get  a new infinite atom structure $\At$ which has a finite
set of blurs (non principal ultrafilters).
These blurs play a double role.
They blur $\A$ at this level, so
it does not embed in $\Tm\At\A$,
and viewed as colours
they are also
used to represent
$\Tm\At$.

But $\C$ is still there on the global level, meaning that it embeds into $\Cm\At$ by sending each atom to
the infinite disjunct of
its copies, the latter is complete, so these joins exist. (They {\it do not exist} in the term algebra, for otherwise it would also
be non representable.)

This implies that  $\Cm\At$ is also outside $\sf L$ because $\C\notin \sf L$,
and $\sf L$ is closed under forming subalgebras.

Let $\At$ be the rainbow atom structure in \cite{Hodkinson} except that we have $n+2$ greens and
$n+1$ reds, that is the rainbow atom structure dealt with in \ref{smooth}.
The rainbow signature now consists of $\g_i: i<n-1$, $\g_0^i\i\in n+2$, $\r_{kl}^t: k,l\in n+1$, $t\in \omega$,
binary relations and $\y_S$ $S\subseteq \Z$,
$S$ finite and a shade of red $\rho$; the latter is outside the rainbow signature,
but it labels coloured graphs during the game, and in fact \pe\ can win the $\omega$ rounded game
and build the $n$ homogeneous model $M$ by using $\rho$ when
she is forced a red.

Then $\Tm\At$ is representable; this can be proved exactly as in \cite{Hodkinson}.
The atoms of $\Tm\At$ are coloured graphs whose edges are not labelled by
the one shade of red  $\rho$; it can also be viewed as a set algebra based on $M$
by relativizing semantics discarding assignments whose edges are labelled
by $\rho$. A coloured graph (an atom) in $\CA_{n+2, n+1}$
is one such that at least one of its edges is labelled red.
Now $\CA_{n+2, n+1}$ embeds into $\Cm\At\A$,
by taking every red graph to the join of its copies, which exists because $\Cm\At$ is complete
(these joins do not exist in the (not complete) term algebra; only joins of finite or cofinitely many reds do, hence it serves non representability.)
A copy of a red graph is one that is isomorphic to this graph, modulo removing superscripts of reds.
Another way to do this is to take every coloured graph to the interpretation of an infinite disjunct of the $MCA$ formulas
(as defined in \cite{Hodkinson}), and to be dealt with below; such formulas define coloured graphs whose edges are not labelled by the shade of red,
hence the atoms, corresponding
to its copies, in the relativized semantics; this defines an embedding,  because $\Cm\At$ is isomorphic to
the set algebra based on the same relativized semantics
using $L_{\infty,\omega}^n$ formulas in the rainbow signature. Here again $M$ is the $n$ homogeneous model constructed
in the new rainbow signature, though the construction is the
same.
But \pa\ can win a certain finite rounded game on $\CA_{n+1, n+2},$ hence it is
outside $S\Nr_n\CA_{n+4}$ and so is $\Cm\At$,  because the former is embeddable in the latter
and $S\Nr_n\CA_{n+4}$ is a variety; in particular, it is closed
under forming subalgebras.

For the definition of $n+k$ complex blur the reader is referred to \cite[definition 3.1]{ANT};
this involves a set $J$ of blurs and a ternary relation
$E$.

\begin{theorem}\label{blurs} Let $k\geq 1$. Assume that there exists a finite relation algebra $\R\notin S\Ra\CA_{n+k+1}$
that has  $n+k$  complex blur  $(J, E)$. Let $\At$ be the infinite atom structure obtained by blowing up and blurring $\R$, in the sense of
\cite[p.72]{ANT}.
Then $\Mat_{n+k}\At$ is an $n+k$ dimensional cylindric basis. Furthermore there exists representable algebras $\C_n$ and
$\C_{n+k}$ such that  $\Tm\Mat_n\At\subseteq \C_n$ and $\Tm\Mat_{n+k}\At\subseteq \C_{n+k}$, $\C_n=\Nr_n\C_{n+k}$ and finally
and $\Cm\At\notin S\Nr_n\CA_{n+k+1}$.
\end{theorem}
\begin{proof} Exactly like the proof in \cite{ANT} by replacing the Maddux algebra defined on p.84 denoted by $\M$ by $\R.$
\end{proof}
We note that $k$ cannot be equal to $0$,
because Andr\'eka provided a Sahlqvist axiomatization of $S\Nr_n\CA_{n+1}$, for any finite $n$,
hence the latter is necessarily
atom canonical.

\begin{theorem}\label{neat} Any class $\K$, such that $\K$ contains
the class of completely representable algebras and is contained in $S_c\Nr_n\CA_{n+3}$ is not
elementary.
\end{theorem}
\begin{proof} Let $G_k$ be the usual atomic game played on networks with $k$ rounds \cite[definition 3.3.2]{HHbook2}.
Let $\A$ be the rainbow algebra $\CA_{\Z,\N}$, then \pe\ has a \ws\ in $G_k$ for all finite $k\geq n$
hence it is elementary equivalent to a
countable completely representable algebra $\B$ \cite{HH}.
Let $F^m$ be the $\omega$ rounded atomic game
except that \pa\ s moves are limited to $m$ pebbles.
Then it can be shown that if $\A\in S_c\Nr_n\CA_m$, then \pe\ has a \ws\ in $F^m$, as follows: 
In the initial round \pa\ plays a graph $\Gamma$ with nodes $0,1,\ldots n-1$ such that $\Gamma(i,j)=\w$ for $i<j<n-1$
and $\Gamma(i, n-1)=\g_i$
$(i=1, \ldots n-2)$, $\Gamma(0,n-1)=\g_0^0$ and $\Gamma(0,1\ldots n-2)=\y_{B}$.
In the following move \pa\ chooses the face $(0,\ldots n-2)$ and demands a node $n$
with $\Gamma_2(i,n)=\g_i$ $(i=1,\ldots n-2)$, and $\Gamma_2(0,n)=\g_0^{-1}.$
\pe\ must choose a label for the edge $(n+1,n)$ of $\Gamma_2$. It must be a red atom $r_{mn}$. Since $-1<0$ we have $m<n$.
In the next move \pa\ plays the face $(0, \ldots n-2)$ and demands a node $n+1$, with $\Gamma_3(i,n)=\g_i$ $(i=1,\ldots n-2)$,
such that  $\Gamma_3(0,n+2)=\g_0^{-2}$.
Then $\Gamma_3(n+1,n)$ $\Gamma_3(n+1,n-1)$ both being red, the indices must match.
$\Gamma_3(n+1,n)=r_{ln}$ and $\Gamma_3(n+1, n-1)=r_{lm}$ with $l<m$.
In the next round \pa\ plays $(0,1\ldots n-2)$ and reuses the node $2$ such that $\Gamma_4(0,2)=\g_0^{-3}$.
This time we have $\Gamma_4(n,n-1)=\r_{jl}$ for some $j<l\in \N$.
Continuing in this manner leads to a decreasing sequence in $\N$.

But it can also be shown that \pa\ has a \ws\ in $F^{n+3}$ \cite[theorem 33, lemma 41]{r}.
It follows that $\A\notin S_c\Nr_n\CA_{n+3}$, but it is elementary equivalent
to a countable completely representable algebra. Indeed, using ultrapowers and an elementary chain argument,
we obtain $\B$ such  $\A\equiv \B$ \cite[lemma 44]{r},
and \pe\ has a \ws\ in $G_{\omega}$, so by \cite[theorem 3.3.3]{HHbook2}, $\B$ is completely representable.
So if $\K$ is as above, then $\A\notin \K$ but $\B$
is in $\K$, since $\A\equiv \B$, it readily follows that $\K$ is not
elementary.
\end{proof}

\subsection{Omitting types in clique guarded semantics, and in other contexts}

We now formulate, and prove,  three (negative)
new omitting types theorems, that are consequences of the algebraic results formulated in theorems \ref{can}, \ref{neat} and
\ref{blurs}

Let $T$ be a countable first order theory and $\Gamma$ be a type that is realized in every model of $\Gamma$.
Then the usual Orey Henkin therorem
tells us that this type is necessarily principal, that is, it is isolated by a formula $\phi$.
We call such a $\phi$ an $m$ witness, if $\phi$ is built up of $m$ variables.

\begin{theorem}\label{OTT}
\begin{enumarab}
\item There is a countable $L_n$ theory $T$ a type $\Gamma$,
realized in every smooth $n+4$ model but has no witness.

\item Assume the hypothesis of theorem \ref{blurs}.
Then there is a countable theory $T$, a type realized in every $n+k+1$
smooth model, but there is no $n+k$ witness.

\item There is a countable $L_n$ theory $T$ and a type $\Gamma$ such that
$\Gamma$ is realized in every smooth $n+3$ model, but does not have a witness.

\end{enumarab}
\end{theorem}

\begin{proof}
\begin{enumarab}
\item Let $\A$ be as in corollary \ref{can}, and let $\Gamma$ be the set of atoms. Then we claim
that $\Gamma$ is realized in all $n+4$
models. Towards proving this claim consider a model $\M$ of $T$. If  $\Gamma$ is not realized in $\M$,
then this gives an $n+4$ complete representation of $\A=\Fm_T$,
which is impossible,  because $\Cm\At\A$ is not in $S\Nr_n\CA_{n+4}$.

Assume that $\phi$ is witness.
Then $\A$ is simple, and so we can assume
without loss of generality, that it is set algebra with a countable
base. Let $\M=(M,R)$  be the corresponding model to this set algebra in the sense of \cite{tarski} sec 4.3.
Then $\M\models T$ and $\phi^{\M}\in \A$.
But $T\models \exists x\phi$, hence $\phi^{\M}\neq 0,$
from which it follows that  $\phi^{\M}$ must intersect an atom $\alpha\in \A$ (recall that the latter is atomic).
Let $\psi$ be the formula, such that $\psi^{\M}=\alpha$. Then it cannot
be the case that
that $T\models \phi\to \neg \psi$,
hence $\phi$ is not a $k$ witness,
and we are done.

\item The same argument exactly using instead the statement in \ref{blurs}

\item No we use  \ref{neat}. Let $\A$ be the term algebra of $\CA_{\Z,\N}$.
Then $\A$ is an  atomic countable representable algebra that is not
in $S_c\Nr_n\CA_{n+3}$ (for the same reasons as above, namely,  \pe\ still can win all finite rounded games,
while \pa\ can win the game $F^{n+3}$, because $\A$ and the term algebra have the same
atom structure).

Assume that $\A=\Fm_T$, and let $\Gamma=\{\phi: \neg \phi_T \text { is an atom }\}$.
Then $\Gamma$ is a non-principal type, because $\A$ is atomic, but it has no $n+3$ flat representation
omitting $\Gamma$,  for such a representation would necessarily yield a complete $n+3$ complete representation of $\A,$ which in turn
implies that it would be in
$S_c\Nr_n\CA_{n+3},$ and we know that this is not the case.
\end{enumarab}
\end{proof}

An unpublished result of Andr\'eka and N\'emeti shows that the omitting types theorem fails for $L_2$ though
Vaught's theorem on existence of atomic models for atomic theories hold, \cite{Sayed}, \cite{HHbook}, \cite{Vaught}
In the next example we show that even Vaught's theorem, hence $OTT$, fails when we consider logics without equality
reflected algebraically by many reducts of polyadic algebras.

We first start with algebras that are cylindrifier free reducts of polyadic algebras. In
this case set algebras are defined exactly like polyadic set algebras by discarding
cylindrifiers. Such algebras are expansions of Pinter's algebras studied by S\'agi,
and explicitly mentioned by Hodkinson \cite{AU}
in the context of searching for algebras, where atomicity coincides with complete representability.

Assem showed that for any ordinal $\alpha>1$, and any infinite cardinal $\kappa$,
there is an atomic set algebra (having as extra Boolean operations only finite substitutions)
with $|A|=\kappa$,  that is not  completely representable.
In particular, $\A$ can be countable, and so the omitting types theorem, and for that matter Vaught's theorem fail.
This works for all dimensions, except that in
the infinite dimensional case, semantics is relativized to weak set algebras.
Do we have an analogous result, concerning failure of the omitting types theorem for fragments of $L_n$ without equality, but with quantifiers.
The answer is yes.

This theorem holds for Pinter's algebras and polyadic algebras, let $\K$ denote either class.
It suffices to show that there exists $\B\in {\sf RK}_n\cap \Nr_n\K_{n+k}$ that is not completely completely representable.
But this is not hard to show. Let $\A$ be the cylindric algebra of dimension $n\geq 3$, $n$ finite,
provided by theorem 1.1 in \cite{ANT}.
Then first we can expand $\A$ to a polyadic equality
algebra because it is a subalgebra of the complex algebra based on the atom structure of basic matrices.
This new algebra will also be in ${\sf RPEA}_n\cap \Nr_n{\sf PEA}_{n+k}$. Its reduct, whether the polyadic or the Pinter, will be as desired.

Indeed consider the $\sf PA$ case, with the other case completely analogous,
this follows from the fact that $\Nr_n\K_n\subseteq \Rd\Nr_n{\sf PEA_{n+k}}=\Nr_n\Rd{\sf PEA}_{n+k}\subseteq \Nr_n{\sf PA}_{n+k}$,
and that $\A$ is completely representable if and only if its diagonal free reduct is.
(This is proved by Hodkinson in \cite{AU}, the proof depends essentially on the fact that algebras considered
are binary generated).

Now what if we only have cylindrifiers, that is dealing with ${\sf Df}_n$, $n\geq 3$.
Let $\A$ be the cylindric algebra as in the previous paragraph. Assume that
there a type $\Gamma$,
that is realized in every representation of $\A$ but has no witness using extra $k$ variables. Let $\B=\Rd_{df}\A$.

Let $f:\B\to \C$ be a diagonal free representation of $\B$. 
The point is that though $\Gamma$ is realized in every {\it cylindric} representation of $\A$,
there might be a representation of its diagonal free reduct that omits $\Gamma$,
these are more, because we do not require preservation of the diagonal
elements. This case definitely needs further research, and we are tempted to think that it is not easy.

\subsection{Omitting types for finite first order expansions of $L_n$}

Such formalisms were studied in \cite{Biro} and \cite{basim}.
First we recall what we mean by first order definable operation on the algebra level. These will be logical connectives
in the corresponding logic, thus expanding $L_n$.

\begin{definition}
Let $\Lambda$ be a first order language with countably many
relation symbols, $R_0, \ldots R_i,\ldots : i\in \omega$
each of arity $n$.
Let $\Cs_{n,t}$ denote the following class of similarity type $t$:
\begin{enumroman}
\item $t$ is an expansion of $t_{\CA_n}.$
\item  $S\Rd_{ca}\Cs_{n,t}=\Cs_n.$ In particular, every algebra in $\Cs_{n,t}$ is a boolean
field of sets with unit $^nU$ say,
that is closed under cylindrifications and contains diagonal elements.
\item For any $m$-ary operation $f$ of $t$, there exists a first order formula
$\phi$ with free variables among $\{x_0,\ldots, x_n\}$
and having exactly $m,$ $n$-ary relation symbols
$R_0, \ldots R_{m-1}$ such that,
for any set algebra ${\A}\in \Cs_{n,t}$
with base $U$, and $X_0, \ldots X_{m-1}\in {\A}$,
we have:
$$\langle a_0,\ldots a_{n-1}\rangle\in f(X_0,\ldots X_{m-1})$$
if and only if
$${\cal M}=\langle U, X_0,\ldots X_{n-1}\rangle\models \phi[a_0,\ldots a_{n-1}].$$
Here $\cal M$ is the first order structure in which for each $i\leq m$,
$R_i$ is interpreted as $X_i,$ and $\models$ is the usual satisfiability relation.
Note that cylindrifications and diagonal elements are so definable.
Indeed for $i,j<n$,  $\exists x_iR_0(x_0\ldots x_{n-1})$
defines $\sf c_i$ and $x_i=x_j$ defines $\d_{ij}.$
\item With $f$ and $\phi$ as above,
$f$ is said to be a first order definable operation with $\phi$ defining $f$,
or simply a first order definable
operation, and $\Cs_{n,t}$ is said to be a first order definable
expansion of $\Cs_n.$
\item $\RCA_{n,t}$ denotes the class $SP\Cs_{n,t}$, i.e. the class of all subdirect products
of algebras
in $\Cs_{n,t}.$ We also refer to
$\RCA_{n,t}$ as a first order definable expansion of $\RCA_n.$
\end{enumroman}
\end{definition}

\begin{theorem} Let $t$ be a finite expansion of the $\CA$ type.
Then there are atomic algebras in $\RCA_{n,t}$ that are not completely representable
\end{theorem}
\begin{proof} Let $k\in \omega$ such that all first order definable operations are built up using at most $n+k$ variables.
This $k$ exists, because we have only finitely many of those operations. Let
$\A\in \RCA_n\cap \Nr_n\CA_{n+k}$ be an atomic cylindric algebra that
is not completely representable. Exists by \cite{ANT}
Then $\A$ is closed under all the first order definable operations. But the cylindric reduct of $\A$, is not completely representable, then
$\A$ itself is not.
\end{proof}

$\RCA_{n,t}$ is a variety, this can be proved exactly like the $\RCA_n$ case, and in fact it is a
discriminator variety with discriminator term $c_0\ldots c_{n-1}$.
(So it is enough to show that it is closed under ultraproducts.)
A result of Biro says that is not finitely axiomatizable.

The logic corresponding to $\RCA_{n,t}$ has $n$ variables, and it has the usual first order connectives;
(here we view cylindrifiers as unary connectives) and it has one connective for each first order definable operation.
In atomic formulas the variables only appear in their natural order; so that they are {\it restricted}, in the sense of \cite{HMT1} sec 4.3.

Semantics are defined as follows: For simplicity we assume that we have an at most  countable collection of $n$-ary relation symbols
$\{R_i: i\in \omega\}$.
(We will be considering countable languages anyway, to violate the omitting types theorem).

The inductive definition for the first order (usual) connectives is the usual.
Now given a model $\M=(M, R_i)_{i\in I}$ and a formula $\psi$ of $L_n^+$; with a corresponding connective
$f_{\phi}$ which we assume is unary to simplify
matters, and $s\in {}^nM$, then for any formula $\psi$:
$$\M\models f_{\phi}(\psi)[s]\Longleftrightarrow (\M, \phi^{\M})\models \psi[s].$$

$_rL_n$ is the logic corresponding to $\CA_n$, which has only restricted formulas,
and $L_n$ is that corresponding to $\RPEA_n$; the latter of course is a first order extension of the former
because the substitutions corresponding to transpositions
are first order definable.

\begin{corollary}\label{first}
No first order finite extension of $L_n$  enjoys an omitting types theorem. This holds for languages with just one binary relation.
\end{corollary}

\begin{proof} Let $\A$ be as above, and let $\Fm_T$ be the corresponding Tarski-Lindenbaum algebra.
We assume that $\A=\Fm_T$ (not just isomorphic) Let $X=\At\A$ and let $\Gamma=\{\phi: \phi/\equiv\in X\}.$
Then $\Gamma$ cannot be omitted.
\end{proof}

\subsection{Omitting types and complete representations for the multi modal logic of substitutions}

This section is based on joint work with Assem \cite{Assem}.
We are certain that the ordinary omitting types theorem fails
for $L_n$ without equality (that is for ${\sf Df}_n$) for $n\geq 3$. One way, among many other, is to construct a
representable countable atomic algebra $\A\in {\sf RDf}_n$, that is not completely representable.
The diagonal free reduct of the cylindric algebra constructed in \cite{ANT} is such.

Now what about ${\sf Df_2}$? We do not know.
But if we have only {\it one} replacement then it fails.
For higher dimensions, the result follows from the following example from
\cite{AGMNS}.

We give the a sketch of the proof, the interested reader can work out the details himself
or either directly consult \cite{AGMNS}.

\begin{example}

Let $\B$ be an atomless Boolean set algebra with unit $U$, that has the following property:
For any distinct $u,v\in U$, there is $X\in B$ such that $u\in X$ and $v\in {}\sim X$.
For example $\B$ can be taken to be the Stone representation of some atomless Boolean algebra.
The cardinality of our constructed algebra will be the same as $|B|$.
Let $R=\{X\times Y: X,Y\in \B\}$
and
$A=\{\bigcup S: S\subseteq R: |S|<\omega\}.$
Then $|R|=|A|=|B|$ and
$\A$ is a subalgebra of $\wp(^2U)$.
Also the only subset of $D_{01}$ in $\A$ is the empty set.
Let $S=\{X\times \sim X: X\in B\}.$
$\bigcup S={}\sim D_{01}.$
and $\sum{}^{\A}S=U\times U.$
But
$S_0^1(X\times \sim X)=(X\cap \sim X)\times U=\emptyset.$
for every $X\in B$.
Thus $S_0^1(\sum S)=U\times U$
and
$\sum \{S_{0}^1(Z): Z\in S\}=\emptyset.$

For $n>2$, one takes $R=\{X_1\times\ldots\times X_n: X_i\in \B\}$ and the definition of $\A$ is the same. Then,
in this case, one takes $S$ to be
$X\times \sim X\times U\times\ldots\times U$
such that $X\in B$. The proof survives verbatim.
By taking $\B$ to be countable, then $\A$ can be countable, and so it violates the omitting types theorem.
\end{example}

\begin{example} We consider a very simple case, when we have only transpositions. In this case omitting types theorems
holds for countable languages and
atomic theories have atomic models. Here all substitutions corresponding to bijective maps are definable.
This class is defined by translating a finite presentation of $S_n$, the symmetric group on $n$ to equations,
and postulating in addition that
the substitution operators are Boolean endomorphisms. In this case, given an abstract algebra $\A$ satisfying these equations and
$a\in A$, non zero, and $F$ {\it any} Boolean ultrafilter containing $a$,
then the map $f:\A\to \wp(S_n)$ defined by $\{\tau\in S_n: s_{\tau}a\in F\}$
defines a Boolean endomorphism such that $f(a)\neq 0$.

\begin{enumarab}

\item Now we show that the omitting types theorem holds. We use a fairly standard Baire category argument.
Each $\eta\in S_n$  is a composition of transpositions, so that ${\sf s}_{\eta}$, a composition of complete endomorphisms,
is itself complete. Therefore $\prod {\sf s}_{\eta}X=0$ for all $\eta\in S_n$.
Then for all $\eta\in S_n$, $B_{\eta}=\bigcap_{x\in X} N_{s_{\eta}}x$ is nowhere dense in the Stone topology
and $B=\bigcup_{\eta\in S_n} B_{\eta}$ is of the first category (in fact, $B$
is also nowhere dense, because it is only a finite union of nowhere dense sets).

Let $F$ be an ultrafilter that contains $a$ and is outside $B$. This ultrafilter exists by the celebrated Baire category theorem,
 since the complement of $B$ is dense. (Stone spaces are compact and Hausdorff). Then for all $\eta\in S_n$,
there exists $x\in X$ such that ${s}_{\tau}x\notin F$. Let $h:\A\to \wp(S_n)$ be the usual
representation function; $h(x)=\{\eta\in S_n: { s}_{\eta}x\in F\}$.
Then clearly $\bigcap_{x\in X} h(x)=\emptyset.$

\item Now  further, with no restriction on cardinalities, every atomic algebra is completely
representable. Indeed, let $\B$ be an atomic transposition algebra, let $X$ be the set of atoms, and
let $c\in \B$ be non-zero. Let $S$ be the Stone space of $\B$, whose underlying set consists of all Boolean ultrafilters of
$\B$. Let $X^*$ be the set of principal ultrafilters of $\B$ (those generated by the atoms).
These are isolated points in the Stone topology, and they form a dense set in the Stone topology since $\B$ is atomic.
So we have $X^*\cap T=\emptyset$ for every nowhere dense set $T$ (since principal ultrafilters,
which are isolated points in the Stone topology,
lie outside nowhere dense sets).
Recall that for $a\in \B$, $N_a$ denotes the set of all Boolean ultrafilters containing $a$.

Now  for all $\tau\in S_n$, we have
$G_{X, \tau}=S\sim \bigcup_{x\in X}N_{{\sf s}_{\tau}x}$
is nowhere dense. Let $F$ be a principal ultrafilter of $S$ containing $c$.
This is possible since $\B$ is atomic, so there is an atom $x$ below $c$; just take the
ultrafilter generated by $x$. Also $F$ lies outside the $G_{X,\tau}$'s, for all $\tau\in S_n$
Define, as we did before,  $f_c$ by $f_c(b)=\{\tau\in S_n: {\sf s}_{\tau}b\in F\}$.
Then clearly for every $\tau\in S_n$ there exists an atom $x$ such that $\tau\in f_c(x)$, so that $S_n=\bigcup_{x\in \At\A} f_c(x)$
Now for each $a\in A$, let
$V_a=S_n$ and let $V$ be the disjoint union of the $V_a$'s.
Then $\prod_{a\in A} \wp(V_a)\cong \wp(V)$. Define $f:\A\to \wp(V)$ by $f(x)=g[(f_ax: a\in A)]$.
Then $f: \A\to \wp(V)$ is an embedding such that
$\bigcup_{x\in \At\A}f(x)=V$. Hence $f$ is a complete representation.
\end{enumarab}
\end{example}

Let us consider algebras when both substitutions corresponding to replacements and transpositions
are available.
For any $\alpha\geq 2$ (infinite included)  we denote this class by $\sf SA_{\alpha}$. Here we have a strong completeness
theorem, namely, there is a finite schema of equations $\Sigma$ such that if $\A\models \Sigma$,
then $\A$ is representable (in the infinite dimensional case we use
weak units) but the class of subdirect products
of set algebras is a variety. When $\alpha$ is finite the finite schema is simply a finite set of equations.
Here is an example that these algebras may not be completely additive (idea borrowed from \cite{AGMNS}):

\begin{example}\label{modal}

First it is easy to show that complete representability forces that the non-boolean operations are completely additive.
Therefore it suffices to construct an atomic algebra such that $\Sigma_{x\in \At\A} {\sf s}_0^1x\neq 1$.

In what follows we produce such an algebra. (This algebra will be uncountable,
due to the fact that it is infinite and complete, so it cannot be countable.
In particular, it cannot be used to violate the omitting types theorem, the most it can say is that the omitting types theorem fails
for uncountable languages, which is not too much of a surprise).

Let $\mathbb{Z}^+$ denote the set of positive integers.
Let $U$ be an infinite set. Let $Q_n$, $n\in \omega$, be a family of $n$-ary relations that form  partition of $^nU$
such that $Q_0=D_{01}=\{s\in {}^nU: s_0=s_1\}$. And assume also that each $Q_n$ is symmetric; for any $i,j\in n$, $S_{ij}Q_n=Q_n$.
Clearly such a partition exists.
Now fix $F$ a non-principal ultrafilter on $\mathcal{P}(\mathbb{Z}^+)$. For each $X\subseteq \mathbb{Z}^+$, define
\[
 R_X =
  \begin{cases}
   \bigcup \{Q_k: k\in X\} & \text { if }X\notin F, \\
   \bigcup \{Q_k: k\in X\cup \{0\}\}      &  \text { if } X\in F
  \end{cases}
\]

Let $$\A=\{R_X: X\subseteq \mathbb{Z}^+\}.$$
Notice that $\A$ is uncountable. Then $\A$ is an atomic set algebra with unit $R_{\mathbb{Z}^+}$, and its atoms are $R_{\{k\}}=Q_k$ for $k\in \mathbb{Z}^+$.
(Since $F$ is non-principal, so $\{k\}\notin F$ for every $k$).
We check that $\A$ is indeed closed under the operations.
Let $X, Y$ be subsets of $\mathbb{Z}^+$. If either $X$ or $Y$ is in $F$, then so is $X\cup Y$, because $F$ is a filter.
Hence
$$R_X\cup R_Y=\bigcup\{Q_k: k\in X\}\cup\bigcup \{Q_k: k\in Y\}\cup Q_0=R_{X\cup Y}$$
If neither $X$ nor $Y$ is in $F$, then $X\cup Y$ is not in $F$, because $F$ is an ultrafilter.
$$R_X\cup R_Y=\bigcup\{Q_k: k\in X\}\cup\bigcup \{Q_k: k\in Y\}=R_{X\cup Y}$$
Thus $A$ is closed under finite unions. Now suppose that $X$ is the complement of $Y$ in $\mathbb{Z}^+$.
Since $F$ is an ultrafilter exactly one of them, say $X$ is in $F$.
Hence,
$$\sim R_X=\sim{}\bigcup \{Q_k: k\in X\cup \{0\}\}=\bigcup\{Q_k: k\in Y\}=R_Y$$
so that  $\A$ is closed under complementation (w.r.t $R_{\mathbb{Z}^+}$).
We check substitutions. Transpositions are clear, so we check only replacements. It is not too hard to show that
\[
 S_0^1(R_X)=
  \begin{cases}
   \emptyset & \text { if }X\notin F, \\
   R_{\mathbb{Z}^+}      &  \text { if } X\in F
  \end{cases}
\]

Now
$$\sum \{S_0^1(R_{k}): k\in \mathbb{Z}^+\}=\emptyset.$$
and
$$S_0^1(R_{\mathbb{Z}^+})=R_{\mathbb{Z}^+}$$
$$\sum \{R_{\{k\}}: k\in \mathbb{Z}^+\}=R_{\mathbb{Z}^+}=\bigcup \{Q_k:k\in \mathbb{Z}^+\}.$$
Thus $$S_0^1(\sum\{R_{\{k\}}: k\in \mathbb{Z}^+\})\neq \sum \{S_0^1(R_{\{k\}}): k\in \mathbb{Z}^+\}.$$
Our next theorem gives a plethora of algebras that are not completely representable. Any algebra which shares the atom structure of $\A$
constructed above cannot have a complete representation. Formally:
\end{example}

\begin{theorem} Let $\A$ be as in the previous example. Let $\B$ be an atomic an atomic algebra in $\sf SA_n$ such that
$\At\A\cong \At\B$. Then $\B$ is not completely representable
\end{theorem}
\begin{proof} Let such a $\B$ be given. Let $\psi:\At\A\to \At\B$ be an isomorphism of the atom structures (viewed as first order structures).
Assume for contradiction that $\B$ is completely representable, via $f$ say; $f:\B\to \wp(V)$ is an injective homomorphism such that
$\bigcup_{x\in \At\B}f(x)=V$. Define $g:\A\to \wp(V)$ by $g(a)=\bigcup_{x\in \At\A, x\leq a} f(\psi(x))$. Then, it can be easily checked that
$f$ establishes a complete representation of $\A$.
\end{proof}

We notice that in this latter case, if we take the subalgebra generated by the atoms, then the complex algebra of its atom structure is
{\it not  its completion}. For, the former is not completely additive while the latter is. The complex algebra is always completely additive, in particular,
if $\A\in SA_n$, then its canonical extension, since complex algebras are always completely additive, is completely representable;
this holds for the infinite dimensional case when we consider weak models, as well.
Thus we have:

\begin{theorem} Let $\alpha>1$ be an arbitrary ordinal.
Then $\A\in \sf SA_{\alpha}$ is representable if and only if its canonical extension is completely representable
on weak units.
\end{theorem}

The above is an open problem for cylindric algebras.
It is not known whether for  infinite dimensions the canonical extension of a representable algebra is completely representable
or not. This is true for finite dimensions though  a classical result of Monk.

In  completely additive varieties, the minimal completion of an atomic algebra $\B$ is just the complex algebra of its atom structure, namely,
$\Cm\At\B$. $\sf BAO$ denotes the class of Boolean algebras with operators.
More generally, we have:

\begin{theorem}\label{cm} Let $\B\in \sf BAO$ of type $t$ be atomic, and let $f\in t\setminus t_{BA}$ be completely additive on $\B$.
Then the map $^*:\B\to \Cm\At\B$, defined via $a\mapsto a^*=\{x\in \At\B: x\leq a\}$ is a Boolean algebra
embedding that respects $f$. If $\B$ is furthermore complete, then $^*$ is surjective.
\end{theorem}
\begin{proof}
It is clear that the map is a Boolean embedding.
We check that it preserves $f$. Suppose that $f(b_0)=b$. We need to show that $\Cm \At\B\models f(b_0^*)=b^*$.
Let $a\in \At\B$. Then by definition of $\Cm\At\B$, we have $\Cm \At\A\models \{a\}\leq f(b_0^*)$ iff $\At\A\models R_f(a_0, a)$
for some $a_0\in b_0^*$. By definition, this is iff $\B\models a\leq f(a_0)$ for some $a_0\in b_0^*$.
But because $f$ is additive, this happens  iff $\B\models a\leq f(b_0)$,
that is, iff $a\in b^*$ iff $\Cm\At\B\models \{a\}\leq b^*$, and we are done.

Now assume that $\B$ is complete and let $X\in \Cm\At\B$, then $X\subset \At\B$. Let $a=\sum X$, then it is easy to show that
$a^*=X$. This shows that the map $^*$ is surjective.
\end{proof}
In what follows, we use the following notation: Let $\B$ be an algebra having signature $t_1.$ If $t\subset t_1$, then $\Rd_t\B$ is the algebra obtained from $\B$ by
restricting operations to $t$, that is, discarding the operations in  $t\setminus t_1$.
\begin{corollary} The algebra $\A$, constructed in the above example, is complete.
Furthermore, if $t$ denotes the signature obtained from that of
$SA_\alpha$ by discarding replacements, then $\Rd_t\Cm\At\A\cong \Rd_t\A$.
On the other hand, if we discard transpositions, then we get an example of an atomic Pinter's algebra with no complete representation.
\end{corollary}
\begin{proof} We first show that $\A$ is complete.
Consider the family $\{R_{X_{i}}: i\in I\}$  where $\{X_i: i\in I\}$ is a given sequence of subsets of $J$.
Let $X=\bigcup \{X_i: i\in I\}$. We will show that
$$R_X=\sum\{R_{X_i}: i\in I\}.$$
Clearly $R_X$ is an upper bound. Also any upper bound
must be of the form $R_Y$ where $R_{X_{i}} \subset R_Y$, so  $X_i\subset Y$ for all $i\in I$.
It follows that $X\subset Y$.
In particular, if $X\in F$, then $Y\in F$, because $F$ is a filter. Therefore, $R_X\subset R_Y$, and we are done.
The second part follows from that transpositions are completely additive,
and the third follows immediately from \ref{cm}.
\end{proof}
However, this is not the case for $\PA$s, but all the same the class of strongly representable
algebras for $\PA_n$, $n>2$ is not elementary, as we proceed
to show:

\section{Graphs and Strong representability}

Throughout this section, $n$ is a finite ordinal $>2$.
\begin{definition}
Let $\Gamma=(G, E)$ be a graph.
\begin{enumerate}
\item{A set $X\subset G$ is said to be \textit{independent} if $E\cap(X\times X)=\phi$.}
\item{The \textit{chromatic number} $\chi(\Gamma)$ of $\Gamma$ is the smallest $\kappa<\omega$
such that $G$ can be partitioned into $\kappa$ independent sets,
and $\infty$ if there is no such $\kappa$.}
\end{enumerate}
\end{definition}

\begin{definition}\ \begin{enumerate}
\item For an equivalence relation $\sim$ on a set $X$, and $Y\subseteq
X$, we write $\sim\upharpoonright Y$ for $\sim\cap(Y\times Y)$. For
a partial map $K:n\rightarrow\Gamma\times n$ and $i, j<n$, we write
$K(i)=K(j)$ to mean that either $K(i)$, $K(j)$ are both undefined,
or they are both defined and are equal.
\item For any two relations $\sim$ and $\approx$. The composition of $\sim$ and $\approx$ is
the set
$$\sim\circ\approx=\{(a, b):\exists c(a\sim c\wedge c\approx b)\}.$$\end{enumerate}
\end{definition}
\begin{definition}Let $\Gamma$ be a graph.
We define an atom structure $\eta(\Gamma)=\langle H, D_{ij},
\equiv_{i}, \equiv_{ij}:i, j<n\rangle$ as follows:
\begin{enumerate}
\item$H$ is the set of all pairs $(K, \sim)$ where $K:n\rightarrow \Gamma\times n$ is a partial map and $\sim$ is an equivalent relation on $n$
satisfying the following conditions \begin{enumerate}\item If
$|n\diagup\sim|=n$, then $\dom(K)=n$ and $rng(K)$ is not independent
subset of $n$.
\item If $|n\diagup\sim|=n-1$, then $K$ is defined only on the unique $\sim$ class $\{i, j\}$ say of size $2$ and $K(i)=K(j)$.
\item If $|n\diagup\sim|\leq n-2$, then $K$ is nowhere defined.
\end{enumerate}
\item $D_{ij}=\{(K, \sim)\in H : i\sim j\}$.
\item $(K, \sim)\equiv_{i}(K', \sim')$ iff $K(i)=K'(i)$ and $\sim\upharpoonright(n\setminus\{i\})=\sim'\upharpoonright(n\setminus\{i\})$.
\item $(K, \sim)\equiv_{ij}(K', \sim')$ iff $K(i)=K'(j)$, $K(j)=K'(i)$, and $K(\kappa)=K'(\kappa) (\forall\kappa\in n\setminus\{i, j\})$
and if $i\sim j$ then $\sim=\sim'$, if not, then $\sim'=\sim\circ[i,
j]$.
\end{enumerate}
\end{definition}
It may help to think of $K(i)$ as assigning the nodes $K(i)$ of
$\Gamma\times n$ not to $i$ but to
the set $n\setminus\{i\}$, so long as its elements are pairwise non-equivalent via $\sim$.\\
For a set $X$, $\mathcal{B}(X)$ denotes the boolean algebra
$\langle\wp(X), \cup, \setminus\rangle$. We write $a\cap b$ for
$-(-a\cup-b)$.
\begin{definition}\label{our algebra}
Let $\B(\Gamma)=\langle\B(\eta(\Gamma)), {\sf c}_{i},
{\sf s}^{i}_{j}, {\sf s}_{ij}, {\sf d}{ij}\rangle_{i, j<n}$ be the algebra, with
extra non-Boolean operations defined as follows:

${\sf d}{ij}=D_{ij}$,
${\sf c}_{i}X=\{c: \exists a\in X, a\equiv_{i}c\}$,
${\sf s}_{ij}X=\{c: \exists a\in X, a\equiv_{ij}c\}$,
${\sf s}^{i}_{j}X=\begin{cases}
{\sf c}_{i}(X\cap D_{ij}), &\text{if $i\not=j$,}\\
X, &\text{if $i=j$.}
\end{cases}$
For all $X\subseteq \eta(\Gamma)$.

\end{definition}

\begin{definition}
For any $\tau\in\{\pi\in n^{n}: \pi \text{ is a bijection}\}$, and
any $(K, \sim)\in\eta(\Gamma)$. We define $\tau(K,
\sim)=(K\circ\tau, \sim\circ\tau)$.\end{definition}

The proof of the following two Lemmas is straightforward.
\begin{lemma}\label{Lemma 1}\ \\
For any $\tau\in\{\pi\in n^{n}: \pi \text{ is a bijection}\}$, and
any $(K, \sim)\in\eta(\Gamma)$. $\tau(K, \sim)\in\eta(\Gamma)$.
\end{lemma}
\begin{lemma}\label{Lemma 2}\ \\For any $(K, \sim)$, $(K', \sim')$, and $(K'', \sim'')\in\eta(\Gamma)$, and $i, j\in n$:
\begin{enumerate}
\item$(K, \sim)\equiv_{ii}(K', \sim')\Longleftrightarrow (K, \sim)=(K', \sim')$.
\item$(K, \sim)\equiv_{ij}(K', \sim')\Longleftrightarrow (K, \sim)\equiv_{ji}(K', \sim')$.
\item If $(K, \sim)\equiv_{ij}(K', \sim')$, and $(K, \sim)\equiv_{ij}(K'', \sim'')$, then $(K', \sim')=(K'', \sim'')$.
\item If $(K, \sim)\in D_{ij}$, then \\
$(K, \sim)\equiv_{i}(K', \sim')\Longleftrightarrow\exists(K_{1},
\sim_{1})\in\eta(\Gamma):(K, \sim) \equiv_{j}(K_{1},
\sim_{1})\wedge(K', \sim')\equiv_{ij}(K_{1}, \sim_{1})$.
\item${\sf s}_{ij}(\eta(\Gamma))=\eta(\Gamma)$.
\end{enumerate}
\end{lemma}
The proof of the next lemma is tedious but not too hard.
\begin{theorem}\label{it is qea}
For any graph $\Gamma$, $\B(\Gamma)$ is a simple
$\PEA_n$.
\end{theorem}
\begin{proof}\
We follow the axiomatization in \cite{thompson} except renaming the items by $Q_i$.
Let $X\subseteq\eta(\Gamma)$, and $i, j, \kappa\in n$:
\begin{enumerate}
\item ${\sf s}^{i}_{i}=ID$ by definition \ref{our algebra}, ${\sf s}_{ii}X=\{c:\exists a\in X, a\equiv_{ii}c\}=\{c:\exists a\in X, a=c\}=X$
(by Lemma \ref{Lemma 2} (1));\\
${\sf s}_{ij}X=\{c:\exists a\in X, a\equiv_{ij}c\}=\{c:\exists a\in X,
a\equiv_{ji}c\}={\sf s}_{ji}X$ (by Lemma \ref{Lemma 2} (2)).
\item Axioms $Q_{1}$, $Q_{2}$ follow directly from the fact that the
reduct $\mathfrak{Rd}_{ca}\mathfrak{B}(\Gamma)=\langle\mathcal{B}(\eta(\Gamma)), {\sf c}_{i}$, ${\sf d}{ij}\rangle_{i, j<n}$
is a cylindric algebra which is proved in \cite{hirsh}.
\item Axioms $Q_{3}$, $Q_{4}$, $Q_{5}$
follow from the fact that the reduct $\mathfrak{Rd}_{ca}\mathfrak{B}(\Gamma)$ is a cylindric algebra, and from \cite{tarski}
(Theorem 1.5.8(i), Theorem 1.5.9(ii), Theorem 1.5.8(ii)).
\item ${\sf s}^{i}_{j}$ is a boolean endomorphism by \cite{tarski} (Theorem 1.5.3).
\begin{align*}
{\sf s}_{ij}(X\cup Y)&=\{c:\exists a\in(X\cup Y), a\equiv_{ij}c\}\\
&=\{c:(\exists a\in X\vee\exists a\in Y), a\equiv_{ij}c\}\\
&=\{c:\exists a\in X, a\equiv_{ij}c\}\cup\{c:\exists a\in Y,
a\equiv_{ij}c\}\\
&={\sf s}_{ij}X\cup {\sf s}_{ij}Y.
\end{align*}
${\sf s}_{ij}(-X)=\{c:\exists a\in(-X), a\equiv_{ij}c\}$, and
${\sf s}_{ij}X=\{c:\exists a\in X, a\equiv_{ij}c\}$ are disjoint. For, let
$c\in({\sf s}_{ij}(X)\cap {\sf s}_{ij}(-X))$, then $\exists a\in X\wedge b\in
(-X)$, such that $a\equiv_{ij}c$, and $b\equiv_{ij}c$. Then $a=b$, (by
Lemma \ref{Lemma 2} (3)), which is a contradiction. Also,
\begin{align*}
{\sf s}_{ij}X\cup {\sf s}_{ij}(-X)&=\{c:\exists a\in X,
a\equiv_{ij}c\}\cup\{c:\exists a\in(-X), a\equiv_{ij}c\}\\
&=\{c:\exists a\in(X\cup-X), a\equiv_{ij}c\}\\
&={\sf s}_{ij}\eta(\Gamma)\\
&=\eta(\Gamma). \text{ (by Lemma \ref{Lemma 2} (5))}
\end{align*}
therefore, ${\sf s}_{ij}$ is a boolean endomorphism.
\item \begin{align*}{\sf s}_{ij}{\sf s}_{ij}X&={\sf s}_{ij}\{c:\exists a\in X, a\equiv_{ij}c\}\\
&=\{b:(\exists a\in X\wedge c\in\eta(\Gamma)), a\equiv_{ij}c, \text{ and }
c\equiv_{ij}b\}\\
&=\{b:\exists a\in X, a=b\}\\
&=X.\end{align*}
\item\begin{align*}{\sf s}_{ij}{\sf s}^{i}_{j}X&=\{c:\exists a\in {\sf }^{i}_{j}X, a\equiv_{ij}c\}\\
&=\{c:\exists b\in(X\cap {\sf d}{ij}),a\equiv_{i}b\wedge
a\equiv_{ij}c\}\\
&=\{c:\exists b\in(X\cap {\sf d}{ij}), c\equiv_{j}b\} \text{ (by Lemma
\ref{Lemma 2} (4))}\\
&={\sf }^{j}_{i}X.\end{align*}
\item We need to prove that ${\sf s}_{ij}{\sf s}_{i\kappa}X={\sf s}_{j\kappa}{\sf s}_{ij}X$ if $|\{i, j, \kappa\}|=3$.
Let $(K, \sim)\in {\sf s}_{ij}{\sf s}_{i\kappa}X$ then
$\exists(K', \sim')\in\eta(\Gamma)$, and $\exists(K'', \sim'')\in X$
such that $(K'', \sim'')\equiv_{i\kappa}(K', \sim')$ and $(K',
\sim')\equiv_{ij}(K, \sim)$.\\
Define $\tau:n\rightarrow n$ as follows:
\begin{align*}\tau(i)&=j\\
\tau(j)&=\kappa\\
\tau(\kappa)&=i, \text{ and}\\
\tau(l)&=l \text{ for every } l\in(n\setminus\{i, j,
\kappa\}).\end{align*} Now, it is easy to verify that $\tau(K',
\sim')\equiv_{ij}(K'', \sim'')$, and $\tau(K',
\sim')\equiv_{j\kappa}(K, \sim)$. Therefore, $(K, \sim)\in
{\sf s}_{j\kappa}{\sf s}_{ij}X$, i.e., ${\sf s}_{ij}{\sf s}_{i\kappa}X\subseteq
{\sf s}_{j\kappa}{\sf s}_{ij}X$. Similarly, we can show that
${\sf s}_{j\kappa}{\sf s}_{ij}X\subseteq {\sf s}_{ij}{\sf s}_{i\kappa}X$.
\item Axiom $Q_{10}$ follows from \cite{tarski} (Theorem 1.5.7)
\item Axiom $Q_{11}$ follows from axiom 2, and the definition of $s^{i}_{j}$.
\end{enumerate}
Since $\Rd_{ca}\B$ is a simple $\CA_{n}$, by
\cite{hirsh}, then $\B$ is a simple $\PEA_n$. This follows from the fact that ideals $I$ is an ideal in $\Rd_{ca}\B$ if and only if it is an ideal in
$\B$.
\end{proof}
\begin{definition}
Let $\C(\Gamma)$ be the subalgebra of
$\B(\Gamma)$ generated by the set of atoms.
\end{definition}
Note that the cylindric algebra constructed in \cite{hirsh} is
$\Rd_{ca}\B(\Gamma)$ not
$\Rd_{ca}\C(\Gamma)$, but all results in
\cite{hirsh} can be applied to
$\Rd_{ca}\C(\Gamma)$. Therefore, since our
results depends basically on \cite{hirsh}, we will refer to
\cite{hirsh} directly when we apply it to get any result on
$\Rd_{ca}\C(\Gamma)$.

\begin{theorem}
$\C(\Gamma)$ is a simple $\PEA_{n}$ generated by the set of the
$n-1$ dimensional elements.
\end{theorem}
\begin{proof}
$\mathfrak{C}(\Gamma)$ is a simple $QEA_{n}$ from  Theorem \ref{it
is qea}. It remains to show that $\{(K, \sim)\}=\prod\{{\sf c}_{i}\{(K,
\sim)\}: i<n\}$ for any $(K, \sim)\in H$. Let $(K, \sim)\in H$,
clearly $\{(K, \sim)\}\leq\prod\{{\sf c}_{i}\{(K, \sim)\}: i<n\}$. For the
other direction assume that $(K', \sim')\in H$ and $(K,
\sim)\not=(K', \sim')$. We show that $(K',
\sim')\not\in\prod\{{\sf c}_{i}\{(K, \sim)\}: i<n\}$. Assume toward a
contradiction that $(K', \sim')\in\prod\{{\sf c}_{i}\{(K, \sim)\}:i<n\}$,
then $(K', \sim')\in {\sf c}_{i}\{(K, \sim)\}$ for all $i<n$, i.e.,
$K'(i)=K(i)$ and
$\sim'\upharpoonright(n\setminus\{i\})=\sim\upharpoonright(n\setminus\{i\})$
for all $i<n$. Therefore, $(K, \sim)=(K', \sim')$ which makes a
contradiction, and hence we get the other direction.
\end{proof}
\begin{theorem}\label{chr. no.}
Let $\Gamma$ be a graph.
\begin{enumerate}\item Suppose that $\chi(\Gamma)=\infty$. Then $\mathfrak{C}(\Gamma)$
is representable.\item If $\Gamma$ is infinite and
$\chi(\Gamma)<\infty$ then $\Rd_{df}\C$
is not
representable
\end{enumerate}
\end{theorem}
\begin{proof}
\begin{enumerate}
\item We have $\Rd_{ca}\C$ is representable (c.f., \cite{hirsh}).
Let $X=\{x\in \C:\Delta x\not=n\}$. Call $J\subseteq \C$ inductive if
$X\subseteq J$ and $J$ is closed under infinite unions and
complementation. Then $\C$ is the smallest inductive
subset of $C$. Let $f$ be an isomorphism of
$\Rd_{ca}\C$ onto a cylindric set algebra with
base $U$. Clearly, by definition, $f$ preserves $s^{i}_{j}$ for each
$i, j<n$. It remains to show that $f$ preserves ${\sf s}_{ij}$ for every
$i, j<n$. Let $i, j<n$, since ${\sf s}_{ij}$ is boolean endomorphism and
completely additive, it suffices to show that $f{\sf s}_{ij}x={\sf s}_{ij}fx$
for all $x\in \At\C$. Let $x\in \At\C$ and $\mu\in
n\setminus\Delta x$. If $\kappa=\mu$ or $l=\mu$, say $\kappa=\mu$,
then$$ f{\sf s}_{\kappa l}x=f{\sf s}_{\kappa
l}{\sf c}_{\kappa}x=f{\sf s}^{\kappa}_{l}x={\sf s}^{\kappa}_{l}fx={\sf s}_{\kappa l}fx.
$$
If $\mu\not\in\{\kappa, l\}$ then$$ f{\sf s}_{\kappa
l}x=f{\sf s}^{l}_{\mu}s^{\kappa}_{l}{\sf s}^{\mu}_{\kappa}{\sf c}_{\mu}x={\sf s}^{l}_{\mu}s^{\kappa}_{l}{\sf s}^{\mu}_{\kappa}{\sf c}_{\mu}fx={\sf s}_{\kappa
l}fx.
$$
\item Assume toward a contradiction that $\Rd_{df}\C$ is representable. Since $\Rd_{ca}\C$
is generated by $n-1$ dimensional elements then
$\mathfrak{Rd}_{ca}\C$ is representable. But this
contradicts Proposition 5.4 in \cite{hirsh}.
\end{enumerate}
\end{proof}
\begin{theorem}\label{el}
Let $2<n<\omega$ and $\mathcal{T}$ be any signature between $\Df_{n}$
and $\PEA_{n}$. Then the class of strongly representable atom
structures of type $\mathcal{T}$ is not elementary.
\end{theorem}
\begin{proof}
By Erd\"{o}s's famous 1959 Theorem \cite{Erdos}, for each finite
$\kappa$ there is a finite graph $G_{\kappa}$ with
$\chi(G_{\kappa})>\kappa$ and with no cycles of length $<\kappa$.
Let $\Gamma_{\kappa}$ be the disjoint union of the $G_{l}$ for
$l>\kappa$. Clearly, $\chi(\Gamma_{\kappa})=\infty$. So by Theorem
\ref{chr.
no.} (1), $\mathfrak{C}(\Gamma_{\kappa})=\mathfrak{C}(\Gamma_{\kappa})^{+}$ is representable.\\
\- Now let $\Gamma$ be a non-principal ultraproduct
$\prod_{D}\Gamma_{\kappa}$ for the $\Gamma_{\kappa}$. It is
certainly infinite. For $\kappa<\omega$, let $\sigma_{\kappa}$ be a
first-order sentence of the signature of the graphs. stating that
there are no cycles of length less than $\kappa$. Then
$\Gamma_{l}\models\sigma_{\kappa}$ for all $l\geq\kappa$. By
Lo\'{s}'s Theorem, $\Gamma\models\sigma_{\kappa}$ for all
$\kappa$. So $\Gamma$ has no cycles, and hence by, \cite{hirsh}
Lemma 3.2, $\chi(\Gamma)\leq 2$. By Theorem \ref{chr. no.} (2),
$\mathfrak{Rd}_{df}\mathfrak{C}$ is not representable. It is easy to
show (e.g., because $\mathfrak{C}(\Gamma)$ is first-order
interpretable in $\Gamma$, for any $\Gamma$) that$$
\prod_{D}\mathfrak{C}(\Gamma_{\kappa})\cong\mathfrak{C}(\prod_{D}\Gamma_{\kappa}).$$
Combining this with the fact that: for any $n$-dimensional atom
structure $\mathcal{S}$

$\mathcal{S}$ is strongly
representable $\Longleftrightarrow$ $\mathfrak{Cm}\mathcal{S}$ is
representable,
the desired follows.
\end{proof}

\subsubsection{The good and the bad}

Here we give a different approach to Hirsch Hodkinson's result.
We use the notation and the general ideas in \cite[lemmas 3.6.4, 3.6.6]{HHbook2}.
An important difference is that our cylindric algebras are binary generated,
and they their atom structures are the set of basic matrices on relation algebras satisfying the same properties.
We abstract the two Monk-like algebras dealt with in the proof of theorem \ref{hodkinson}.

Let $\G$ be a graph. One can  define a family of first order structures (labelled graphs)  in the signature $\G\times n$, denote it by $I(\G)$
as follows:
For all $a,b\in M$, there is a unique $p\in \G\times n$, such that
$(a,b)\in p$. If  $M\models (a,i)(x,y)\land (b,j)(y,z)\land (c,l)(x,z)$, then $| \{ i, j, l \}> 1 $, or
$ a, b, c \in \G$ and $\{ a, b, c\} $ has at least one edge
of $\G$.
For any graph $\Gamma$, let $\rho(\Gamma)$ be the atom structure defined from the class of models satisfying the above,
these are maps from $n\to M$, $M\in I(\G)$, endowed with an obvious equivalence relation,
with cylindrifiers and diagonal elements defined as Hirsch and Hodkinson define atom structures from classes of models,
and let $\M(\Gamma)$ be the complex algebra of this atom structure.

We define a relation algebra atom structure $\alpha(\G)$ of the form
$(\{1'\}\cup (\G\times n), R_{1'}, \breve{R}, R_;)$.
The only identity atom is $1'$. All atoms are self converse,
so $\breve{R}=\{(a, a): a \text { an atom }\}.$
The colour of an atom $(a,i)\in \G\times n$ is $i$. The identity $1'$ has no colour. A triple $(a,b,c)$
of atoms in $\alpha(\G)$ is consistent if
$R;(a,b,c)$ holds. Then the consistent triples are $(a,b,c)$ where

\begin{itemize}

\item one of $a,b,c$ is $1'$ and the other two are equal, or

\item none of $a,b,c$ is $1'$ and they do not all have the same colour, or

\item $a=(a', i), b=(b', i)$ and $c=(c', i)$ for some $i<n$ and
$a',b',c'\in \G$, and there exists at least one graph edge
of $G$ in $\{a', b', c'\}$.

\end{itemize}Note that some monochromatic triangles
are allowed namely the 'dependent' ones.
This allows the relation algebra to have an $n$ dimensional cylindric basis
and, in fact, the atom structure of $\M(\G)$ is isomorphic (as a cylindric algebra
atom structure) to the atom structure $\Mat_n$ of all $n$-dimensional basic
matrices over the relation algebra atom structure $\alpha(\G)$.

We show that $\alpha(\G)$ is strongly representable iff $\M(\G)$ is representable iff
$\G$ has infinite chromatic number. This will give the result that strongly representable atom structures of both relation algebras
and cylindric algebras of finite dimension $>2$ in one go, using Erdos' graphs of large
chromatic number and girth.

The idea here is that the shades of red (addressed in the proof of theorem \ref{hodkinson} for the two Monk-like algebras)
will appear in the {\it ultrafilter extension} of $\G$, if it has infinite chromatic number as a reflexive node, \cite[definition 3.6.5]{HHbook2}.
and its $n$ copies,
can be used to completely represent
$\M(\G)^{\sigma}$ (the canonical extension of $\M(\G)$).
Let $\M(\G)_+$ be the ultrafilter atom structure of $\M(\G)$

Fix $\G$, and let  $\G^* $ be the ultrafilter extension of $\G$.
First define a strong bounded morphism $\Theta$
form $\M(\G)_+$ to $\rho(I(\G^*))$, as follows:
For any $x_0, x_1<n$ and $X\subseteq \G^*\times n$, define the following element
of $\M(\G^*)$:
$$X^{(x_0, x_1)}=\{[f]\in \rho(I(\G^*)): \exists p\in X[M_f\models p(f(x_0),f(x_1))]\}.$$
Let $\mu$ be an ultrafilter in $\M(\G)$. Define $\sim $ on $n$ by $i\sim j$ iff $\sf d_{ij}\in \mu$.
Let $g$ be the projection map from $n$ to $n/\sim$.
Define a $\G^*\times n$ coloured graph with domain $n/\sim$ as follows.
For each $v\in \Gamma^*\times n$
and $x_0, x_1<n$, we let
$$M_{\mu}\models v(g(x_0), g(x_1))\Longleftrightarrow  X^{(x_0, x_1)}\in \mu.$$
Hence, any ultrafilter $\mu\in \M(\G)$ defines $M_{\mu}$ which is  a $\G^*$ structure.
If $\Gamma$ has infinite chromatic number, then $\G^*$ has a reflexive node, and this can be used
to completely represent $\M(\G)^{\sigma}$, hence represent  $\M(\G)$ as follows:
To do this one tries to show  \pe\ has a \ws\ in the usual $\omega$ rounded atomic game on networks \cite{HHbook2}, that test complete
representability.

Her strategy can be implemented using the following argument.
Let $N$ be a given $\M(\G)^{\sigma}$ network. Let $z\notin N$ and let $y=x[i|z]\in {}^n(N\cup \{z\}={}^nM$.
Write $Y=\{y_0,\ldots y_{n-1}\}$.  We need to complete the labelling of edges of $M$.
We have a fixed $i\in n$. Defines $q_j$ $:j\in n\sim \{i\}$, the unique label of any two distinct elements in $Y\sim y_j$,
if the latter elements are pairwise distinct, and arbitrarily
otherwise.
Let $d\in \G^*$ be a reflexive node in the copy that does not contain any of the $q_j$s (there number is $n-1$),
and define $M\models d(t_0, t_1)$ if $z\in \{t_0, t_1\}\nsubseteq Y$.
Labelling the  other edges are like $N$.
The rest of the proof is similar to \cite{HHbook2}.

The idea above is essentially due to Hirsch and Hodkinson, it also works for relation and cylindric algebras, and this is the essence. For each graph
$\Gamma$, they associate a cylindric algebra atom structure of dimension $n$, $\M(\Gamma)$ such that $\Cm\M(\Gamma)$ is representable
if and only if the chromatic number of $\Gamma$, in symbols $\chi(\Gamma)$, which is the least number of colours needed, $\chi(\Gamma)$ is infinite.
Using a famous theorem of Erdos as we did above, they construct  a sequence $\Gamma_r$ with infinite chromatic number and finite girth,
whose limit is just $2$ colourable, they show that the class of strongly representable
algebras  is not elementary. This is a {\it reverse process} of Monk-like  constructions,
which gives a sequence of graphs of finite chromatic number whose limit (ultraproduct) has infinite
chromatic number.

In more detail, some statement fails in $\A$ iff $\At\A$
be partitioned into finitely many $\A$-definable sets with certain
`bad' properties. Call this a {\it bad partition}.
A bad partition of a graph is a finite colouring. So Monk's result finds a sequence of badly partitioned atom structures,
converging to one that is not.  This boils down, to finding graphs of finite chromatic numbers $\Gamma_i$, having an ultraproduct
$\Gamma$ with infinite chromatic number.

Then an  atom structure is {\it strongly representable} iff it
has {\it no bad partition using any sets at all}. So, here, the idea  is to {\it find atom structures, with no bad partitions
with an ultraproduct that does have a bad partition.}
From a graph Hirsch and Hodkinson constructed  an atom structure that is strongly representable iff the graph
has no finite colouring.  So the problem that remains is to find a sequence of graphs with no finite colouring,
with an ultraproduct that does have a finite colouring, that is, graphs of infinite chromatic numbers, having an ultraproduct
with finite chromatic number.

It is not obvious, a priori, that such graphs actually exist.
And here is where Erdos' methods offer solace.
Indeed, graphs like are found using the probabilistic methods of Erdos, for those methods
render finite graphs of arbitrarily large chormatic number and girth.
By taking disjoint unions as above, one {\it can get}
graphs of infinite chromatic number (no bad partitions) and arbitrarily large girth. A non principal
ultraproduct of these has no cycles, so has chromatic number 2 (bad partition).
This motivates:

\begin{definition}
\begin{enumarab}

\item A Monks algebra is good if $\chi(\Gamma)=\infty$

\item A Monk's algebra is bad if $\chi(\Gamma)<\infty$

\end{enumarab}
\end{definition}
It is easy to construct a good Monks algebra
as an ultraproduct (limit) of bad Monk algebras. Monk's original algebras can be viewed this way.
The converse is, as illustrated above, is much harder.
It took Erdos probabilistic graphs, to get a sequence of good graphs converging to a bad one.

Using our modified Monk-like algebras we obtain the result formulated in theorem \ref{el}.
An immediate corollary is:
The following corollary answers a question of Hodkinson's in \cite{AU} see p. 284.
\begin{corollary} For $\K$ any class between $\Df$ and $\PEA$ and any finite $n>2$,
the class $\K_s=\{\A\in \K_n: \Cm\At\A\in \sf RK_n\}$, is not elementary.
\end{corollary}

\section{Completions and decidability}

Throughout this section $m>2$ will denote the dimension and $n$ will be always finite $>m.$
We  address the decidability of the problem as to whether a finite $\CA_m$ is in $S\Nr_m\CA_n$ 
We obtain a negative result the lowest value of
$m$ namely when $m=3$ and $n\geq m+3$. 
This generalizes the result of Hodkinson in \cite{AU} proved for $\RCA_n$ but only for $n=3$. For higher dimensions, the problem to the best 
of our knowledge remains unsettled. We also connect algebras having a (complete) neat embedding property to 
(complete) relativized representations, generalizing results proved
by Hirsch and Hodkinson for relation algebras, witness \cite[theorems 13.45, 13. 46]{HHbook}.

\begin{theorem}\label{decidability} Let $m\geq 3$. Assume that for any simple atomic relation algebra $\A$ with atom structure $S$,
there is a cylindric atom structure $H$, constructed effectively from $\At\A$,  such that:
\begin{enumarab}
\item If $\Tm S\in \sf RRA$, then $\Tm H\in \RCA_m$.
\item If $S$ is finite, then $H$ is finite
\item $\Cm S$ is embeddable in $\Ra$ reduct of $\Cm H$.
\end{enumarab}
Then  for all $k\geq 3$, $S\Nr_m\CA_{m+k}$ is not closed under completions.
and  it is undecidable whether a finite cylindric algebra is in
and $S\Nr_m\CA_{m+k}$, and
\end{theorem}

\begin{demo}{Proof} For the first part. Let $S$ be a relation atom structure such that $\Tm S$ is representable while $\Cm S\notin \RA_6$.
Such an atom structure exists \cite[lemmas 17.34-35-36-37]{HHbook}. 
It follows that $\Cm S\notin {\bf S}\Ra\CA_m$.
Let $H$ be the $\CA_m$ atom structure provided by the hypothesis of the previous theorem.
Then $\Tm H\in \RCA_m$. We claim that $\Cm H\notin {\bf S}\Nr_m\CA_{m+k}$, $k\geq 3$.
For assume not, i.e. assume that $\Cm H\in {\bf S}\Nr_m\CA_{m+k}$, $k\geq 3$.
We have $\Cm S$ is embeddable in $\Ra\Cm H.$  But then the latter is in ${\bf S}\Ra\CA_6$
and so is $\Cm S$, which is not the case.

For the second part, suppose for contradiction that there is an algorithm ${\sf Al}$ to determine whether a finite
algebra is in $S\Nr_m\CA_{m+k}$. We claim that we can now decide effectively whether a finite simple atomic
relation algebra is in $S\Ra\CA_5$  by constructing the $\CA_m$, $\C=\Tm H$ which is finite,
and returning  the answer ${\sf Al}(\C)$.
This is the correct answer.  However,
this contradicts the result that it is undecidable to tell whether a finite relation algebra is representable or not
\cite[theorem 18.13]{HHbook}.

\end{demo}

\begin{corollary} Lety $n\geq 6$. Then the following hold: 
\begin{enumarab}

\item The set of isomorphism types of algebras in $S\Nr_3\CA_n$ with infinite flat representations is not recursively enumerable

\item The equational theory of $S\Nr_3\CA_n$ is undecidable

\item The variety $S\Nr_3\CA_n$  is not finitely axiomatizable even 
in $n$ order logic.

\item For every $n\geq 6$, there exists a finite algebra in $S\Nr_3\CA_6$ that does not have a finite $n$ dimensional hyperbasis

\end{enumarab}
\end{corollary}

\begin{proof}
\begin{enumarab}
\item  Direct

\item Any such axiomatization will give a decision procedure for the class of finite algebras

\item  Let $\K=S\Nr_3\CA_n$. 
We reduce the problem of telling if a finite simple algebra is not in 
$\K$ to the problem of telling if an equation in the language of $\CA_3$ is valid in 
$\K$. Let $\A$  be a finite relation algebra. Form $\Delta(\A)$, the diagram of $\A$, 
but using variables instead of constants; define its conjunction. 
Consider $g=\exists_{\bar{a}\in \A} \Delta(A)$ relative to some enumeration $\bar{a}$ of $A$. 
Then $\B\models g$ iff $A\nsubseteq \B$. Since  $\K$ is a discriminator variety, 
the quantifier free $\neg \Delta(\A)$ is equivalent over
simple algebras to an equation $e$, which is effectively constructed 
from $\A$. So, in fact, we have
$\B\models e$ iff $\A\nsubseteq \B$. 
So $e$ is valid over simple $\K$ algebra iff for all $\B$ in $\K$, 
$\A\nsubseteq \B$. Since $\A$ is simple, this happens if and only if $\A\notin \K$.

\item  Assume for contradiction that every finite algebra in $S\Nr_3\CA_n$ has a finite
$n$ dimensional hyperbasis. 
We claim that there is an algorithm that decides membership in $S\Nr_3\CA_6$ for finite algebras:
\begin{itemize}
\item Using a recursive axiomatization of $S\Nr_3\CA_n$ (exists), recursively enumerate all isomorphism types of
finite $\CA_3$s that are not in $S\Nr_3\CA_n.$
\item Recursively enumerate all finite algebras in $S\Nr_3\CA_n$.  
For each such algebra, enumerate all finite sets of $n$ dimensional hypernetworks over $\A$, 
using $\N$ as hyperlabels, and check  to
see if it is a hyperbasis. When a hypebasis is located specify $\A$. 
This recursively enumerates  all and only the finite algebras in $S\Nr_3\CA_n$. 
Since any finite $\CA_3$ is in exactly one of these enumerations, the process will decide
whether or not it is in  $S\Nr_3\CA_6$ in a finite time.
\end{itemize}

\end{enumarab}
\end{proof}

Concerning theorem \ref{decidability}, we note that Monk and Maddux constructs such an $H$ for $n=3$ and
Hodkinson constructs  an $H$ for arbitrary dimensions $>2$, but $H$
unfortunately the relation algebra
does not embed into $\Cm H$ \cite{AU}.

\begin{corollary}\label{sah} Assume the hypothesis in \ref{blurs}; let $k\geq 1$ and $n$ be finite with $n>2$.
Then the following hold; in particular, when $k=3$ we know, by theorem \ref{smooth},  that the following indeed hold.

\begin{enumarab}

\item There exist two atomic
cylindric algebras of dimension $m$  with the same atom structure,
one  representable and the other is not in $S\Nr_m\CA_{m+k+1}$.

\item For $n\geq 3$ and $k\geq 3$, ${S}\Nr_m\CA_{m+k+1}$
is not closed under completions and is not atom-canonical.
In particular, $\RCA_m$ is not atom-canonical.

\item There exists an algebra in $S\Nr_m\CA_{m+k+1}$  with a dense representable
subalgebra.

\item For $m\geq 3$ and $k\geq 3$,  ${S}\Nr_n\CA_{m+k+1}$
is not Sahlqvist axiomatizable. In particular, $\RCA_m$ is not Sahlqvist axiomatizable.

\item There exists an atomic representable
$\CA_n$ with no $m+k+1$ smooth complete representation; in particular it has no complete
representation.

\item The omitting types theorem fails for clique guarded semantics, when size of cliques are $< m+k+1$.

\end{enumarab}

\end{corollary}\label{flat}

\begin{proof} We use the cylindric algebra, proved to exist conditionally, namely,
$\A=\C_{n}$ in \ref{blurs}; in conformity with our notation we switch its dimension to $m$. 
The term algebra, which is contained in $\C_m$ also can be used.

\begin{enumarab}

\item $\A$ and $\Cm\At\A$ are such.

\item $\Cm\At\A$ is the \d\ completion of $\A$ (even in the $\PA$ and $\Sc$ cases, because $\A$,
hence its $\Sc$ and $\PA$ reducts are completely additive), hence $S\Nr_m\CA_{m+k+1}$ is not atom canonical \cite[proposition 2,88,
theorem, 2.96]{HHbook}.

\item $\A$ is dense in $\Cm\At\A$.

\item Completely additive varieties defined by Sahlqvist equations are closed under \d\ completions \cite[theorem 2.96]{HHbook}.

\item Assume that $\A$ has an $n=m+k+1$ smooth complete representation $\M$. $L(A)$ denotes the signature that contains
an $n$ ary predicate for every $a\in A$.
For $\phi\in L(A)_{\omega,\infty}^n$,
let $\phi^{M}=\{\bar{a}\in C^n(M):M\models_C \phi(\bar{a})\}$,
and let $\D$ be the algebra with universe $\{\phi^{M}: \phi\in L_{\infty,\omega}^n\}$ with usual
Boolean operations, cylindrifiers and diagonal elements, cf. theorem 13.20 in \cite{HHbook}. The polyadic operations are defined
by swapping variables.
Define $\D_0$ be the algebra consisting of those $\phi^{M}$ where $\phi$ comes from $L^n$.
Assume that $M$ is $n$ square, then certainly $\D_0$ is a subalgebra of the $\sf Crs_n$ (the class
of algebras whose units are arbitrary sets of $n$ ary sequences)
with domain $\wp(C^n(M))$ so $\D_0\in {\sf Crs_n}$. The unit $C^n(M)$ of $\D_0$ is symmetric,
closed under substitutions, so
$\D_0\in \sf G_n$ (these are relativized set algebras whose units are locally cube, they
are closed under substitutions.)  Since $M$ is $n$ flat we
have that cylindrifiers commute by definition,
hence $\D_0\in \CA_n$.

Now since $M$ is infinitary $n$ smooth then it is infinitary $n$ flat.
Then one proves that  $\D\in \CA_n$ in exactly the same way.
Clearly $\D$ is complete. We claim that $\D$ is atomic.
Let $\phi^M$ be a non zero element.
Choose $\bar{a}\in \phi^M$, and consider the infinitary conjunction
$\tau=\bigwedge \{\psi\in L_{\infty}: M\models_C \psi(\bar{a})\}.$
Then $\tau\in L_{\infty}$, and $\tau^{M}$ is an atom, as required

Now defined the neat embedding by $\theta(r)=r(\bar{x})^{M}$.
Preservation of operations is straightforward.  We show that $\theta$ is injective.
Let $r\in A$ be non-zero. But $M$ is a relativized representation, so there $\bar{a}\in M$
with $r(\bar{a})$ hence $\bar{a}$ is a clique in $M$,
and so $M\models r(\bar{x})(\bar{a})$, and $\bar{a}\in \theta(r)$. proving the required.

We check that it is a complete embedding under the assumption that
$M$ is a complete relativized representation.
Recall that $\A$ is atomic. Let $\phi\in L_{\infty}$ be such that $\phi^M\neq 0$.
Let $\bar{a}\in \phi^M$. Since
$M$ is complete and $\bar{a}\in C^n(M$) there is $\alpha\in \At\A$, such
that $M\models \alpha(\bar{a})$, then $\theta(\alpha).\phi^C\neq 0.$
and we are done.
Now $\A\in S_c\Nr_m\CA_{m+k}$; it embeds completely into $\Nr_n\D$, $\D$ is complete, then
so is $\Nr_n\D$, and consequently $\Cm\At\A\subseteq \Nr_n\D$, which is impossible, because we know that
$\Cm\At\A\notin S\Nr_m\CA_{m+k}.$

\item By theorem \ref{OTT}
\end{enumarab}
\end{proof}

\begin{corollary} Let $2<m<n$. Then the class of algebras having an $n+1$ flat representation is a variety, and
it  is not finitely axiomatizable over the class 
having $n$ flat representations. The class of algebras having complete $n$ smooth representations, when $n\geq m+3$ 
is not even elementary.
\end{corollary}
\begin{proof}
For brevity, we denote the class of algebras having an $n$ flat represerntation by $\RCA_{n, f}$, 
and that of having $n$ smooth representation by $\RCA_{n,s}$. The dimension is $m$ and
$n>m>2$.
For the first part. Assume that $\A$ has an $n$ flat  representation $\M$. As in the above proof, take 
$\D_0$ be the algebra consisting of those $\phi^{M}$ where $\phi$ comes from $L^n$. Since $M$ is $n$ flat we
have that cylindrifiers commute by definition,
hence $\D_0\in \CA_n$.
Now as above the map  $\theta(r)=r(\bar{x})^{M}$ is a neat embeding, hence $\A\subseteq \Nr_m\CA_n$, so that
$A\in S\Nr_m\CA_n$. Conversely, if $\A\in S\Nr_m\CA_n$, then one can build an $n$ flat representation 
by showing that $\A$ has an $n$ dimensional hyperbasis, see 
theorem \ref{smooth}.
Hence $\RCA_{n,f}\subseteq S\Nr_n\CA_m$. But by the celebrated result 
of Hirsch and Hodkinson \cite[theorem 15.1(4)]{HHbook}, we have that for $k\geq 1$, 
$S\Nr_m\CA_{m+k+1}$ is not finitely axiomatizable 
over $S\Nr_m\CA_{m+k}$ and we are done.

For the second part, clearly, every  complete representation 
is $n$ smooth. By theorem \ref{neat}, it suffices to show that $\sf RCA_{m+3,s}\subseteq S_c\Nr_n\CA_{m+3}$.
Now assume that $\A$ has a complete $m$ flat representation.
But the above, using the same notation, the neat  embedding is a complete embedding under the assumption that
$M$ is a complete $n$ smooth relativized representation into $\D$, 
and we are done.

\end{proof}
\begin{theorem} Let $1<n<m$.
Then the following are equivalent:
\begin{enumarab}
\item $\A\in S\Nr_n\CA_m$ 
\item $\A$ has an smooth representation
\item  $\A$ has has $n$ infinitary $n$ flat representation
\end{enumarab}
\end{theorem}

\begin{proof}
\begin{enumarab}

\item  (1) to (2).  Assume that $\A\in S\Nr_n\CA_m$. Then $\A^+\in S_c\Nr_n\CA_m$ 
has a hyperbasis, hence $\A$ has an $n$ smooth relativized
representation, which is infinitary $n$ flat, witness theorem \ref{smooth}.

\item (2) to (3). Here we use the elementary view to relativized representation. We translate the existence of an $n$ 
smooth representation to a first order theory then we use
saturation, that is the base of the relativized representation will be an $\omega$ saturated model of this theory.

Assume that $\A$ has an $n$ smooth representation.
Let ${\sf Clique}(\bar{x})$ be the formula 
$$\bigwedge_{{i_0,\ldots  i_{m-1}<n}}1(x_{i_0},\ldots x_{i_{m-1}})$$
We extend the theory $T(A)$ by
$$\forall \bar{x}({\sf Clique}(\bar{x})\land {\sf c}_ka(\bar{y})\implies \exists x_{i_k}({\sf Clique}({\bar{x}})
\land  a(\bar{z})  \land (\bar{z})_k=x_{i_k}\land \bar{z}\equiv_k \bar{y})$$
Here $\bar{x}$ is of length $n$ and $\bar{y}$ is of length $m$. This is  
for all $i_0, \ldots i_{m-1}, i_k<n$ and all $a\in A$,
$i_k\notin \{i_0,\ldots i_{m-1}\}.$
Now extend $L(\A)$ to the language $L(\A, E)$ by adding $2l$ many ary 
predicates for $0<l\leq n$ and extend the theory $Sq^n(\A)$ to
$$\forall \bar{x}\bar{y}\bar{z}({\sf Clique} (\bar{x})\to [E^l(x,y)\land E^l(y,z)\to E^l(x,z)]$$
and the other axioms are defined the obvious way:
$$\forall x y (E^l(x,y)\to E^j(x\circ \theta, y\circ \theta)),\ j,l<n, \theta:j\to l$$
$$\forall x y(E^m(x, y) \land r(\bar{x})\to r(\bar{y})), r\in \A$$
$$\forall x y(E^{n-2}(x, y) \land {\sf Clique} (xx)\land {\sf Clique} (yy) 
\to \exists z(E^{n-1}(xx, yz)\land {\sf Clique}(yyz))$$

Assume that $Sq^n(\A)$ has an $n$ smooth relativization, then it is consistent. 
Let $M$ be an $\omega$ saturated model, of this theory, then we show that it is a complete $n$ smooth relativized representation. 
We will define an injective homomorphism $h: \A^+\to \wp(^nM)$ 

First note that the set $f_{\bar{x}}=\{a\in A: a(\bar{x})\}$ 
is an ultrafilter in $\A$, whenever $\bar{x}\in M$ and $M\models 1(\bar{x})$ 
Now define
$$h(S)=\{\bar{x}\in 1: f_{\bar{x}}\in S\}.$$
We check only injectivity uses saturation.
The rest is straightforward.
It sufiuces to show that for any ultrafilter $F$ of $\A$ which is an atom in $\A^+$, we have 
$h(\{F\})\neq 0$. 
Let $p(\bar{x})=\{a(\bar{x}): a\in F\}$. Then this type is finitely satisfiable.
Hence by $\omega$ saturation $p$ is realized in $M$ by $\bar{y}$, say . 
Now $M\models 1(\bar{y})$ and $F\subseteq f_{\bar{x}}$, 
since these are both ultrafilters,  equality holds.
Note that any partial isomorphism of $M$ is also a partial isomorphism of $M$ regraded
as a complete representation of $\A^+$. 
Hence $\A^+$ has an infinitary $n$ flat, hence $\A$ has infinitary 
$n$ flat.

\item (3) to (1) From above.
\end{enumarab}
\end{proof}

\begin{remark}

\begin{enumarab}

For any cardinal $\kappa$, $K_{\kappa}$ will denote the complete irreflexive graph with $\kappa$ nodes.
Let $p<\omega$, and $I$ a linearly irreflexive ordered set, viewed as model to a signature containg a binary relation $<$.
$M[p,I]$ is the disjoint union of $I$ and the complete graph $K_p$ with $p$ nodes.
$<$ is interpreted in this structure as follows $<^{I}\cup <{}^{K_p}\cup I\times K_p)\cup (K_p\times I)$
where the order on $K_p$ is the edge relation.

${\sf CRA_{m,n}}$ denotes the class of $\PEA_m$s with $n$ smooth relativized representations.
For $n\geq m+3$, the class ${\sf CRA_{m,n}}$ is not elementary
A rainbow argument can be used  lifting winning strategies from the \ef\ pebble game to rainbow atom structures.
Let $A=M[n-4, \Z]$ and $B=M[n-4, \N]$, then it can be shown  \pe\ has a \ws\ for all finite rounded games (with $n$ nodes)
on $\CA_{A,B}$, namely, in $G^n_r$ for all $r>n$, so she has a \ws\ in $G^n$ 
on any non trivial ultrapower, from which an elementary countable
subalgebra $\B$ can be extracted (using an elementary chain argument)
in which \pe\ also has \ws\ in $G_{\omega}^n$, so that $\B$ has
an $n$ smooth complete representation.

But it can also be shown that \pa\ can  the $\omega$ rounded game on $\CA_{A,B}$, also with $n$ nodes,
hencethe latter does not have an $n$ complete relativized representation,
but is elementary equivalent to one that does.
Let ${\sf EF}_r^p[A, B]$ denote the \ef\ pebble forth game defined in \cite{HHbook} between two structures $A$ and $B$ with
$p$ pebbles and
$r$ rounds, which each player will use as a private game to guide her/him in the rainbow game on coloured graphs.
In his private game, \pa\ always places the pebbles on distinct elements of $\Z$.
She uses rounds $0,\ldots n-3$, to cover $n-4$ and first two elements of $\Z$.
Because at least two out of three distinct colours are related by $<$, \pe\ must respond by pebbling
$n-4 \cup \{e,e'\}$ for some  $e,e'\in \N$.
Assuming that \pa\ has not won, then he has at least arranged that two elements of $\Z$
are pebbled, the corresponding pebbles in $B$ being in $\N$.
Then \pa\ can force \pe\ to play a two pebble game of length $\omega$ on $\Z$, $\N$
which he can win, bombarding her with cones with green tints, in the graph game.
in her private game, \pa\ picks up a spare pebble pair and place the first pebble of it on $a\in A$.
By the rules of the game, $a$ is not currently occupied by a pebble. \pe\ has to choose which element of $B$ to put the
pebble on.  \pe\ chooses an unoccupied element in $n-4$, if possible. If they are all already occupied,
she chooses $b$ to be an arbitrary element
$x\in \N$. Because there are only $n-3$ pebble pairs, \pe\ can always implement this strategy and win.

We can now lift her \ws\ of  the same game but now played on coloured
graphs, the atoms of $\CA_{A,B}$, as before. In this game \pe\ has a \ws\ in all finite rounded games, but \pa\ can win the 
$\omega$ rounded game.

\item This example is a modification of \cite[exercise 1, p. 485]{HHbook2}, by lifting the relation algebra construction therein to the cylindric 
case.
Let $\A(n)$ be an infinite atomic atomic relation algebra;
the atoms of $\A(n)$ are $Id$ and $a^k(i,j)$ for each $i<n-2$, $j<\omega$ $k<\omega_1$.
\item All atoms are self converse.
\item We list the forbidden triples $(a,b,c)$ of atoms of $\A(n)$- those such that
$a.(b;c)=0$. Those triples that are not forbidden are the consistent ones.
This defines composition: for $x,y\in A(n)$ we have
$$x;y=\{a\in \At(\A(n)); \exists b,c\in \At\A: b\leq x, c\leq y, (a,b,c) \text { is consistent }\}$$
Now all permutations of the triple $(Id, s,t)$ will be inconsistent unless $t=s$.
Also, all permutations of the following triples are inconsistent:
$$(a^k(i,j), a^{k'}(i,j), a^{k''}(i,j')),$$
if $j\leq j'<\omega$ and $i<n-1$ and $k,k', k''<\omega_1.$
All other triples are consistent.

Then for any $r\geq 1$, $\A(n-1,r)$ embeds completely in $\A(n)$ the obvious way,
hence, the latter has no $n$ dimensional hyperbasis, because the former does not.
Indeed $\A(n-1,r)=\A^+(n-1, r)\subseteq_c \A(n)\in S_c\Ra\CA_n$ which is a contradiction.

Then $\A(n)$ has an $m$ dimensional hyperbasis for each $m<n-1$, by proving that \pe\ has
\ws\ in the hyperbasis game $G_r^{m,n}(\A(n), \omega)$, for any $r<\omega$, that is, for any finite rounded game.

Now write $T$ for $G_r^{m,n}$. Consider $\M=\M(\A(n), m, n, \omega, \C)$ as a
$5$ sorted structure with sorts $\A(n), \omega, H_m^n(\A, \omega)$ and $\C$.
Then \pe\ has a \ws\ in $G(T\upharpoonright 1+2r, \M)$.
\pe\ has a \ws\ in $G(T|r, \M)$ for all finite $r>0$. So \pe\ has a \ws\ in $G(T, \prod_D \M)$, for any non primcipal ultrafilter on $\omega$.
Hence there is a countable elementary subalgebra $\E$ of $\prod_D\M$ such
that \pe\ has a \ws\ in $G(T,\E)$ Hence $\E$ has the form $\M(\B, m, n, \Lambda, \A)$  for some atomic $\B\in \RA$ countable set
$\Lambda$ and countable atomic $m$ dimensional $\A\in \CA_m$ such
that $\At\A\cong \At\Ca(H_m^n(\B, \Lambda))$.
Furthermore, we have  $\B\prec \prod _D\A(n)$  and $\A\prec \prod_D\A(n)$.
Thus \pe\ has a \ws\ in $G(T, \M, m, n, \Lambda, \A)$
and she also has a \ws\ in $G_{\omega}^{m,n}(\B,\Lambda)$.
So $\B\in S\Ra\CA_n$ and $\A$ embeds into $\Ca(H_m^n(\B, \Lambda)\in \Nr_m\CA_n$.
and we are done.
In fact, one can show that both $\B$ and $\A$ are actually representable, by finding a representation of
$\prod \A(n)/F$ (or an elementary countable subalgebra of it it) embedding every
$m$ hypernetwork in such a way so that
the embedding respects $\equiv_i$ for
every $i<m$, but we do not need that much.

\item  Our next $\CA_m$ is $\A_r^n$, the rainbow cylindric algebra based on $A=M[n-3, 2^{r-1}],$
and $B=M[n-3, 2^{r-1}-1]$, as defined in \cite[ lemma 17.15, 17.16, 17.17]{HHbook}. 
These structures were used by Hirsch and Hodkinson to show that ${\sf RA}_n$
is not finitely axiomatizable over
${\sf RA}_{n+1}$; here ${\sf RA}_m$ is the variety of relation algebras with  $m$ dimensional relational basis.
Now we put them to a different use:

\begin{definition} Let $\K\subseteq \L$ be classes of algebras. We say that the distance between $\K$ and $\L$ is infinite
if there exists a sequence $\A_r\in \L\sim \K$ such that $\prod_{r\in F} \A_r\in \K$, for any non principal ultrafilter $F$
\end{definition}
Note that if $K$ and $L$ are varieties then this means that the former is not finitely axiomatizable over the latter,
like the class of algebras having $n+1$ smooth relativized representations and that of those algebras having 
$n$ smooth relativized representations.
But even if the classes are not even elementary, like $\sf CRA_{m,n}=S_c\Nr_n\CA_{n+k}$, 
then this definition also makes sense as we proceed to show.
\begin{enumarab}
\item \pe\ has a \ws\ in the game $G_{\omega}^{n}(\A_r^n)$
\item \pe\ has a \ws\ in $G_r^{\omega}(\A_r^n)$.
\item \pa\ has a \ws\ in $G_{\omega}^{n+1}(\A_r^n)$
\item The distance between $\sf CRA_{m,n+1}$ and
$\sf CRA_{m,n}$ is infinite.
\end{enumarab}

We have $\A_r^n\in \sf CRA_n\sim CRA_{n+1}$. \pe\ has a \ws\ in $G_r^{n+1}(\A_r^n)$ for all finite $r$, then
\pe\ has a \ws\ in $G_{\omega}^{n+1}(\prod_{r}\A_r^n/D)$, for any non principal ultrafilter $D$,
so the latter is in $\sf CRA_{n+1}$
\end{enumarab}
\end{remark}

\subsection{ More on atom structures}

\begin{definition}
\begin{enumarab} A class $\K$ is gripped by its atom structures, if whenever $\A\in \K\cap \At$, and $\B$ is atomic such that
$\At\B=\At\A$, then $\B\in \K$.

\item A class $\K$ is strongly gripped by its atom structures, if whenever $\A\in \K\cap \At$, and $\B$ is atomic such that
$\At\B\equiv \At\A$, then $\B\in \K$.

\item A class $\K$ of atom structures is infinitary gripped
if whenever $\A\in \K\cap \At$ and $\B$ is atomic, such that $\At\B\equiv_{\infty,\omega}\At\B$,
then $\B\in K$.
\item An atomic game is strongly gripping for $\K$ if whenever \pe\ has a \ws\ for all finite rounded games on $\At\A$, then  $\A\in \K$
\item An atomic game is gripping if \pe\ has a \ws\ in the $\omega$ rounded game on $\At\A$, then $\A\in \K$.
\end{enumarab}
\end{definition}

Notice that infinitary gripped implies strongly gripped implies gripped (by its atom structures).
For the sake of brevity, we write only (strongly) gripped, without referring to atom structures.
In the next theorem, all items except the first applies to all algebras considered.
The first applies to any class $\K$ between $\sf Sc$ and
$\PEA$ (where the notion of neat reducts is not trivial). 
The $n$th  Lyndon condition is a first order sentences that codes a \ws\ for \pe\ in $n$ rounds. 

The elementary class satisfying al such sentences is denoted by ${\sf LCA_m}$. 
It is not hard to show that ${\sf UpUr}\CRA_m={\sf LCA_m}$ 
for any $n>2$. This follows from the simple observation that if \pe\ has a \ws\ in all finite rounded 
atomic games on an atom structure of a $\CA_m$, then this algebra is necessarily elementary equivalent to a countable completely representable
algebra.

\begin{theorem}\label{SL}
\begin{enumarab}
\item The class of neat reducts for any dimension is not gripped,
hence is neither strongly gripped nor infinitary gripped.
\item The class of completely representable algebras is gripped but not strongly gripped.
\item The class of algebras satisfying Lyndon conditions is gripped and strongly gripped.
\item The class of representable algebras is not gripped.
\item The Lyndon usual atomic game is gripping but not strongly gripping
for completely representable algebras, it is strongly gripping for ${\sf LCA_m}$, when $m>2$.
\end{enumarab}
\end{theorem}

\begin{proof}
\begin{enumarab}

\item This example is an adaptation of an example used in \cite{SL} to show that the class of neat reducts is not closed under
forming subalgebras, and also used  other contexts proving negative results on various amalgamation properties for
cylindric-like algebras \cite{STUD}, \cite{Sayedneat}. 

Here we slightly generalize the example by allowing an arbitrary field to rather than $\mathbb{Q}$
to show that there  is an atom structure that carries simultaneously an algebra in $\Nr_{\alpha}\CA_{\alpha+\omega}$ 
and an algebra not in $\Nr_{\alpha}\CA_{\alpha+1}$. This 
works for all $\alpha>1$ (infinite included) and other cylindric-like algebras as will be  clear from
the proof. Indeed, the proof works for any class of algebras whose signature is between $\Sc$ and $\sf QEA$. 
(Here we are using the notation $\QEA$ instead of $\PEA$ because we are
allowing infinite dimensions).

Let $\alpha$ be an ordinal $>1$; could be infinite. Let $\F$ is field of characteristic $0$.
$$V=\{s\in {}^{\alpha}\F: |\{i\in \alpha: s_i\neq 0\}|<\omega\},$$
$${\C}=(\wp(V),
\cup,\cap,\sim, \emptyset , V, {\sf c}_{i},{\sf d}_{i,j}, {\sf s}_{\tau})_{i,j\in \alpha, \tau\in FT_{\alpha}}.$$
Then clearly $\wp(V)\in \Nr_{\alpha}\sf QPEA_{\alpha+\omega}$.
Indeed let $W={}^{\alpha+\omega}\F^{(0)}$. Then
$\psi: \wp(V)\to \Nr_{\alpha}\wp(W)$ defined via
$$X\mapsto \{s\in W: s\upharpoonright \alpha\in X\}$$
is an isomorphism from $\wp(V)$ to $\Nr_{\alpha}\wp(W)$.
We shall construct an algebra $\A$, $\A\notin \Nr_{\alpha}{\sf QPEA}_{\alpha+1}$.
Let $y$ denote the following $\alpha$-ary relation:
$$y=\{s\in V: s_0+1=\sum_{i>0} s_i\}.$$
Let $y_s$ be the singleton containing $s$, i.e. $y_s=\{s\}.$
Define as before
${\A}\in {\sf QPEA}_{\alpha}$
as follows:
$${\A}=\Sg^{\C}\{y,y_s:s\in y\}.$$

Now clearly $\A$ and $\wp(V)$ share the same atom structure, namely, the singletons.
Then we claim that
$\A\notin \Nr_{\alpha}{\sf QPEA}_{\beta}$ for any $\beta>\alpha$.
The first order sentence that codes the idea of the proof says
that $\A$ is neither an elementary nor complete subalgebra of $\wp(V)$.
Let $\At(x)$ be the first order formula asserting that $x$ is an atom.
Let $$\tau(x,y) ={\sf c}_1({\sf c}_0x\cdot {\sf s}_1^0{\sf c}_1y)\cdot {\sf c}_1x\cdot {\sf c}_0y.$$
Let $${\sf Rc}(x):=c_0x\cap c_1x=x,$$
$$\phi:=\forall x(x\neq 0\to \exists y(\At(y)\land y\leq x))\land
\forall x(\At(x) \to {\sf Rc}(x)),$$
$$\alpha(x,y):=\At(x)\land x\leq y,$$
and  $\psi (y_0,y_1)$ be the following first order formula
$$\forall z(\forall x(\alpha(x,y_0)\to x\leq z)\to y_0\leq z)\land
\forall x(\At(x)\to \At(\sf c_0x\cap y_0)\land \At(\sf c_1x\cap y_0))$$
$$\to [\forall x_1\forall x_2(\alpha(x_1,y_0)\land \alpha(x_2,y_0)\to \tau(x_1,x_2)\leq y_1)$$
$$\land \forall z(\forall x_1 \forall x_2(\alpha(x_1,y_0)\land \alpha(x_2,y_0)\to
\tau(x_1,x_2)\leq z)\to y_1\leq z)].$$
Then
$$\Nr_{\alpha}{\sf QPEA}_{\beta}\models \phi\to \forall y_0 \exists y_1 \psi(y_0,y_1).$$
But this formula does not hold in $\A$.
We have $\A\models \phi\text {  and not }
\A\models \forall y_0\exists y_1\psi (y_0,y_1).$
In words: we have a set $X=\{y_s: s\in V\}$ of atoms such that $\sum^{\A}X=y,$ and $\A$
models $\phi$ in the sense that below any non zero element there is a
{\it rectangular} atom, namely a singleton.

Let $Y=\{\tau(y_r,y_s), r,s\in V\}$, then
$Y\subseteq \A$, but it has {\it no supremum} in $\A$, but {\it it does have one} in any full neat reduct $\B$ containing $\A$,
and this is $\tau_{\alpha}^{\B}(y,y)$, where
$\tau_{\alpha}(x,y) = {\sf c}_{\alpha}({\sf s}_{\alpha}^1{\sf c}_{\alpha}x\cdot {\sf s}_{\alpha}^0{\sf c}_{\alpha}y).$

In $\wp(V)$ this last is $w=\{s\in {}^{\alpha}\F^{(\bold 0)}: s_0+2=s_1+2\sum_{i>1}s_i\},$
and $w\notin \A$. The proof of this can be easily distilled from \cite[main theorem]{SL}. 
For $y_0=y$, there is no $y_1\in \A$ satisfying $\psi(y_0,y_1)$.
Actually the above proof proves more. It proves that there is a 
$\C\in \Nr_{\alpha}{\sf QEA}_{\beta}$ for every $\beta>\alpha$ (equivalently $\C\in \Nr_n\QEA_{\omega}$), and $\A\subseteq \C$, such that
$\Rd_{\Sc}\A\notin \Nr_{\alpha}\Sc_{\alpha+1}$.
See \cite[theorems 5.1.4-5.1.5]{Sayedneat} for an entirely different example.

\item The algebra $\PEA_{\Z,\N}$, and its various reducts down to $\Sc$s, 
shows that the class of completely representable algebras is not strongly gripped.
Indeed, it can be shown that \pe\ can win all finite rounded atomic games but \pa\ can win the 
$\omega$ rounded game. It is known that this class is gripped. An atom structure is completely representable iff one, 
equivalently, all atomic algebras sharing
this atom structure are completely representable \cite{HHbook2}, \cite{HH}.
\item This is straightforward from the definition of Lyndon conditions.

\item Any weakly representable atom structure that is not strongly representable detects this, see e.g.
\cite{Hodkinson}, \cite{weak}, \cite[theorems 1.1, 1.2]{ANT} and theorems \ref{can}, and theorem \ref{hodkinson}.
For a potential stronger result, see theorem \ref{blurs} above.
\item Follows directly from the definition.
\end{enumarab}
\end{proof}

Recall the definition of relativized representation, definition \ref{rel}.
Here $m<\omega$ denotes the dimension and ${\sf CRA}_{m,n}$ denotes the class of $\CA_m$s with an $n$ relativized
smooth representation. 
There is no restriction whatsoever on $n$ except that it is $>m$. In particular $n$ can be infinite.

\begin{theorem}\label{longer} Regardless of cardinalities, $\A\in {\sf CRA}_{m, \omega}$
iff \pe\ has a \ws\ in $G_{\omega}^{\omega}$.
\end{theorem}
\begin{proof} One side is obvious. Now assume that \pe\ has a \ws\ in the $\omega$ rounded game, using $\omega$ many pebbles.
We need to build an $\omega$ relativized complete representation.
The proof goes as follows. First the atomic networks are finite, so we need to convert them into $\omega$ dimensional
atomic networks. For a network $N$, and  a map $v:\omega\to N$,
let $Nv$ be the network  induced by $v$, that is $Nv(\bar{s})=N(v\circ \bar{s})$.
let $J$ be the set of all such $Nv$, where $N$ occurs in some play
of $G_{\omega}^{\omega}(\A)$ in which \pe\ uses his \ws\ and $v:\omega\to N$ (so via these maps we are climbing up $\omega$).

This  can be checked to be  an $\omega$ dimensional hyperbasis (extended to the cylindric case the obvious way).
So $\A\in S_c\Nr_n\CA_{\omega}$. We can use that the basis consists of $\omega$ dimensional atomic networks, such that for
each such network, there is a finite bound on the size of its strict networks.
Then a complete $\omega$ relativized representation  can be obtained in a step by step way, requiring inductively
in step $t$, there for any finite clique $C$ of $M_t$, $|C|<\omega$, there is
a network in the base, and an embedding $v:N\to M_t$ such that $\rng v\subseteq C$.
Here we consider finite sequences of arbitrarily large length, rather than fixed length $n$ tuples.
This  is because an $\omega$ relativized representation only requires cylindrifier witnesses over
finite sized  cliques, not necessarily cliques that are uniformly bounded.
\end{proof}

\begin{theorem}\label{rel}${\sf CRA_m}\subset {\sf CRA}_{m,\omega}$, the strict inclusion can be only witnessed on
uncountable algebras. Furthermore, the class ${\sf CRA}_{m,\omega}$ is not elementary.
The classes $S_c\Nr_m\CA_{\omega}$, ${\sf CRA}_m$, and ${\sf CRA}_{m, \omega}$,
coincide on atomic countable
algebras.
\end{theorem}

\begin{proof}That ${\sf CRAS}_{m,\omega}$ is not elementary is witnessed by the rainbow algebra
$\PEA_{K_{\omega}, K}$, where the latter
is a disjoint union of $\K_n$, $n\in \omega$. \pe\ has a \ws for all finite length games, but \pa\ can win the infinite rounded game.
Hence \pe\ can win the transfinite game on an uncountable non-trivial ultrapower of $\A$,
and using elementary chains one can find an elementary countable
subalgebra $\B$ of this ultrapower such that \pe\ has a \ws\ in the $\omega$ rounded game.
This $\B$ will have an $\omega$ square representation
hence will be in $\sf CRA_{m,\omega}$, and $\A$ is not in the latter class.
The three classes coincide on countable atomic algebras with the class of completely representable algebras \cite[theorem 5.3.6]{Sayedneat}
\end{proof}
\begin{theorem}  If $\A\in \CA_3$ is finite, and has an $n$ square relativized representation, with $3<n\leq \omega$,
then it has a {\it finite} $n$ square relativized
representation. If $n\geq 6$,
this is not true for $n$ smooth relativized representations.
\end{theorem}
\begin{demo}{Proof} The first part follows from the fact that tthe first order theory coding the existence of squiare representations 
can be coded in the clique guarded fragment of first order logic,
and indeed in the loosely guarded fragment of first order logic which
has the finite base property \cite[corollary 19.7]{HHbook}.
The second part follows from the fact that  the problem of deciding whether a finite $\CA_3$ is in
$S\Nr_3\CA_n$, when $n\geq 6$, is undecidable,
from which one can conclude that there are finite algebras in $S\Nr_3\CA_n$, $n\geq 6$
that do not have a finite $n$ dimensional hyperbasis \cite[corollary 18.4]{HHbook} and
these cannot possibly have finite representations.
\end{demo}

A contrasting result is:

\begin{theorem}
Let $\A$ be a finite $\CA_m$. Then the following are equivalent
\begin{enumarab}
\item $\A$ has a finite $n$ smooth relativized representation
\item $\A\in S\Nr_m\CA_n$
\item $\A$ has a finite $n$ dimensional hyperbasis.
\end{enumarab}
\end{theorem}

\begin{proof} (1) to (2) to (3) is exactly as above, by noting that if $\A$ is finite then $\D$ as defined in corollary \ref{flat}, 
is also finite, 
and that this gives necessarily a finite
hyperbasis, witness the proof of theorem \ref{smooth}. 

It remains to show that  (1) implies (3). Let $H$ be a finite $n$ dimensional hypebasis; 
we can assume 
that is symmetric, that is, closed under substitutions. This does not
affect finiteness. Let $L$ be the finite signature consisting of an $m$-ary 
relation symbol for every element of $\A$, together with an $n$ ary relation symbol $R_N$ 
for each $N\in H$. Define an $L$ structure by
$$M\models a(\bar{x})\text { iff } 1(\bar{x})\text { and }M(\bar{x})\leq a,$$
and 
$$M\models R_N(x_0,\ldots, x_{n-1})\text { iff } M(x_0,\ldots, x_n-1)=N$$ 
for all $a\in A$ and $N\in H$ Then it is not hard to show that $M$ 
satisfies the axioms postulated in theorem \ref{step} and these 
can be coded as a fragment of the loosely gaurded fragment.
\end{proof}

\section{Neat atom structures}

Next  we introduce several definitions an atom structures concerning neat embeddings:
Here we denote the dimension by $n$, where $n$ is finite. $n$ will be always $>1$ and often greater than $2$.
\begin{definition}
\begin{enumarab} 
\item Let $1\leq k\leq \omega$. Call an atom structure $\alpha$ {\it weakly $k$ neat representable}, 
if the term algebra is in $\RCA_n\cap \Nr_n\CA_{n+k}$, but the complex algebra is not representable.
\item Call an atom structure {\it $k$ neat}, $k>n$,  if there is an atomic algebra $\A$, such that $\At\A=\alpha$ and $\A\in \Nr_n\CA_{k}.$ 
\item Let $k\leq \omega$. Call an atom structure $\alpha$ {\it $k$ complete}, 
if there exists $\A$ such that $\At\A=\alpha$ and $\A\in S_c\Nr_n\CA_{n+k}$.
\end{enumarab}
\end{definition}
\begin{definition} Let $\K\subseteq \CA_n$, and $\L$ be an extension of first order logic. 
$\K$ is {\it detectable} in $\L$, if for any $\A\in \K$, $\A$ atomic, and for any atom structure 
$\beta$ such that $\At\A\equiv_{\L}\beta$,
if $\B$ is an atomic algebra such that $\At\B=\beta$, then $\B\in \K.$
\end{definition}
Roughly speaking, $K$ is detectable in $\L$ if whenever an atomic algebra is not in $\K$ then $\L$ can {\it witness this}. 
In particular, a class that is not detectable
in first order logic is simply elementary. A class that is not witnessed by quasi (equations) is a (quasi) variety.

We investigate the existence of such structures, and the interconnections. 
Note that if $\L_1$ is weaker than $\L_2$ and $\K$ is not detectable
in $\L_2$, then it is not detectable in $\L_1$.
We also present several $\K$s and $\L$s as in the second definition. 
All our results extend to Pinter's algebras and quasi polyadic algebras with and without equality.
But first another definition.

We now prove:
\begin{theorem}\label{main} 
\begin{enumarab}
\item Let $n$ be finite $n\geq 3$. Then there exists a countable weakly $k$ neat atom structure of dimension $n$ if and only if $k<\omega.$
\item There is an $\omega$ rounded game that determines neat atom structures.
\item The class of completely representable algebras, and strongly representable 
ones of dimension $>2$, is not detectable in $L_{\omega,\omega}$, while the class
$\Nr_n\CA_m$ for any ordinals $1<n<m<\omega$, is not detectable 
even in $L_{\infty,\omega}.$ For for infinite $n$, $\Nr_n\CA_m$
is not detectable in first order logic nor in the quantifier free reduct of $L_{\infty,\omega}$.
\item There is an atom structure that is not $n+3$ complete but is elementary equivalent to one that is $\omega$ neat
\end{enumarab}
\end{theorem}
\begin{proof} 
\begin{enumarab}
\item Follows from \cite{ANT}, see also theorem\ref{blurs} above. Here $k$ cannot be infinite,.
for else the term algebra will be in $\Nr_n\CA_{\omega}$, hence by \cite[5.3.6]{Sayedneat} 
would be completely representable, which makes its atom 
structure $\At$ strongly representable \cite[3.5.1]{HHbook2},  
but then $\Cm\At$ will be representable.

\item For the definition of a network and atomic games on networks, we refer to \cite[definitions 3.3.1, 3.3.2, 3.3.3]{HHbook2}.
For an atomic network and for  $x,y\in \nodes(N)$, we set  $x\sim y$ if 
there exists $\bar{z}$ such that $N(x,y,\bar{z})\leq {\sf d}_{01}$.
The equivalence relation $\sim$ over the set of all finite sequences over $\nodes(N)$ is defined by
$\bar x\sim\bar y$ iff $|\bar x|=|\bar y|$ and $x_i\sim y_i$ for all
$i<|\bar x|$.(It can be checked that this indeed an equivalence relation.) 

A \emph{ hypernetwork} $N=(N^a, N^h)$ over an atomic polyadic equality algebra $\C$
consists of a network $N^a$
together with a labelling function for hyperlabels $N^h:\;\;^{<
\omega}\!\nodes(N)\to\Lambda$ (some arbitrary set of hyperlabels $\Lambda$)
such that for $\bar x, \bar y\in\; ^{< \omega}\!\nodes(N)$
\begin{enumerate}
\renewcommand{\theenumi}{\Roman{enumi}}
\setcounter{enumi}3
\item\label{net:hyper} $\bar x\sim\bar y \Rightarrow N^h(\bar x)=N^h(\bar y)$.
\end{enumerate}
If $|\bar x|=k\in \N$ and $N^h(\bar x)=\lambda$ then we say that $\lambda$ is
a $k$-ary hyperlabel. $(\bar x)$ is referred to a a $k$-ary hyperedge, or simply a hyperedge.
(Note that we have atomic hyperedges and hyperedges)
When there is no risk of ambiguity we may drop the superscripts $a,
h$.
There are {\it short} hyperedges and {\it long} hyperedges (to be defined in a while). The short hyperedges are constantly labelled.
The idea (that will be revealed during the proof), is that the atoms in the neat reduct are no smaller than the atoms
in the dilation. (When $\A=\Nr_n\B,$ it is common to call $\B$ a dilation of $\A$.)
We know that there is a one to one correspondence between networks and coloured graphs.
If $\Gamma$ is a coloured graph, then by $N_{\Gamma}$
we mean the corresponding network defined on $n-1$ tuples of the nodes of $\Gamma$ to
to coloured graphs of size $\leq n$.
\begin{enumarab}
\item A hyperedge $\bar{x}\in {}^{<\omega}\nodes (\Gamma)$ of length $m$ is {\it short}, if there are $y_0,\ldots y_{n-1}\in \nodes(N)$, such that
$N_{\Gamma}(x_i, y_0, \bar{z})\leq {\sf d}_{01}$, or $N(_{\Gamma}(x_i, y_1, \bar{z})\ldots$ or $N(x_i, y_{n-1},\bar{z})\leq {\sf d}_{01}$ for all $i<|x|$,
for some (equivalently for all)
$\bar{z}.$ Otherwise, it is called {\it long.}
\item A hypergraph $(\Gamma, l)$
is called {\it $\lambda$ neat} if $N_{\Gamma}(\bar{x})=\lambda$ for all short hyper edges.
\end{enumarab}
This game is similar to the games devised by Robin Hirsch in \cite[definition 28]{r}, played on relation algebras.
However, lifting it to cylindric algebras is not straightforward, for in this new context the moves involve hyperedges of length $n$ (the dimension), 
rather than edges. In the $\omega$ rounded game $J$,  \pa\ has three moves.

The first is the normal cylindrifier move. There is no polyadic move.
The next two are amalgamation moves.
But the games are not played on hypernetworks, they are played on coloured hypergraphs, consisting of two parts,
the graph part
that can be viewed as an $L_{\omega_1, \omega}$ model for the rainbow signature, and the part dealing with hyperedges with a
labelling function.
The amalgamation moves roughly reflect the fact, in case \pe\ wins, then for every $k\geq n$ there is a $k$ dimensional hyperbasis,
so that the small algebra embeds into cylindric algebras of arbitrary large dimensions.
The game is played on $\lambda$ neat hypernetworks,  translated to $\lambda$ neat hypergraphs,
where $\lambda$ is a label for the short hyperedges.

For networks $M, N$ and any set $S$, we write $M\equiv^SN$
if $N\restr S=M\restr S$, and we write $M\equiv_SN$
if the symmetric difference
$$\Delta(\nodes(M), \nodes(N))\subseteq S$$ and
$M\equiv^{(\nodes(M)\cup\nodes(N))\setminus S}N.$ We write $M\equiv_kN$ for
$M\equiv_{\set k}N$.

Let $N$ be a network and let $\theta$ be any function.  The network
$N\theta$ is a complete labelled graph with nodes
$\theta^{-1}(\nodes(N))=\set{x\in\dom(\theta):\theta(x)\in\nodes(N)}$,
and labelling defined by
$$(N\theta)(i_0,\ldots i_{\mu-1}) = N(\theta(i_0), \theta(i_1), \theta(i_{\mu-1})),$$
for $i_0, \ldots i_{\mu-1}\in\theta^{-1}(\nodes(N))$.
We call this game $H$. It is $\omega$ rounded. Its first move by \pa\ is the usual cylindrifier move (equivalently)
\pa\ s move in $F^{\omega}$), but \pa\ has more moves which makes it harder for
\pe\ to win. These notions apply equaly well to hypernetworks.

\pa\ can play a \emph{transformation move} by picking a
previously played hypernetwork $N$ and a partial, finite surjection
$\theta:\omega\to\nodes(N)$, this move is denoted $(N, \theta)$.  \pe\
must respond with $N\theta$.

Finally, \pa\ can play an
\emph{amalgamation move} by picking previously played hypernetworks
$M, N$ such that $M\equiv^{\nodes(M)\cap\nodes(N)}N$ and
$\nodes(M)\cap\nodes(N)\neq \emptyset$.
This move is denoted $(M,
N)$. Here, unlike $H$, there is no restriction on the number of overlapping nodes, so in principal
the game is harder for \pe\ to win.

To make a legal response, \pe\ must play a $\lambda_0$-neat
hypernetwork $L$ extending $M$ and $N$, where
$\nodes(L)=\nodes(M)\cup\nodes(N)$.

The next theorem is a cylindric-like variation on \cite[theorem 39]{r}, formulated for polyadic equality algebras,
but it works for many cylindric like algebras like $\sf Sc$, $\sf PA$ and $\sf CA$. 
But it does not apply to $\Df$s because the notion of neat reducts for $\Df$s is trivial.

Let $\A$ be an atomic polyadic equality algebra with a countable atom structure $\alpha$.
If \pe\ can win the $\omega$ rounded game $J$ on $\alpha$,
then there exists a locally finite $\PEA_{\omega}$ such that
$\At\A\cong \At\Nr_n\C$. Furthermore, $\C$ can be chosen to be complete, and $\Cm\At\A=\Nr_n\C$.

For the first part. Fix some $a\in\alpha$. Using \pe\ s \ws\ in the game of neat hypernetworks, one defines a
nested sequence $N_0\subseteq N_1\ldots$ of neat hypernetworks
where $N_0$ is \pe's response to the initial \pa-move $a$, such that
\begin{enumerate}
\item If $N_r$ is in the sequence and
and $b\leq {\sf c}_lN_r(f_0, \ldots, x, \ldots  f_{n-2})$,
then there is $s\geq r$ and $d\in\nodes(N_s)$ such
that $N_s(f_0, f_{l-1}, d, f_{l+1},\ldots f_{n-2})=b$.
\item If $N_r$ is in the sequence and $\theta$ is any partial
isomorphism of $N_r$ then there is $s\geq r$ and a
partial isomorphism $\theta^+$ of $N_s$ extending $\theta$ such that
$\rng(\theta^+)\supseteq\nodes(N_r)$.
\end{enumerate}
Now let $N_a$ be the limit of this sequence, that is $N_a=\bigcup N_i$, the labelling of $n-1$ tuples of nodes
by atoms, and the hyperedges by hyperlabels done in the obvious way.
This limit is well-defined since the hypernetworks are nested.
We shall show that $N_a$ is the base of a weak set algebra having unit  $V={}^{\omega}N_a^{(p)}$,
for some fixed sequence $p\in {}^{\omega}N_a$.

We can make $U_a$ into the universe an $L$ relativized structure ${\cal N}_a$;
here relativized means that we are only taking those assignments agreeing cofinitely with $f_a$,
we are not taking the standard square model.
However, satisfiability  for $L$ formulas at assignments $f\in U_a$ is defined the usual Tarskian way, except
that we use the modal notation, with restricted assignments on the left:
For $r\in \A
l_0, \ldots l_{n-1}, i_0 \ldots, i_{k-1}<\omega$, \/ $k$-ary hyperlabels $\lambda$,
and all $L$-formulas $\phi, \psi$, let
We can make $U_a$ into the base of an $L$-structure ${\cal N}_a$ and
evaluate $L$-formulas at $f\in U_a$ as follow.  For $b\in\alpha,\;
l_0, \ldots l_{\mu-1}, i_0 \ldots, i_{k-1}<\omega$, \/ $k$-ary hyperlabels $\lambda$,
and all $L$-formulas $\phi, \psi$, let
\begin{eqnarray*}
{\cal N}_a, f\models b(x_{l_0}\ldots  x_{n-1})&\iff&N_a(f(l_0),\ldots  f(l_{n-1}))=b\\
{\cal N}_a, f\models\lambda(x_{i_0}, \ldots,x_{i_{k-1}})&\iff&  N_a(f(i_0), \ldots,f(i_{k-1}))=\lambda\\
{\cal N}_a, f\models\neg\phi&\iff&{\cal N}_a, f\not\models\phi\\
{\cal N}_a, f\models (\phi\vee\psi)&\iff&{\cal N}_a,  f\models\phi\mbox{ or }{\cal N}_a, f\models\psi\\
{\cal N}_a, f\models\exists x_i\phi&\iff& {\cal N}_a, f[i/m]\models\phi, \mbox{ some }m\in\nodes(N_a)
\end{eqnarray*}
For any $L$-formula $\phi$, write $\phi^{{\cal N}_a}$ for the set of all $n$ ary assignments satisfying it; that is
$\set{f\in\;^\omega\!\nodes(N_a): {\cal N}_a, f\models\phi}$.  Let
$D_a = \set{\phi^{{\cal N}_a}:\phi\mbox{ is an $L$-formula}}.$
Then this is the universe of the following weak set algebra
\[\D_a=(D_a,  \cup, \sim, {\sf D}_{ij}, {\sf C}_i)_{ i, j<\omega}\]
then  $\D_a\in\RCA_\omega$. (Weak set algebras are representable).

For any $L$-formula $\phi$, write $\phi^{{\cal N}_a}$ for
$\set{f\in\;^\omega\!\nodes(N_a): {\cal N}_a, f\models\phi}$.  Let
$Form^{{\cal N}_a} = \set{\phi^{{\cal N}_a}:\phi\mbox{ is an $L$-formula}}$
and define a cylindric algebra
\[\D_a=(Form^{{\cal N}_a},  \cup, \sim, {\sf D}_{ij}, {\sf C}_i, i, j<\omega)\]
where ${\sf D}_{ij}=(x_i= x_j)^{{\cal N}_a},\; {\sf C}_i(\phi^{{\cal N}_a})=(\exists
x_i\phi)^{{\cal N}_a}$.  Observe that $\top^{{\cal N}_a}=U_a,\; (\phi\vee\psi)^{{\cal N}_a}=\phi^{\c
N_a}\cup\psi^{{\cal N}_a}$, etc. Note also that $\D$ is a subalgebra of the
$\omega$-dimensional cylindric set algebra on the base $\nodes(N_a)$,
hence $\D_a\in {\sf Lf}_{\omega}\cap {\sf Ws}_\omega$, for each atom $a\in \alpha$, and is clearly complete.

Let $\C=\prod_{a\in \alpha} \D_a$. (This is not necessarily locally finite).
Then  $\C\in\RCA_\omega$, and $\C$ is also complete, will be shown to be is the desired generalized weak set algebra,
that is the desired dilation.
Note that unit of $\C$ is the disjoint union of the weak spaces.
Then $\Nr_{n}\C$ is atomic and $\alpha\cong\At\Nr_{n}\C$ --- the isomorphism
is $b \mapsto (b(x_0, x_1,\dots x_{n-1})^{\D_a}:a\in A)$

Now we can work in $L_{\infty,\omega}$ so that $\C$ is complete
by changing the defining clause for infinitary disjunctions to
$$N_a, f\models (\bigvee_{i\in I} \phi_i) \text { iff } (\exists i\in I)(N_a,  f\models\phi_i)$$
By working in $L_{\infty, \omega},$ we assume that arbitrary joins hence meets exist,
so $\C_a$ is complete, hence so is $C$. But $\Cm\At\A\subseteq \Nr_n\C$ is dense and complete, so
$\Cm\At\A=\Nr_n\C$.

\item Follows from \cite{hirsh}, generalized in theorem \ref{el} above, \cite[theorem 3.6.11, corollary 3.7.1]{HHbook2}, 
and \cite[theorems 5.1.4, 5.1.5]{Sayedneat}.

\item see \ref{neat}

\item \pe\ has a \ws\ in the game devised below on the rainbow algebra $\CA_{\Z,\N}$ which  
is not completely representable but is elementary equivalent to an algebra in $\Nr_n\CA_{\omega}.$
Since atom structures are interpretable in their algebras, we are done.
\end{enumarab}
\end{proof}

\begin{remark}

\begin{enumarab}
\item Note that if we use the following game, then we get a much stronger result (witness the next item), 
namely, that $\A=\Nr_n\C$.
Here \pa\ has even more moves and this allows is to remove $\At$
from both sides of the above equation, obtaining a much stronger result
The main play of the stronger game $K(\A)$ is a play of the game $J(\A).$

The base of the main board at a certain point will be the atomic network $X$ and we write
$X(\bar{x})$ for the atom that labels the edge $\bar{x}$ on the main board.
But \pa\ can make other moves too, which makes it harder for \pe\ to win and so a \ws\ for \pe\ will give a stronger result.
An $n$  network is a finite complete graph with nodes including $n$
with all edges labelled by elements of $\A$. No consistency properties are assumed.

\pa\ can play an arbitrary $n$ network $N$, \pe\ must replace $N(n)$  by
some element $a\in A$. The idea, is that the constraints represented by $N$ correspond to an element of the $\RCA_\omega$ being constructed on $X$,
generated by $A$. This network is placed on the side of the main board. $N$ asserts that whenever it appears in $X$
you can never have an atom not below  holding between the embedded images of $n$.
But it also asserts that whenever an atom below $a$ holds in $X$,  there are also points in $X$ witnessing all the nodes of
$N$. The final move is that \pa\ can pick a previously played $n$ network $N$ and pick any  tuple $\bar{x}$
on the main board whose atomic label is below $N(\bar{n})$.

\pe\ must respond by extending the main board from $X$ to $X'$ such that there is an embedding $\theta$ of $N$ into $X'$
 such that $\theta(0)=x_0\ldots , \theta(n-1)=y_{n-1}$ and for all $\bar{i} \in N$ we have
$X(\theta(i_0)\ldots, \theta(i_{n-1}))\leq N(\bar{i})$. This ensures that in the limit, the constraints in
$N$ really define $a$.
If \pe\ has a \ws\ in $K(A)$ then the extra moves mean that every $n$ dimensional element generated by
$\A$ in the $\RCA_\omega$ constructed in the play is an element of $\A$.

\item The example in \cite{SL} shows that there is a neat atom structure that  carries algebra that is not in $\Nr_n\CA_{n+1}$
\end{enumarab}
\end{remark}
\begin{theorem} If there exists an atomic  polyadic equality  algebra with countably many atoms 
such that \pe\ has a \ws\ for $J_n$ for every $n$,
and \pa\ has a \ws\ in $F^{n+k}$ on its $\Sc$ reduct, then 
$\Rd_{sc}\A\notin S_c\Nr_n\Sc_{n+k}$ but $\A\in {\sf UpUr}\Nr_n{\sf QPEA}_{\omega}$.
In fact, there is a countable atomic algebra $\D\in \Nr_n\sf QEA_{\omega}$, such that $\A\equiv \D$.
In particular, for any class $\K$ of Pinter's algebras, cylindric algebras and polyadic equality algebras
any $m\geq n+k$, and any class $\sf L$ such that 
$\Nr_n\K_{\omega}\subseteq {\sf L}\subseteq S_c\Nr_n\K_{m+k}$, ${\sf L}$ is not elementary. 
In particular,  ${\sf CRK}_n$, and $\Nr_{n}\CA_{n+k}$ are  not elementary.
\end{theorem}
\begin{proof}
Assume  \pa\ can win the game $F^{n+k}$ on $\Rd_{sc}{\A}.$
Hence $\Rd_{sc}\A\notin S_c\Nr_n\sf Sc_{n+3}$.
For $n<\omega,$ assume that \pe\ has a \ws\ $\sigma_n$ in $J_n(\A)$.
We can assume that $\sigma_n$ is deterministic.
Let $\B$ be a non-principal ultrapower of $\A$.  Then
\pe\ has a \ws\ $\sigma$ in $J(\B)$   --- essentially she uses
$\sigma_n$ in the $n$'th component of the ultraproduct so that at each
round of $J(\B)$, \pe\ is still winning in co-finitely many
components, this suffices to show she has still not lost.
Now use an elementary chain argument to construct countable elementary
subalgebras $\A=\A_0\preceq\A_1\preceq\ldots\preceq \B$.  For this,
let $\A_{i+1}$ be a countable elementary subalgebra of $\B$
containing $\A_i$ and all elements of $\B$ that $\sigma$ selects
in a play of $J(\B)$ in which \pa\ only chooses elements from
$\A_i$. Now let $\A'=\bigcup_{i<\omega}\A_i$.  This is a
countable elementary subalgebra of $\B$ and \pe\ has a \ws\ in
$H(\A')$. Hence by the elementary chain argument
there is a countable $\A'$ such that \pe\ can win the $\omega$ rounded game on its atom structure, hence $\A'\equiv \A$ but
the former is in $\Nr_n{\sf PEA}_{\omega}.$
\end{proof}

\begin{theorem} When $k\leq 3$ the above statement is strictly stronger than the statement in theorem \ref{neat}.
\end{theorem}
\begin{proof}
This follows from the fact that the inclusion $\Nr_n\K_{\omega}\subset S_c\Nr_n\K_{\omega}$ 
is strict. In fact, for any $n>1$, the inclusion $\Nr_n\CA_{n+k}\subset S_c\Nr_n\CA_{n+k}$ is strict for 
every $k>n$, witness  \cite{Sayedneat} and also first item of \ref{SL}. Also in the first statement we have 
$\A\cong \D$ and $\At\D\cong \At\Nr_n\C$, for some $\C\in {\sf Lf}_{\omega}$, while in the present case we can 
remove $\At$ from the two sides of the equation, namely, we have $\D\cong \Nr_n\C$. 
In view of example \ref{SL} this is obviously  much stronger.
\end{proof}


\begin{thebibliography}{}
\bibitem{1} H. Andr\'eka, M. Ferenczi, I. N\'emeti(Editors), {\bf Cylindric-like Algebras and Algebraic Logic},
 Bolyai Society Mathematical Studies, Vol. 22 (2013).

\bibitem{Assem} M. Assem {\it Masters thesis}, 2013.

\bibitem{Studia} T. Sayed Ahmed, {\it Martin's axiom, Omitting types and Complete representations in algebraic logic} 
Studia Logica {\bf 72} (2002), pp. 285-309.

\bibitem{basim}T. Sayed Ahmed and Basim Samir
{\it A neat embedding theorem for expansions of cylindric algebras} Logic Journal of IGPL {\bf 15}(2007)
p. 41-51

\bibitem{Biro} B. Bir\'o, {\it Non finite axiomatizability results in cylindric algebras}, Journal of Symbolic Logic, {\bf 57}(1992) pp. 832-843



\bibitem{AGMNS} H. Andr\'eka, S. Givant, S. Mikulas, I. N\'emeti,A. Simon,
{\it Notions of density that imply representability in algebraic logic,}Annals of Pure and Applied logic, {\bf 91}(1998), p. 93 -190.

\bibitem{ANT} H. Andr\'eka, I. N\'emeti, T. Sayed Ahmed, {\it Omitting types for finite variable fragments and complete representations},
Journal of Symbolic Logic {\bf 73} (2008) p. 65-89

\bibitem{HHbook} R. Hirsch and I. Hodkinson, {\it Relation algebras by games.}
Studies in Logic and the Foundations of Mathematics, volume {\bf 147} (2002)

\bibitem {HHbook2} R. Hirsch and I. Hodkinson, {\it Completions and complete representations in algebraic logic .} In \cite{1}

\bibitem{Hodkinson} I. Hodkinson, {\it Atom structures of relation and cylindric algebras}, Annals of pure and applied logic,
{\bf 89}(1997)p. 117-148.



\bibitem {AU} I. Hodkinson, \emph{A construction of cylindric and polyadic algebras from atomic relation algebras},
Algebra Universalis, Vol. 68 (2012), pp. 257-285.

\bibitem{weak} T. Sayed Ahmed, {\it Weakly representable atom structures that are not strongly representable,
with an application to first order logic}, Mathematical Logic Quarterly, {\bf 54}(3)(2008) p. 294-306 (2008).

\bibitem{Vaught} T. Sayed Ahmed, {\it On a theorem of Vaught for first order logic with finitely many variables.}
Journal of Applied Non-Classical Logics {\bf 19}(1) (2009) p. 97-112.

\bibitem{Sayed} T. Sayed Ahmed, {\it Completions, Complete representations and Omitting types}, in \cite{1}.

\bibitem{can} T. Sayed Ahmed {\it The class $S\Nr_n\CA_{n+k}$ for finite $n>2$ and any $k\geq 4$ is not atom canonical.}
Submitted to the Journal of Symbolic logic

\bibitem{SL} T.Sayed Ahmed and Istvan Nemeti {On neat reducts of algebras of logic} Strudia Logica {\bf 68} (2001) p. 229-262
\bibitem{ST} I. Sain and R. Thompson {\it Strictly finite  schema axiomatization of  quasi polyadic algebras}
In {\bf Algebraic Logic} North Holland (1990), Editors H. Andr\'eka, J. D. Monk,
and I. N\'emeti.



\bibitem{HH} R. Hirsch and I. Hodkinson, \emph{Complete representations in algebraic logic}, Journal of Symbolic Logic, {\bf 62}(3)(1997)p. 816-847.


\bibitem{h} Hokinson {\it Constructing cylindric and polyadic algebras from atomic relation algebras}
Algebra Universalis 68 (2012) 257-285.

\bibitem{HHbook} Hirsch R., Hodkinson.I., {\it Relation algebras by games.}
Studies in Logic and the Foundations of Mathematics. Volume 147.

\bibitem{STUD} T Sayed Ahmed {\it On amalgmation in algebrs of logic} Studia Logica  {\bf 81}(2005) 61-71

\bibitem{t} Sayed Ahmed {\it The class $SNr_3\CA_k$ for $k\geq 6$ is not
closed under completions} Logic Journal of IGPL, 18 (2008) 427-429

\bibitem{r} Hirsch R. {\it Relation algebra reducts of cylindric algebras and complete representations}.
Journal of Symbolic Logic, {\bf 72}(2) (2007), p.673-703.



\bibitem{Sayedneat} T. Sayed Ahmed,{\it Neat reducts and neat embeddings in cylindric algebras}, in \cite{1}.

\bibitem{tarski}
Leon Henkin, J.Donald Monk, and Alfred Tarski, Cylindric algebras,
part I, II, North-Holland, publishing company, Amsterdam London.

\bibitem{thompson}
I. Sain and R. Thompson. Strictly finite schema axiomatization of Quasi-Polyadic algebras. In "Algebraic Logic".
Editors: H. Andr\'eka, J. Monk and I. N\'emeti. North Holland 1989.

\bibitem{hirsh}
R. Hirsch and I. Hodkinson. Strongly representable atom structures
of cylindric algebras. J. Symbolic Logic, 74:811-828, 2009.

\bibitem{Erdos}
P. Erd\"{o}s, Graph theory and probability, Canadian Journal of
Mathematics, vol. 11 (1959), pp. 34-38.

\bibitem{graph}
R. Diestel. Graph theory, volume 173 of Graduate Texts in
Mathematics. Springer-Verlag, Berlin, 1997.
\end{thebibliography}
\end{document}